\documentclass[leqno,12pt]{amsart} 
\usepackage{amsmath,amsthm}
\usepackage{amssymb,amscd,latexsym}
\usepackage{boxedminipage}
\usepackage{eepic}
\usepackage{epic}
\usepackage{wrapfig}
\usepackage{graphicx}

\theoremstyle{plain}

\newtheorem*{def-theo}{Definition-Theorem}

\theoremstyle{definition}

\theoremstyle{remark}
\newtheorem{remark}{Remark}

\newtheorem{thm}{Theorem}[section]

\newtheorem{lem}[thm]{Lemma}

\theoremstyle{definition}

\newtheorem{definition}[thm]{Definition}

\newtheorem{rem}[thm]{Remark}
\newtheorem{theorem}[thm]{Theorem}

\numberwithin{equation}{section}
\newcommand{\bp}{\begin{pmatrix}}
\newcommand{\ep}{\end{pmatrix}}
\newcommand{\bps}{\begin{smallmatrix}}
\newcommand{\eps}{\end{smallmatrix}}
\def\C{{\mathbb C}}

\def\R{{\mathbb R}}
\def\Z{{\mathbb Z}}

\def\al{{\alpha}}

\def\z{{\mathfrak z}}

\def \wt#1{\widetilde{#1}}

\def \0{{\bf 0}}
\def \1{{\bf 1}}

\def\Res{{\operatorname{Res}}}

\def\Hom{\mathrm{Hom}}

\def\wt{\mathrm{wt}}

\def \Im{\mathrm{Im}}

\def \rad{\mathrm{rad}}

\def \mf#1#2#3#4{
\xymatrix{{#1}\  \ar@<0.4ex>[r]^{{#2}} & \ {#4}
\ar@<0.4ex>[l]^{{#3}}
}
}

\def \mfs#1#2#3#4{\!
\xymatrix@C=1.5em{{#1} \! \ar@<0.2ex>[r]^{{#2}} & \! {#4}
\ar@<0.2ex>[l]^{{#3}}
}
\!}

\def \mfl#1#2#3#4{
\xymatrix@C=2.6em{{#1}\  \ar@<0.4ex>[r]^{{#2}} &\  {#4}
\ar@<0.2ex>[l]^{{#3}}
}
}

\def \mfss#1#2#3#4{\!
\xymatrix@C=1.5em{{#1} \ar@<0.3ex>[r]^{{#2}} & {#4}
\ar@<0.3ex>[l]^{{#3}}
}
\!}

\begin{document}

\title[Period integrals of type $\mathrm{A_2, B_2}$ and $\mathrm{G_2}$]{A view on Elliptic Integrals \\ from Primitive Forms\\
\vspace{0.2cm}
{\footnotesize  (Period Integrals of type $\mathrm{A_2, B_2}$ and $\mathrm{G_2}$)}
}

\author{Kyoji Saito}
\address{
Institute for Physics and Mathematics of the Universe, the University of Tokyo, 5-1-5 Kashiwanoha, Kashiwa, 277-8583, Japan}
\email{kyoji.saito@ipmu.jp}

\address{
Research Institute for Mathematical Sciences, Kyoto University, Sakyoku Kitashirakawa, Kyoto, 606-8502, Japan}
\address{
Laboratory of AGHA, Moscow Institute of Physics and Technology, 9 Institutskiy per., Dolgoprudny, Moscow Region, 141700, Russian Federation}

\date{}
\maketitle

 {\renewcommand{\baselinestretch}{0.1}


Elliptic integrals, since Euler's finding of addition theorem 1751, has been studied extensively from various view points. Present paper gives a  view point from primitive integrals of types $\mathrm{A_2}, \mathrm{B_2}$ and $\mathrm{G_2}$ for the three families of elliptic curves of Weierstrass, Jacobi-Legendre and Hesse, respectively. 
We solve Jacobi inversion problem for the period maps in the sense explained in the introduction (c.f.\ \cite{Si} Chap.1,13) 
by introducing certain generalized  Eisenstein series of types $\mathrm{A_2}, \mathrm{B_2}$ and $\mathrm{\mathrm{G_2}}$, which generate the ring of modular forms on the period domain for the congruence subgroups $\Gamma_1(N)$ {\footnotesize ($N=1,2$ and $3$)}.  In particular, Eisenstein series of type $\mathrm{B_2}$ includes the case of weight two, and Eisenstein series of type $\mathrm{G_2}$ includes the cases of weight one and two, which seem to be  of new feature. The goal of the paper is a partial answer to the discriminant conjecture, which claims an existence of certain cusp form  of weight 1 with character of topological origin, giving a power root of the discriminant form (Aspects Math., E36,p.\ 265-320.\ 2004). 
See \S12 Concluding Remarks for more about back grounds of present paper.

\tableofcontents

\section{Introduction}
We study period integrals for three families of elliptic curves: Weierstrass family, Legendre-Jacobi family and Hesse family. We call the families of type $\mathrm{A_2}, \mathrm{B_2}$ and $\mathrm{G_2}$, respectively, according to the lattice structure of vanishing cycles studied in \S3.  In order to treat these three types  simultaneously, we use notation $\mathrm{I_2}(p)$ for the dihedral groups \cite{Bourbaki} and  identify
$\mathrm{A_2}=\mathrm{I_2}(3)$, $\mathrm{B_2}=\mathrm{I_2}(4)$ and $\mathrm{\mathrm{G_2}}=\mathrm{I_2}(6)$. 
Then the theory for the  type $\mathrm{A_2}$ is nothing but the classical well-known theory of elliptic integrals for the Weierstrass family of elliptic curves.  The purpose of the present paper is to show that there exist some parallel worlds for types $\mathrm{B_2}$ and $\mathrm{G_2}$, even in a somewhat deeper manner.

\smallskip
 We refer the reader to \S12 Concluding Remarks for the motivation and the background of the present study. Therefore, in the present introduction, we restrict ourselves to  brief explanations of the contents. 

\smallskip
As mentioned, we study the three  families of elliptic curves of types $\mathrm{A_2}$, $\mathrm{B_2}$ and $\mathrm{\mathrm{G_2}}$ \eqref{eq:equation}. 
In the first sections \S\S2-4, we study the geometric and topological aspects of the family. Namely, the geometric family of affine (i.e.\ punctured by points at infinity) elliptic curves over the base space $S_{\mathrm{I_2}(p)}$ together with the discriminant loci $D_{\mathrm{I_2}(p)}$ is introduced in \S2. The number $N:=[p/2]$ of points at infinity is called the {\it level}. That is, we study the period maps for the families of level 1, 2 and 3.

\smallskip
Then in \S3, by the help of a real structure of the family, the lattice of vanishing cycles in the fiber is described  in terms of classical  root lattice of type $\mathrm{A_2}$, $\mathrm{B_2}$ and $\mathrm{\mathrm{G_2}}$. 
The fundamental group of the complement $S_{\mathrm{I_2}(p)}\setminus D_{\mathrm{I_2}(p)}$ is described in \S4 in terms of the Artin group of type $\mathrm{A_2}$, $\mathrm{B_2}$ and $\mathrm{\mathrm{G_2}}$. Actually, the image of  the monodromy representation becomes the congruence modular group $\Gamma_1(N)$ of level $N=1,2$ and 3 \eqref{eq:congruence1}.
 In particular, the expression of the modular group leads to a construction of its certain character $\vartheta_{\mathrm{I_2(p)}}$ \eqref{eq:character} whose $k(\mathrm{I_2(}p))$th power \eqref{eq:k} is the sign character: $\Gamma_1(N)\to \{\pm1\}$.  We shall comeback to the character $\vartheta_{\mathrm{I_2(p)}}$ at the final Theorem \ref{cuspform} of the present paper. 

\smallskip
The primitive form and associated period map, i.e.\ the map obtained by the values of  integrals of the primitive form over vanishing cycles from the monodromy covering space $\widetilde{S}_{\mathrm{I_2}(p)}$ of the compliment of the discriminant to the period domain $\widetilde{\mathbb{H}}$, are introduced in \S5. We ask whether the period map is invertible from the period domain $\widetilde{\mathbb{H}}$ to the original defining domain $S_{\mathrm{I_2}(p)}\setminus D_{\mathrm{I_2}(p)}$ (in the present paper, we shall call this question {\it Jacobi inversion problem}). 

\smallskip
In \S\S6-9, we study the analytic aspects of the theory.  Namely,  we fix a point at infinity and consider the indefinite integral of the primitive form over open paths on the elliptic curve starting from the point at infinity. Actually, regarding this integral value $z$ as the time variable, the inverse function of the integral (a meromorphic doubly periodic function in $z$) becomes a solution of the Hamilton equation \eqref{eq:Hamiltonean} of the motion together with the energy constraint \eqref{eq:energy}.  Conversely, any formal meromorphic solutions of the Hamilton equations at a point infinity are convergent and the set of solutions is in one to one correspondence with the set of points at infinity (which were chosen as the initial condition for the indefinite integral) (\S7 Lemma \eqref{eq:LaurentSolution}).

\smallskip
This equivalence of geometric solutions and formal solutions is a key step in the next stage to solve the Jacobi inversion problem. Namely, in the next section \S9, by a help of the Hamilton equations \eqref{eq:Hamiltonean} and \eqref{eq:energy},  we can expand those global meromorphic solutions on the full $z$-plane into  partial fractions by the help of Weierstrass $\mathfrak{p}$-functions or zeta-functions, which depends  only on the period variable $(\omega_1,\omega_0)\in \widetilde{\mathbb{H}}$.   Then, again by expanding these global meromorphic functions into Laurent series at the origin $z=0$,  as coefficients of the expansion, we obtain an infinite sequence of functions in $(\omega_1,\omega_0)\in \widetilde{\mathbb{H}}$, which we call the {\it Eisenstein series of type $\mathrm{I_2}(p)$} in \S10. Actually, the ring of Eisenstein series, as functions defined on $\widetilde{\mathbb{H}}$, by pulling back by the period map, is identified with ring of the Cartesian coordinate ring of the parameter space 
$S_{\mathrm{I_2}(p)}$ (Theorem \ref{primitiveautomorphicform}). 
Then, those Eisenstein series enable us to construct the inversion map $\widetilde{\mathbb{H}} \to S_{\mathrm{I_2}(p)}$, which finally leads us to the solution of Jacobi inversion problem in the section \S9.

\smallskip
Next in the section \S10, we identify the ring of Eisenstein series of types $\mathrm{A_2}, \mathrm{B_2}$ and $\mathrm{G_2}$ with the ring of modular forms of the congruence group $\Gamma_1(N)$, where the ring of modular forms were already determined explicitly by Aoki-Ibukiyama \cite{A-S}. In order to determine the homomorphism exactly, we have  to determine the exact values of the Eisenstein series at cusps of the modular group. Actually, Eisenstein series of type $\mathrm{A_2}$ are classical, whose Fourier expansions are well known. Eisenstein series of types $\mathrm{B_2}$ and $\mathrm{G_2}$ are no-longer classical. However, when their weights are larger or equal than 3,  then their expressions are still obtained by a ``shift" of the constant term of the classical Eisenstein series (see \eqref{eq:classicalEisenstein}). Then, it is still possible to evaluate the values of the shifted classical Eisenstein series at cusps by  using either the classical Riemann zeta function or Dirichlet's L-function (see \S10 Table 2). 

\smallskip
However, Eisenstein series of types $\mathrm{B_2}$ and $\mathrm{G_2}$ of weights less than or equal to 2 have expression by special values of Weierstrass $\mathfrak{p}$-functions or by difference of special values of Weierstrass zeta-functions, which seem to be less standard. We determine their values at cusps in a separate note joint with Aoki \cite{A-S} (c.f.\ \cite{D-S,Lang}). These determinations lead us to the identification  of  the ring Eisenstein series of types $\mathrm{A_2}, \mathrm{B_2}$ and $\mathrm{G_2}$  with the ring of modular forms of $\Gamma_1(N)$ ($N=1,2$ and $3$) (Theorem \ref{Eisen-Modular}).
It is marvelous that  this identification induces further a one to one correspondence of the set of irreducible components of the discriminant of our family of elliptic curves with the set of modular form (up to constant factors) which vanishes exactly once at one $\Gamma_1(N)$-orbit of cusps (\S10 Lemma \ref{cusp-disc}). 

\smallskip
This  modular form, generating the ideal vanishing at an equivalence class of cusps, on one hand as an equation for an irreducible component of the discriminant, is nowhere vanishing on $\widetilde{\mathbb{H}}$. On the other hand as a modular form described by theta-series, it has integral Fourier coefficients.  Such form can be expressed as suitable quotient of products of shifted Dedekind eta-functions (see Table 5).  Doing this for all cusps, we determine the eta-product expressions of the discriminant \eqref{eq:discriminant-eta} and the reduced discriminant  \eqref{eq:reduced-discriminant-eta} of our family of elliptic curves.  

\smallskip
These expressions lead us to the final Theorem  \ref{cuspform}:

\smallskip
(1) {\it  
There exists a cusp form of weight 1 of the congruence group $\Gamma_1(N)$ 
with respect to the character $\vartheta_{\mathrm{I_2(p)}}$ in \S4 such  the reduced discriminant of the family of elliptic curves for the type $\mathrm{I_2}(p)$ is identified with the  $2k(\mathrm{I_2}(p))$th power of the cusp form ($p=3,4$ and $6$). }

(2) {\it The discriminants of the families of elliptic curves of all types $\mathrm{A_2}$, $\mathrm{B_2}$ and $\mathrm{\mathrm{G_2}}$, up to a rational constant times the power $\pi^{12}$, are identified with the modular discriminant $q\prod_{n=1}^\infty(1-q^n)^{24} $ of weight 12}.

\medskip 
These two statements give positive answers to the discriminant conjecture 6 posed in \cite{S5} \S6, 
and we close the present paper. 
  

\section{\!\!Families of elliptic curves of type  $\mathrm{A_2}$, $\mathrm{B_2}$ and $\mathrm{\mathrm{G_2}}$}

We start with the three families, defined by the equations \eqref{eq:equation},  of affine elliptic curves in the $(x,y)$-plane parametrized by two weighted homogeneous coordinates  $\underline{g}=(g_s,g_l)$ (where $s$ and $l$ stands for small or large weights of the coordinates so that the equations become weighted homogeneous polynomials. See Table 1). 
\begin{equation}
\begin{array} {rcl}
\label{eq:equation}
\mathrm{A_2} & : &
F_{\mathrm{A_2}}(x,y,\underline{g}):=y^2-(4x^3-g_{s}x-g_l) \\
\\
 \mathrm{B_2} &: &
F_{\mathrm{B_2}}(x,y,\underline{g}):=y^2-(x^4-g_{s}x^2+g_l+g_{s}^2/8) \\
\\
\mathrm{G_2} & : &
F_{\mathrm{G_2}}(x,y,\underline{g}):=x(y^2-x^2)+ g_{s}(3x^2+y^2) -g_l-2g_{s}^3.
\end{array}
\end{equation}
Historically, they are called  Weierstrass、Legendre-Jacobi and Hesse family of elliptic curves (Exactly, compared with historical expressions, some coordinate change is done from a view point of primitive forms. See Footnote 1 and 2).  


As given in the left side of the equations, we call the families by the names of root systems of rank 2, i.e.\ by $\mathrm{A_2}, \mathrm{B_2}$ and $\mathrm{G_2}$.  \footnote
{To be exact, when we call  a  family of type $\mathrm{A_2}, \mathrm{B_2}$ or $\mathrm{G_2}$, we shall mean the family given in the equation \eqref{eq:equation} together with an action of an automorphism given in \eqref{eq:autom}  of the family. They are subfamilies of  the bigger families of affine elliptic curves of type $\mathrm{A_2, A_3}$ and $\mathrm{D_4}$ which admits an automorphism $\sigma$ of order 1, 2 and 3, respectively, so that 
the present family is the subfamilies over the parameters which are fixed by $\sigma$ (see Footnote 9).
The study of the periods for the types $\mathrm{A_3}$ and $\mathrm{D_4}$ (unpublished) are beyond the scope of present paper  and shall appear elsewhere. See also Footnote 7 and 14.
}
 We shall justify this renaming in \S3 from a view point of vanishing cycles. 
 In order to treat these three cases simultaneously, let us use the notation $\mathrm{I_2}(p)$ for dihedral groups \cite{Bourbaki}. Namely, let us recall the following identifications. 
\begin{equation}
\label{eq:dihedral}
\mathrm{A_2}= \mathrm{I_2}(3),\quad  \mathrm{B_2}= \mathrm{I_2}(4) , \quad \mathrm{G_2}= \mathrm{I_2}(6) 
\end{equation}
These three cases are exactly the cases when the dihedral group of type $\mathrm{I_2}(p)$  is crystallographic, corresponding to classical root systems defined over $\mathbb{Z}$ which shall play crucial roles in the present paper.

 Let us give Table of weights of the variables in the equation  $F_{\mathrm{I_2}(p)}$. We  normalize them so that the total weight of  $F_{\mathrm{I_2}(p)}$  is equal to 1.
$$
\begin{array}{ccccccccc}
 \vspace{0.1cm}
         &\!\! \wt(F_{\mathrm{I_2}(p)}) \!\!&\!  \wt(x)  & \wt(y) \!\!&\! s\!\!=\!\!\wt(g_s) \!\!&\!  l\!\!=\!\!\wt(g_l)  \!\!&\!\!  \wt(\Delta_{\mathrm{I_2}(p)} )\!\!& \wt(z) \\
\mathrm{A_2}  \!\! &   1                    &    1/3   &   1/2   &     2/3       &\!      1       &\!  2  &\! \!   -1/6 \\
\mathrm{B_2}  \!\! &   1                    &    1/4   &   1/2   &     1/2     & \!   1        & \! 3     & \!\!  -1/4  \\
\mathrm{G_2}  \!\! &   1                    &    1/3   &   1/3   &     1/3     &\!    1      &\!  4       &\!\!   -1/3 \\
\end{array}
$$
\centerline{\rm\large  Table 1:\quad {\normalsize Weights of functions and coordinates}}

\bigskip
\noindent
Here, $\Delta_{\mathrm{I_2}(p)}$ and $z$ are the discriminant introduced below \eqref{eq:discriminant} and the Hamilton time coordinate $z$ of the elliptic curve $\tilde{E}_{\mathrm{I_2}(p)}$ introduced in \S6 \eqref{eq:Hamilton-time}, respectively. Note that the weight {\small $\wt(z)\! := \! \wt(x) \! + \! \wt(y) \! - \! \wt(F_{\mathrm{I_2}(p)})$} of the variable $z$  is negative caused from a classification of vanishing cycles, where the negativity plays an essential role in the present paper.

The equations \eqref{eq:equation} define geometric families of affine elliptic curves. Namely, let  us consider the morphisms:
\begin{equation}
\label{eq:family}
\begin{array}{rrl}
\qquad \qquad  & \pi_{\mathrm{I_2}(p)}\ : & X_{\mathrm{I_2}(p)} \ \longrightarrow  \ S_{\mathrm{I_2}(p)}, \quad \qquad p=3,4 \text{ or } 6\\
\end{array}
\end{equation}
where $S_{\mathrm{I_2}(p)}$ 
is the two dimensional complex parameter space of the coordinates $\underline{g}=(g_s,g_l)$, 
\footnote{
The coordinates $\underline{g}=(g_s,g_l)$ are, up to constant factor, flat coordinates of the family \cite{S0} \cite{S5}, whose weights are equal to exponents of $\mathrm{I_2}(p)$ plus $1/p$ \cite{Bourbaki}.}
 $X_{\mathrm{I_2}(p)}$ is the affine subvariety  in $\mathbb{C}^2\times S_{\mathrm{I_2}(p)}$ defined by the equation $F_{\mathrm{I_2}(p)}=0$,  and $\pi_{\mathrm{I_2}(p)}$ is the morphism induced on $X_{\mathrm{I_2}(p)}$ from the projection to the second factor $S_{\mathrm{I_2}(p)}$, respectively. 

The relative critical point set $C_F$ of the map $\pi_{\mathrm{I_2}(p)}$ defied by $\frac{\partial F_{\mathrm{I_2}(p)}}{\partial x}=\frac{\partial F_{\mathrm{I_2}(p)}}{\partial y}=0$ 
lies proper finite over the base parameter space $S_{\mathrm{I_2}(p)}$. The image set $ \pi_{\mathrm{I_2}(p)}(C_F)$  in  $S_{I_2(p)}$ is a one codimensional subvariety 
\begin{equation}
\label{eq:discriminantloci}
D_{\mathrm{I_2}(p)} \ \subset \ S_{\mathrm{I_2}(p)},
\end{equation}
called the {\it discriminant loci}, which is parametrizing singular elliptic curves.  The defining equation $\Delta_{\mathrm{I_2}(p)}$ of the discriminant together its multiplicity (up to a constant factor) is given by   (c.f. \cite{S1}) 
\vspace{-0.1cm}
\begin{equation}
\label{eq:discriminant}
\begin{array}{rclcl}
\vspace{0.1cm}
\Delta_{\mathrm{A_2}} &= & -27g_l^2\ +\ g_{s}^3=({\sqrt{27}}g_l+g_{s}^{3/2})(-{\tiny \sqrt{27}}g_l+g_{s}^{3/2}),
\\ 
\vspace{0.1cm}
\Delta_{\mathrm{B_2}} & = & (8g_l+g_{s}^2)(-8g_l+g_{s}^2)^2,\\
\vspace{0.1cm}
\Delta_{\mathrm{G_2}} & = & (g_l+2g_{s}^3)(-g_l+2g_{s}^3)^3 .
\end{array}
\vspace{-0.2cm}
  \footnote{
The decomposition in case of type $\mathrm{A_2}$ has meaning  only when we consider the real parameter space $S_{\mathrm{I_2}(p)}^\mathbb{R}$ and it was unnecessary to fix the branch.
}
\end{equation}

%

The fiber $E_{\mathrm{I_2}(p),\underline{g}}:=\pi_{\mathrm{I_2}(p)}^{-1}(\underline{g})$
over a point $\underline{g}$  in $S_{\mathrm{I_2}(p)}$
is an affine open curve in the $(x,y)$-plane which is compactified 
to an elliptic curve by adding  a single, two or three points at infinity according as the types $\mathrm{A_2}$, $\mathrm{B_2}$ or $\mathrm{G_2}$. \footnote{
A  compactification of $E_{\mathrm{I_2}(p),\underline{g}}$ is obtained by the curve defined in $\mathbb{P}^2$ by the homogenization of the equation $F_{\mathrm{I_2}(p)}$. In case of type $\mathrm{A_2}$, it is tangent of order 3 to the infinite line at a single infinite point, in case of $\mathrm{B_2}$, it is twice tangent to the infinite line at the same infinite point (so that we need to normalize the compactified curve at the point to separate branches), and in case of type $\mathrm{G_2}$, it is intersecting transversally with the infinite line at three distinct infinite points.
 }
 Using the identification \eqref{eq:dihedral}, we denote by $[p/2]$ the number of the points at infinity for the family of the type $\mathrm{I_2}(p)$.  

Let us denote the points at infinity  by
\begin{equation}
\label{eq: atinfty}
\infty_1, \ \cdots \ , \infty_{[p/2]}
\end{equation}
(where exact labeling is fixed in  
\S3), the compactified curve by 
 \begin{equation}
 \label{eq:compact1}
 \overline{E}_{\mathrm{I_2}(p),\underline{g}}\ =\ E_{\mathrm{I_2}(p),\underline{g}} \cup \overset{[p/2]}{\underset{i=1}{\cup}} \{\infty_i\}   
 \text{\quad  and \quad}
  \overline{E}_{\mathrm{I_2}(p),\underline{g}}\cap X_{\mathrm{I_2}(p)} = \ E_{\mathrm{I_2}(p),\underline{g}}  \!\!\!\!
 \end{equation}
and the fiberwise compactified family by
 \begin{equation}
 \label{eq:compact2}
\overline{\pi}_{\mathrm{I_2}(p)}\  : \ \overline{X}_{\mathrm{I_2}(p)} = X_{\mathrm{I_2}(p)}  \cup \overset{[p/2]}{\underset{i=1}{\cup}} \{\infty_i \times S_{\mathrm{I_2}(p)}\}   \ \longrightarrow  \ S_{\mathrm{I_2}(p)}.
 \end{equation}
Actually,  $\overline{X}_{\mathrm{I_2}(p)}$ is smooth at the points $\infty_i$ and $\overline{\pi}_{\mathrm{I_2}(p)}$ is transversal to the divisors $\infty_i \times S_{\mathrm{I_2}(p)}$ (see Footnote 4). 
  The weighted homogeneity of the equation $F_{\mathrm{I_2}(p)}$ implies that there is the $t\in \mathbb{C}^\times$-action $(x,y,g_s,g_l)\mapsto (t^{\wt(x)}x,t^{\wt(y)}y,t^{\wt(g_s)}g_s,t^{\wt(g_l)}g_l)$ leaving the space $X_{\mathrm{I_2}(p)}$ invariant so that the \eqref{eq:family} is equivariant with the action.  We note that the action on $X_{\mathrm{I_2}(p)}$ extends to $\overline{X}_{\mathrm{I_2}(p)}$ continuously so that the divisor $\infty_i \times S_{\mathrm{I_2}(p)}$ is invariant and the morphism  \eqref{eq:compact2} is still equivariant.
 
The restriction of the family \eqref{eq:family} over the complement $S_{\mathrm{I_2}(p)}\setminus D_{\mathrm{I_2}(p)}$ of the discriminant loci (i.e.\ the  space of regular values of the family \eqref{eq:family}) gives a locally topologically trivial family of punctured elliptic curves, which also induces a topologically locally trivial  family of compact smooth elliptic curves. 

\medskip
\noindent
{\it Note.}  Compact elliptic curve has well-known abelian group structure (after choosing the origin). Then, we have the following elementary, but important fact, which we shall reformulate in  \eqref{eq:deltaperiod2} \S5.

\medskip
\noindent
{\bf Fact 1.}
{\it The difference of two infinite points $[\infty _i]$ and $[\infty_j]$ ($1\le i,j\le [p/2]$) is a torsion element of order $[p/2]$.}

\section{Real elliptic curves and  Vanishing cycles} 

We study some real geometry of the family \eqref{eq:family}. 
It provides a description of the middle (i.e.\ one dimensional) homology groups of the affine elliptic curves $E_{I_2(p), \underline{g}}$ in terms of  vanishing cycles of root systems of types $\mathrm{A_2}, \mathrm{B_2}$ and $\mathrm{\mathrm{G_2}}$.

Let us denote by $X_{\mathrm{I_2}(p)}^{\mathbb{R}}$
and 
$S_{\mathrm{I_2}(p)}^{\mathbb{R}}$  the subspaces of $X_{\mathrm{I_2}(p)}$
and 
$S_{\mathrm{I_2}(p)}$ 
consistings of the points 
  where the coordinates $(x,y)$ and $\underline{g}$ take real values,  
and call them the real total space and the real parameter space of the family \eqref{eq:family}, respectively.
   In the real parameter space, we are interested in a particular connected component of the compliment  $S_{\mathrm{I_2}(p)}^{\mathbb{R}} \setminus D_{\mathrm{I_2}(p)}$ of the real discriminant loci $D_{\mathrm{I_2}(p)}\cap S_{\mathrm{I_2}(p)}^{\mathbb{R}}$, so called,  the totally real component $\Gamma_{\mathrm{I_2}(p)}$
\footnote{
Actually, the discriminant loci $D_{\mathrm{I_2}(p)}$ of the family \eqref{eq:family} is identified with the discriminant loci in the quotient space of a vector space by the irreducible finite reflection group action of type $\mathrm{I_2}(p)$.  Then, there is the unique connected component $\Gamma$ of the complement of the real discriminant loci such that the inverse image of a point in $\Gamma$ is totally real in the original vector space of the representation \cite{S3}. Then, $\Gamma$ is homeomorphic to a chamber of type $\mathrm{I_2}(p)$ and $\partial_\pm\Gamma_{\mathrm{I_2}(p)}$ are homeomorphic to its walls.  These facts are used for a calculation of fundamental group of the compliment of the discriminant loci (cf. \S4).
}
defined by 
\begin{equation}
\label{eq:totallyreal}
\Gamma_{\mathrm{I_2}(p)} \quad : \quad   -|g_s|^{p/2}<cg_l<|g_s|^{p/2} \ \text{ and} \quad g_s>0 , 
\end{equation}
and its boundary as the union of two edges
\begin{equation}
\label{eq:boundary}
\partial_{\pm}\Gamma_{\mathrm{I_2}(p)} \quad : \quad   cg_l=\pm |g_s|^{p/2} \ \text{  and} \quad g_s>0 
\end{equation}
 and the origin $\{0\}$ (here  $c=\sqrt{27},8$ or $1/2$ according as $p=3,4$ or $6$ \vspace{0.1cm}
\eqref{eq:discriminant}).  
See Figure 1 and its following explanations.

 \begin{figure}[h]
 \center
  \hspace{-0.3cm}\includegraphics[width=11.8cm]{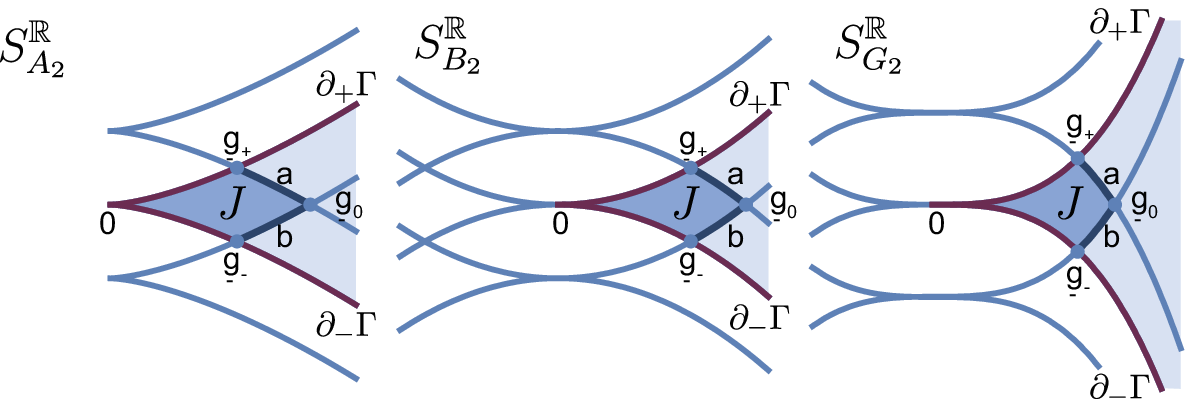}
\vspace{0.1cm}
\noindent
{\bf Figure\! 1.}Real\! discriminant\! $D_{\mathrm{I_2}(p)}^{\mathbb{R}}$,\! totally\! real\! component\! $\Gamma_{\mathrm{I_2}(p)}$\! and\! $J_{\mathrm{I_2}(p)}$
 \end{figure}

\noindent
 {\small Explanation:  The union of curves pathing through the origin $0$ is the real discriminant $D_{\mathrm{I_2}(p)}\cap S_{\mathrm{I_2}(p)}^{\mathbb{R}}$. The union of the lightly and darkly shaded areas, bounded by the real discriminant, is the totally real component $\Gamma_{\mathrm{I_2}(p)}$.  The boundary  $\partial\Gamma_{\mathrm{I_2}(p)}$ is the union of the upper boundary $\partial_+\Gamma_{\mathrm{I_2}(p)}$, lower boundary $\partial_-\Gamma_{\mathrm{I_2}(p)}$ and the origin $0$. We choose points $\underline{g}_0$ and  $\underline{g}_\pm$  generically in $\Gamma_{\mathrm{I_2}(p)}$ and in $\partial_\pm\Gamma_{\mathrm{I_2}(p)}$, respectively.  
 See \S4 for the role of the other curves, the dark shaded area, called $J$, and its boundary component $a$ and $b$.}

\medskip

 %
 We observe the following {\bf Facts 1-5.}  of vanishing cycles in the family of affine elliptic curves $E_{\mathrm{I_2}(p)}$. 
 Proof is achieved by explicit direct calculations (see Figure 2), and we omit its details.

\begin{figure}[h]
 \center
{\center
\hspace{0.1cm}\includegraphics[width=11.3cm]{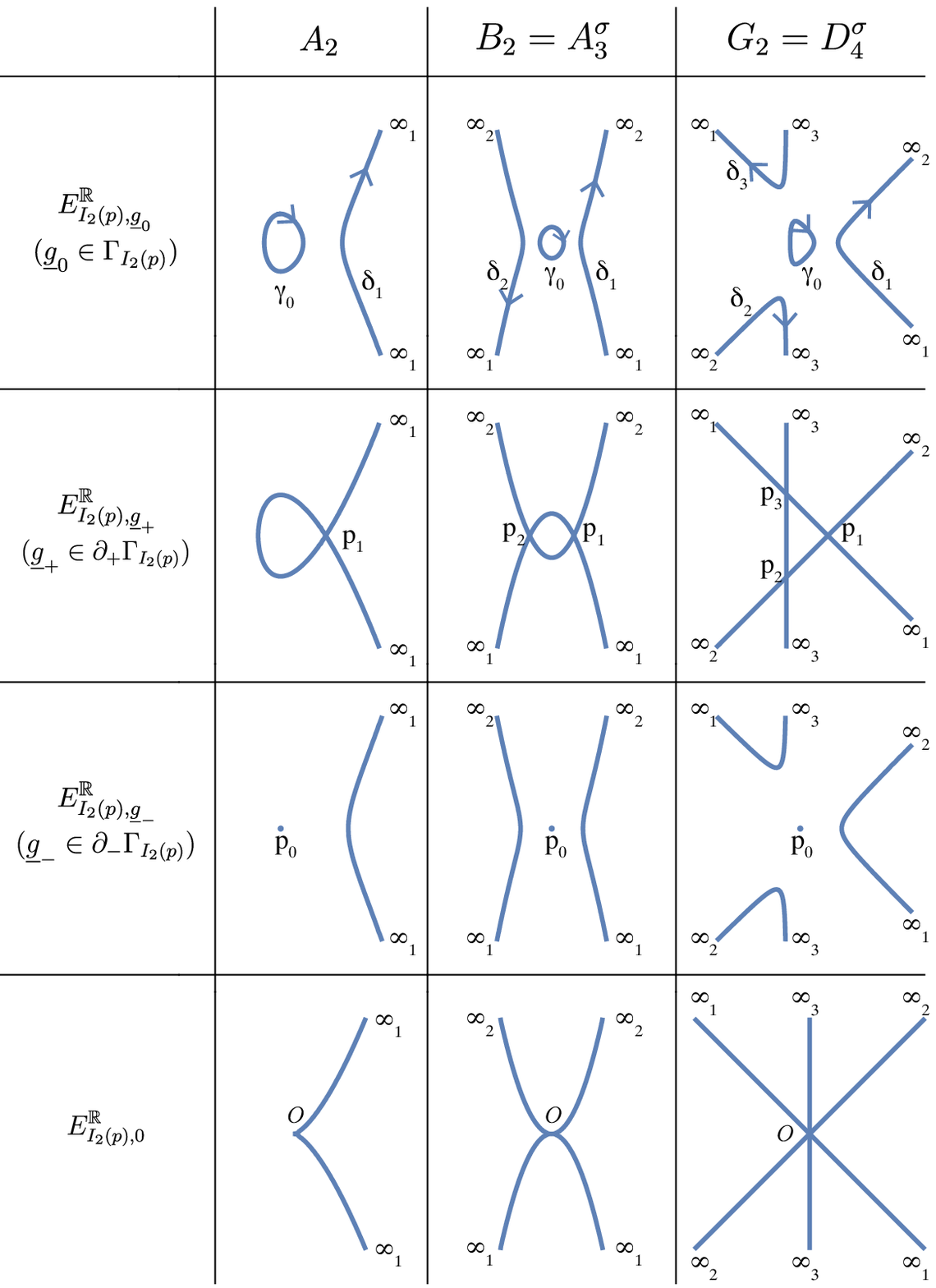}
\vspace{0.2cm}
}

\noindent
 \centerline{ {\bf Figure 2.} { Real elliptic curves of type $\mathrm{A_2}$, $\mathrm{B_2}$ and $\mathrm{G_2}$}}\\
\vspace{0.2cm}
\footnotesize{The curves in the second and the third column of the table, called of type $\mathrm{B_2}$\\
 and $\mathrm{G_2}$, are also regarded as the real affine elliptic curve of type  $A_3$ and $D_4$\\
 (see Footnote 1.) together with the action of an automorphism $\sigma_{\mathrm{I_2}(p)}$ \eqref{eq:autom}.
}
\end{figure}

\medskip
Choose a base point $\underline{g}_0\in \Gamma_{\mathrm{I_2}(p)}$ (recall Figure 1), and consider the affine  elliptic curve $E_{\mathrm{I_2}(p),\underline{g}_0}=\pi_{\mathrm{I_2}(p)}^{-1}(\underline{g})$. 

\bigskip
\noindent
{\bf Fact}  
\noindent
{\bf 1.} 
The real affine elliptic curve $E_{\mathrm{I_2}(p),\underline{g}_0}^{\mathbb{R}}:=E_{\mathrm{I_2}(p),\underline{g}_0} \cap X_{\mathrm{I_2}(p)}^{\mathbb{R}}$ (exhibited in the first row of  Figure 2.) consists of  a single compact component (=oval), which we call
\begin{equation}
\label{eq:gammazero}
\gamma_0 
\end{equation} 
and  $[p/2]$-number of non-compact connected components (arcs) $\delta_1,\cdots,\delta_{[p/2]}$. The non-compact components are bounded by the points at infinity in the compactification $ \overline{E}_{\mathrm{I_2}(p),\underline{g_0}}$. 
We fix  orientations of the arcs $\delta_i$'s and the numbering of the points at infinity such that the cyclic union:
\begin{equation}
\label{eq:realcycle}
\{\infty_1\} \cup \delta_1 \cup \{\infty_2\} \cup \delta_2 \cup \cdots \cup \delta_{[p/2]} \cup \{\infty_1\}
\end{equation}
form an oriented closed cycle (i.e.\ $\partial \delta_1=\infty_2-\infty_1,\cdots, \partial \delta_{[p/2]}=\infty_1-\infty_{[p/2]}$). 
Then, this cycle is homologous to  $\gamma_0$ after choosing an orientation of $\gamma_0$ accordingly. 
However, this condition determine only a cyclic ordering of arcs. So, there still remains an ambiguity of reversing the orientation. We  resolve this ambiguity in the following observation.

\medskip
\noindent
{\bf Fact 2.}  There exists a unique choice of an orientation and a cyclic numbering of arcs $\delta_i$'s which satisfies the initial direction condition
\begin{equation}
\label{eq:delta1}
\begin{array}{rcccc}
(\delta_1,\infty_1) & \subset & (\mathbb{R}_{>0},+\infty)\times (\mathbb{R}_{<0},-\infty) \\
\end{array}
\end{equation}
Here, $(\delta_1,\infty_1)$ in LHS means the germ of the arc $\delta_1$ at $\infty_1$, and the $(\mathbb{R}_{>0},+\infty)\times (\mathbb{R}_{<0},-\infty)$ in RHS means a neighborhood of $(+\infty,-\infty)$ in $\mathbb{R}_{>0}\times \mathbb{R}_{<0} \subset \mathbb{C}\times\mathbb{C}$.\footnote
{A priori,  there exists neither a guarantee that there exists such pair $(\delta_i,\infty_i)$ satisfying the condition \eqref{eq:delta1}, nor a guarantee that the condition \eqref{eq:delta1} choose the pair $(\delta_i,\infty_i)$ uniquely.  Therefore, Fact 2.\ claims that the conditiion actually choose the unique one (which is easily confirmed from Figure 2). Further more,  we remark that the choice \eqref{eq:delta1} in the real blow up space of $\mathbb{P}^2(\mathbb{R})$, in case of type $\mathrm{B_2}$, separates two infinity points $\infty_1$ and $\infty_2$, which the $\mathbb{P}^2(\mathbb{C})$-compactification did not separate (recall Footnote 3). Indeed, we have $(\delta_2,\infty_2)  \subset  (\mathbb{R}_{<0},-\infty)\times (\mathbb{R}_{>0},+\infty) $.
}


\medskip
\noindent
{\bf Fact 3.} Let us move the  point $\underline{g}_0$ inside $\Gamma_{\mathrm{I_2}(p)}$  to a point  $\underline{g}_+$ on the upper boundary edge $\partial_+\Gamma_{\mathrm{I_2}(p)}$ (e.g.\ move along the path $a$ in Figure 1). Then, accordingly,  the cycle $\gamma_0$ and each arc  $\delta_i$ in the fiber $E_{\mathrm{I_2}(p),\underline{g}_0}$  
are getting close to each other, and, finally at $\underline{g}_+ \in \partial_+\Gamma_{\mathrm{I_2}(p)}$ on the boundary, they intersect to a node $p_i$ of the curve $E_{\mathrm{I_2}(p),\underline{g}_+}$  ($i=1,\cdots,[p/2]$) (see the second row  of Figure 2).

\medskip
\noindent
{\bf  Fact 4.}  Let us move the point $\underline{g}_0$ inside $\Gamma_{\mathrm{I_2}(p)}$  to a point  $\underline{g}_-$ on the lower boundary edge $\partial_-\Gamma_{\mathrm{I_2}(p)}$ (e.g.\ move along the path $b$ in Figure 1). Then, accordingly,  the cycle  $\gamma_0$ in the fiber $E_{\mathrm{I_2}(p),\underline{g}_0}$ pinches to a Morse singularity $p_0$ in the fiber $E_{\mathrm{I_2}(p),\underline{g}_-}$ (see the third row of Figure 2). 
This implies that the cycle $\gamma_0$ is the vanishing cycle generated by the Morse singularity $p_0$.

\medskip
\noindent
{\bf Fact 5.}  Each nodal point $p_i$ generates a vanishing cycle
\begin{equation}
\label{eq:gammaai}
\gamma_i
\end{equation}
$i=1,\cdots,[p/2]$, in the nearby elliptic curve  $E_{\mathrm{I_2}(p),\underline{g}_0}$, which intersects with  $\gamma_0$ and $\delta_i$ (recall Fact 1.) transversally. Since $\gamma_i$ lies in the complexification $E_{\mathrm{I_2}(p),\underline{g}}$, we do not exhibit it in Figure 2, but some conceptual expression of it shall be given  in the first row of Figure 3. We choose the orientation of $\gamma_i$ by the following sign condition on the intersection: 
\begin{equation}
\label{eq:intersection1}
\begin{array}{rll}
&      \langle \gamma_0,  \gamma_i \rangle \ =\ \langle \delta_i,  \gamma_i \rangle \ =\   1   & \qquad (i=1,\cdots,[p/2]) \\
\qquad &     \langle \gamma_i,  \gamma_j \rangle \ =\ \langle \gamma_j,  \gamma_i \rangle \ =\   0  & \qquad (i,j=1,\cdots,[p/2]).
\end{array}
\end{equation}   
Here we denote  by $\langle\gamma,\gamma'\rangle$ the intersection number of paths $\gamma$ and $\gamma'$ whose sign is fixed as follows:
{\small The orientation of  $E_{\mathrm{I_2}(p),\underline{g}_0}$ as a real surface is fixed by its complex structure. If a 
path $\gamma'$ crosses another path $\gamma$  {\it counter-clockwisely}, \footnote{
 A  path in the complex-plane crosses other path {\it counter-clockwisely},  if and only if their tangent vectors $a$ and $b$ at  the crossing point satisfies $\Im(a/b)>0$.
}
then the local intersection number is $\langle\gamma,\gamma'\rangle=+1$. 
}
 
 \smallskip
\begin{rem} 
The most degenerated real curve $E_{\mathrm{I_2}(p),0}^{\mathbb{R}}$ is exhibited in the 4th row of Figure 2.
\end{rem}


\smallskip
As a consequence of Facts 1-5., we obtain the following description of the homology group of the affine elliptic curve $E_{\mathrm{I_2}(p),\underline{g}_0}$. 

\medskip
\noindent
{\bf  Fact 6.}  
{\it The classes of $\gamma_0$ and  $\gamma_i$ ($i=1,..,[p/2]$) in the first homology group of $E_{\mathrm{I_2}(p),\underline{g}_0}$ form  free basis  (see the first row of Figure 3). We denote the classes by the same notation $\gamma_i$ since we shall use notation $[\gamma_i]$ for another meaning below.
\vspace{-0.2cm}
\begin{equation}
\label{eq:homology}
L \ := \ \mathrm{H}_1(E_{\mathrm{I_2}(p),\underline{g}_0},\mathbb{Z})\ = \ \mathbb{Z} \gamma_0 \ \oplus \ \bigoplus_{i=1}^{[p/2]} \ \mathbb{Z}\gamma_i.
\end{equation}
The open embedding $E_{\mathrm{I_2}(p),\underline{g}_0} \subset \overline{E}_{\mathrm{I_2}(p),\underline{g}_0}$ induces a surjective homomorphism 
\begin{equation}
\label{eq:compacthomology}
L \ \longrightarrow  \ \mathrm{H}_1(\overline{E}_{\mathrm{I_2}(p),\underline{g}_0},\mathbb{Z})
\end{equation}
whose kernel is the radical 
$$
\mathrm{rad}(L) \: = \{ \gamma\in L \mid \langle \gamma,\delta \rangle =0 ,\ \forall \delta \in L\}
$$
of the lattice $L$ which is additively generated by  homologous relations 
$$
\gamma_1\sim \cdots\sim \gamma_{[p/2]}.
$$
in the compactification $\overline{E}_{\mathrm{I_2}(p),\underline{g}}$.\ 
Hence, $L/\mathrm{rad}(L)$\! is a rank 2 free abelian group generated by the equivalence classes $[\gamma_0]$ and $[\gamma_1]=\cdots=[\gamma_{[p/2]}]$. 
\begin{equation}
\label{eq:compacthomology2}
L/\mathrm{rad}(L) \ = \ \mathbb{Z} [\gamma_1] \oplus \mathbb{Z} [\gamma_0].
\end{equation}
Here, we denote by $[*]$ the equivalence class of $*$ in the quotient module.  
The map  \eqref{eq:compacthomology} preserves the intersection form, since it is the quotient morphism by the radical.
}

\bigskip 
We next consider a $\mathrm{SL}_2$-linear automorphism of the $(x,y)$-plane by
\footnote{
There is an $\mathrm{SL}_2$-linear automorphism $\sigma$ of order 1, 2 and 3   on  bigger families of type $\mathrm{A_2}, A_3$ and $D_4$ (see Footnote 1). Then, the action \eqref{eq:autom} is induced from $\sigma$ as fiber-wise action on the subfamilies of type $\mathrm{A_2}$, $\mathrm{B_2}$ and $\mathrm{G_2}$, respectively. The $\sigma$-action has not only homological implications as discussed in this section, it has another important implication on a certain cohomology class called the primitive form (see Footnote 17).
}
\begin{equation}
\label{eq:autom}
\begin{array}{rcl}
\sigma_{\mathrm{A_2}}(x,y) & := & (x,y), \\ 
\sigma_{\mathrm{B_2}}(x,y) & := & (-x,-y), \\ 
\sigma_{\mathrm{G_2}}(x,y) & := & (\frac{y-x}{2},\frac{-3x-y}{2}).
\end{array}
\end{equation}
\medskip
It fixes the equation $F_{\mathrm{I_2}(p)}$ and, hence, induces a fiber automorphism of order $[p/2]$ of  the family \eqref{eq:family}, whose action fixes  the base space $S_{\mathrm{I_2}(p)}$ point-wise.  Since $\sigma_{\mathrm{I_2}(p)}$ leaves  the real structure $X_{\mathrm{I_2}(p)}^{\mathbb{R}}$ invariant, it acts on each complex and real curve $E_{\mathrm{I_2}(p),\underline{g}}$. One checks directly that $\sigma_{\mathrm{I_2}(p)}$-action induces the cyclic permutation of the cycles $\gamma_i$, oriented arcs $\delta_i$  and the points $\infty_i$ at infinity  for $i\in \mathbb{Z}/[p/2]\mathbb{Z}$, respectively. However, the cycle $\gamma_0$ is invariant by the $\sigma_{\mathrm{I_2}(p)}$-action. 

Let us consider the sub-lattice of $\mathrm{H}_1(E_{\mathrm{I_2}(p),\underline{g}_0},\mathbb{Z})$ consisting of $\sigma_{\mathrm{I_2}(p)}$-fixed elements.
\begin{equation}
\label{eq:invariant} \qquad 
L_{\mathrm{I_2}(p)}\ :=\ L^{\sigma_{\mathrm{I_2}(p)}} \ = \ \mathrm{H}_1(E_{\mathrm{I_2}(p),\underline{g}_0},\mathbb{Z})^{\sigma_{\mathrm{I_2}(p)}} 
\end{equation}
 It is immediate to see that  $L_{\mathrm{I_2}(p)}$ is a rank 2 sub-lattice generated by 
\begin{equation}
\label{eq:alpha-beta}
\begin{array}{rcl}
\alpha:=\sum_{i=1}^{[p/2]} \gamma_i   & \text{\quad  and \quad} & \beta:=\gamma_0,
\end{array}
\end{equation}
whose intersection number is counted by \eqref{eq:intersection1} as
\begin{equation}
\label{eq:intersection2}
\langle \alpha \ ,\  \beta \rangle \ =\ -[p/2],\quad \langle \alpha \ ,\  \alpha \rangle \ =\ 0 
\quad \text{and} \quad \langle \beta \ ,\  \beta \rangle \ =\ 0.
\end{equation}
The composition of the embedding $L_{\mathrm{I_2}(p)} \subset L$ with the radical quotient map  \eqref{eq:compacthomology} is again an isometric embedding of lattices:\footnote
{Since we shall no-longer use the lattice $L$ in the  present paper, so far as there is a no-confusion, we shall regard $ L_{\mathrm{I_2}(p)}$  as a sub-lattice of $L/\mathrm{rad}(L)$ (e.g.\ \eqref{eq:basechange}). 
}
\begin{equation}
\label{eq:sublattice}
  L_{\mathrm{I_2}(p)} \ \subset\ L/\mathrm{rad}(L), \ \ \alpha\mapsto [p/2][\gamma_1], \ \beta \mapsto [\gamma_0]
\end{equation}
of finite index [p/2].

\medskip
We give below a root lattice theoretic interpretation of what we have calculated above (see also Footnotes 8, 9 and 11), which answers to the question on the naming of the family \eqref{eq:equation}, posed in \S2. 

In the following Figure 3, we exhibit: 

\medskip
 (1) The first row exhibits a conceptual description of the cycles $\gamma_0$ and $\gamma_i$ ($i=1,\cdots,[p/2]$) in the surface $E_{\mathrm{I_2}(p),\underline{g}_0}$, which is a complexification of the real curves in the first row of Figure 2. Complexfication of real curves in the second and the third row of Figure 2. can be obtained from this surface $E_{\mathrm{I_2}(p),\underline{g}_0}$ by pinching either the cycles $\gamma_i$ ($i=1,\cdots,[p/2]$) or the cycle $\gamma_0$, respectively. 
 
 (2) The second row exhibits the intersection diagram of the basis $\gamma_0$ and $\gamma_i$ ($i=1,\cdots,[p/2]$) of the homology group $L:=\mathrm{H}_1(E_{\mathrm{I_2}(p),\underline{g}_0},\mathbb{Z})$ (actually, they are known as diagrams of types $\mathrm{A_2,  A_3}$ and $\mathrm{D_4}$, respectively (note Footnote 1 and 9)).
 
 (3) The third row exhibits the folding of the intersection diagram in previous  (2) by the action of the automorphism $\sigma_{\mathrm{I_2}(p)}$ which is the intersection diagram for the invariant basis $\alpha,\beta$ of the invariant homology group \eqref{eq:invariant} (actually, they are known as diagrams of types $\mathrm{A_2},  \mathrm{B_2}$ and $\mathrm{G_2}$, respectively).

\medskip
Consequently, we observe that the $\sigma_{\mathrm{I_2}(p)}$-invariant 1-cycles of the families \eqref{eq:equation} are indexed by the lattices of types  $\mathrm{A_2},  \mathrm{B_2}$ and $\mathrm{G_2}$, respectively. This is the reason why we want to call the families according to the type of the root systems. \footnote
{To be exact, what we wrote here needs more explanations in the following sense. The intersection form $\langle \cdot,\cdot\rangle$ on the free abelian group $L:=\sum_{i=0}^{[p/2]} \mathbb{Z} \gamma_i$ is skew-symmetric so that the pair $(L,\langle \cdot,\cdot\rangle)$ is not a root lattice. In order to justify what we wrote above, we consider a pair $(L,J)$ of the abelian group $L$ with the non-symmetric Seifert form  $J$ on it defined by the Seifert matrices 
\vspace{-0.1cm}
$\begin{bmatrix}
1 \!\!&\!\! -1\\
0 \!\! &\!\! 1 
\end{bmatrix}$
,
$\begin{bmatrix}
1\!\! &\!\! -1\!\! &\!\! -1\\
0 \!\!&\!\!1 \!\!&\!\! 0\\
0 \!\!&\!\! 0 \!\!&\!\! 1
\end{bmatrix}$
or
$\begin{bmatrix}
1 \!\!&\!\! -1 \!\!&\!\! -1\!\!&\!\! -1\\
0 &1 & 0 & 0\\
0 & 0 & 1 & 0\\
0 & 0& 0 & 1
\end{bmatrix}$
(which is the table of linking numbers between the ordered basis $\gamma_i$ \eqref{eq:homology}) according to $p=3,4$ or $6$, respectively. Then, the difference $J_{\mathrm{I_2}(p)} \! - ^{t}\!J_{\mathrm{I_2}(p)}$ is a skew symmetric form on $L$ which is identified with the intersection form $\langle \cdot,\cdot\rangle$ on $\mathrm{H}_1(E_{\mathrm{I_2}(p),\underline{g}_0},\mathbb{Z})$. On the other hand, the sum $I:=J+\ ^{t}J$ is a symmetric bilinear form so that the pair $(L, I)$ is isomorphic to the root lattice of type $\mathrm{A_2}, A_3$ and $D_4$ with the simple root basis $\gamma_0,\cdots,\gamma_{[p/2]}$. The geometric automorphism \eqref{eq:autom} induces an automorphism of the lattice $(L,J)$, denoted again by $\sigma_{\mathrm{I_2}(p)}$. Then the invariant sub-lattice $L_{\mathrm{I_2}(p)}:=L^{\sigma_{\mathrm{I_2}(p)}}$ is spanned by $\alpha:=\sum_{i=1}^{[p/2]}\gamma_i$ and $\beta:=\gamma_0$ so that the Seifert form $J_{\mathrm{I_2}(p)}$ on the basis gives the matrix 
$\begin{bmatrix}
[p/2]\!\!&\!\!\!-[p/2]\\0\!&\! 1
\end{bmatrix}
$. Then 
$(L_{\mathrm{I_2}(p)}, I|_{L_{\mathrm{I_2}(p)}})$ is isomorphic to the root lattice of  type $\mathrm{A_2}, \mathrm{B_2}$ and $\mathrm{G_2}$ with the simple root basis $\alpha$ and $\beta$ according as $p=3,4$ or $6$. 

There are two 
constructions which realize the above formal justification: (1) Consider the category of matrix factorization of the singularity $F_{\mathrm{I_2}(p)}(x,y,0)$ (see \cite{K-S-T}). Since the simple singularities are self-mirror, the category is isomorphic to the category of vanishing cycles in the Milnor fiber $F_{\mathrm{I_2}(p)}(x,y,\underline{g}_0)=0$. Actually, one finds strongly exceptional collections generating the category such that their images $\gamma_i$ in the $K$-group of the category ($\simeq$ the middle homology group\! of\! the\! Milnor\! fiber)\! gives\ the\ simple\! basis\! where\! the\! Euler\! form\! $ 
\sum_k\! (\!-\!1)^k \! Ext^k(\!\gamma_i,\!\gamma_j\!)$\! is identified with the Seifert form. 
(2) Consider the equation $F_{\mathrm{I_2}(p)}(x,y,\underline{g}_0)\!+\!z^2$ in three variables $(x,y,z)$. 
Then the suspensions, denoted by $\Sigma \gamma_i$, of the basis $\gamma_i$ ($0\le i\le [p/2]$)  of the vanishing cycles in the complex affine curve $F_{\mathrm{I_2}(p)}(x,y,\underline{g}_0)=0$ form vanishing cycle basis of the second homology group of the complex affine surface $F_{\mathrm{I_2}(p)}(x,y,\underline{g}_0)+z^2=0$. Then the intersection form $-\langle \Sigma \gamma_i, \Sigma \gamma_j\rangle$ of the homology classes coincides with the symmetric form $ I$.
}
The singularity at the origin $0\in \mathbb{C}^2$ of the curve $E_{\mathrm{I_2}(p),0}$ for $\underline{g}=0$ (exhibited in the 4th row of Figure 2) together with the action of  $\sigma_{\mathrm{I_2}(p)}$ on it is called the singularity of type $\mathrm{I_2}(p)$ (i.e.\ of type $\mathrm{A_2}$, $\mathrm{B_2}$ and $\mathrm{G_2}$ according as $p=3$, $4$ and $6$) and the lattice $L_{\mathrm{I_2}(p)}$ is called the lattice of vanishing cycles for the singularity. Then, the family \eqref{eq:family} may be regarded also as the universal unfolding (by the parameter space $S_{\mathrm{I_2}(p)}$) of the singularity of type $\mathrm{I_2}(p)$.

\bigskip

\hspace{-0.4cm}
\vspace{-0.2cm}
{
\begin{tabular}{lll}
$E_{A_2,\underline{g}_0}$&$E_{B_2,\underline{g}_0}$&$E_{G_2,\underline{g}_0}$\\
\includegraphics[width=3.5cm]{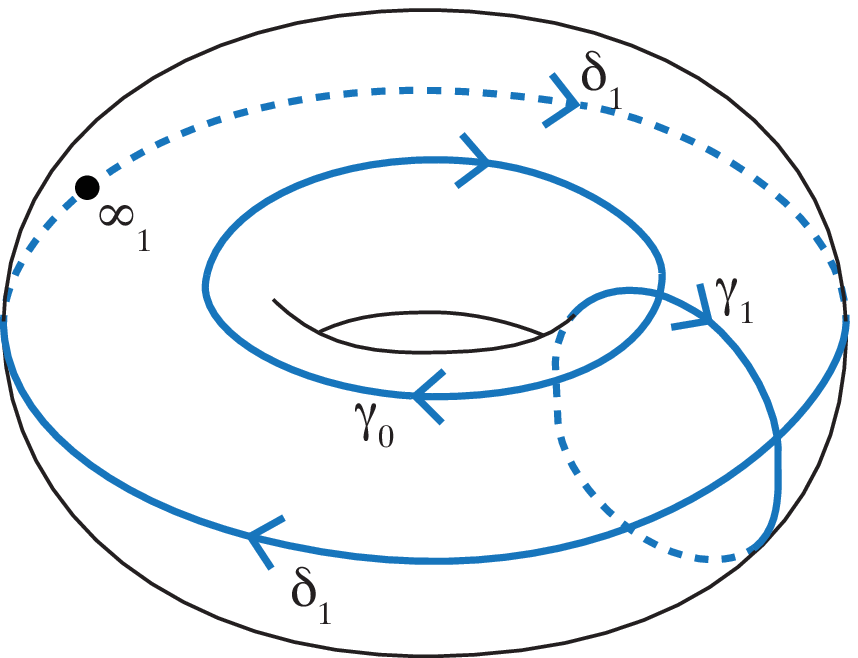}
&\hspace{2mm}\includegraphics[width=3.5cm]{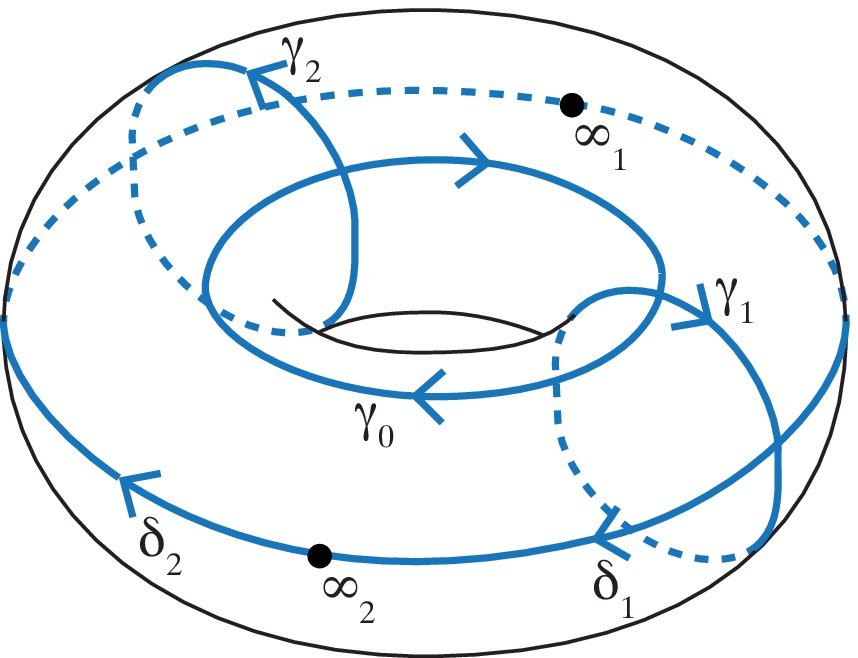}
&\hspace{2mm}\includegraphics[width=3.5cm]{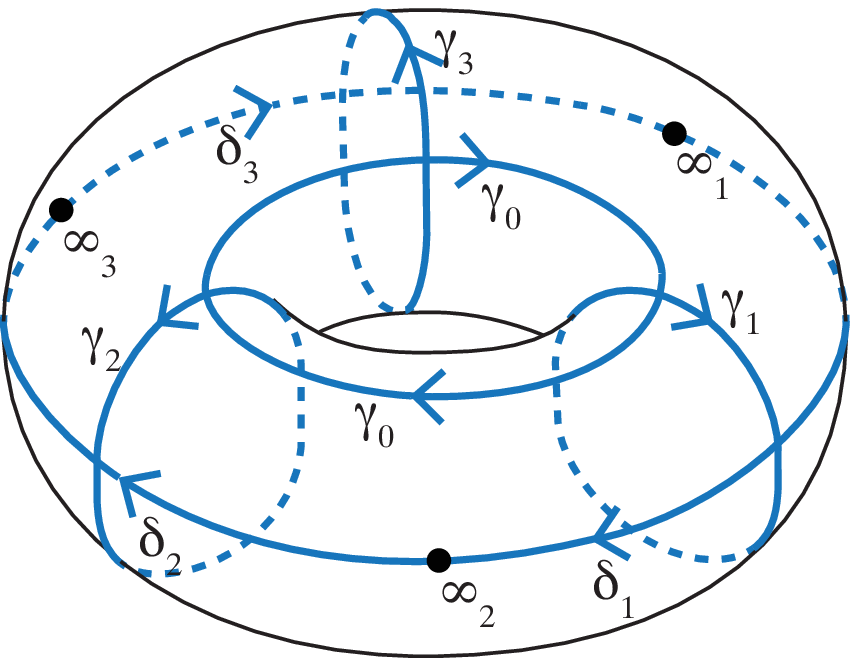}\\
&&
\end{tabular}

\setlength{\unitlength}{0.00083333in}
\begingroup\makeatletter\ifx\SetFigFont\undefined%
\gdef\SetFigFont#1#2#3#4#5{%
  \reset@font\fontsize{#1}{#2pt}%
  \fontfamily{#3}\fontseries{#4}\fontshape{#5}%
  \selectfont}%
\fi\endgroup%
\renewcommand{\dashlinestretch}{30}
\begin{picture}(5446,2200)(-100,-10)
\put(1118,1737){\ellipse{78}{78}}
\put(368,1737){\ellipse{78}{78}}
\path(406,1729)(1081,1729)
\put(1118,462){\ellipse{78}{78}}
\put(368,462){\ellipse{78}{78}}
\path(406,454)(1081,454)
\put(3068,462){\ellipse{78}{78}}
\put(2318,462){\ellipse{78}{78}}
\path(2352,431)(3035,431)
\path(2352,488)(3027,488)
\path(2659,469)(2734,544)
\path(2664,466)(2739,391)

\path(706,1204)(706,904)
\path(676.000,1024.000)(706.000,904.000)(736.000,1024.000)
\path(2656,1204)(2656,904)
\path(2626.000,1024.000)(2656.000,904.000)(2686.000,1024.000)
\path(4681,1204)(4681,904)
\path(4651.000,1024.000)(4681.000,904.000)(4711.000,1024.000)
\put(751,978){\makebox(0,0)[lb]{\smash{{{\SetFigFont{14}{16.8}{\familydefault}{\mddefault}{\updefault}/}}}}}
\put(828,978){\makebox(0,0)[lb]{\smash{{{\SetFigFont{10}{12.0}{\familydefault}{\mddefault}{\updefault}${(\mathbb Z}/1{\mathbb Z})$}}}}}
\put(2711,978){\makebox(0,0)[lb]{\smash{{{\SetFigFont{14}{16.8}{\familydefault}{\mddefault}{\updefault}/}}}}}
\put(2787,978){\makebox(0,0)[lb]{\smash{{{\SetFigFont{10}{12.0}{\familydefault}{\mddefault}{\updefault}${(\mathbb Z}/2{\mathbb Z})$}}}}}
\put(4761,978){\makebox(0,0)[lb]{\smash{{{\SetFigFont{14}{16.8}{\familydefault}{\mddefault}{\updefault}/}}}}}
\put(4831,978){\makebox(0,0)[lb]{\smash{{{\SetFigFont{10}{12.0}{\familydefault}{\mddefault}{\updefault}${(\mathbb Z}/3{\mathbb Z})$}}}}}

\put(5095,457){\ellipse{78}{78}}
\put(4345,457){\ellipse{78}{78}}
\path(4383,454)(5058,454)
\path(4373,484)(5068,484)
\path(4363,424)(5078,424)
\path(4683,449)(4758,524)
\path(4688,454)(4763,379)

\put(2243,1737){\ellipse{78}{78}}
\put(4268,1737){\ellipse{78}{78}}
\put(2993,1962){\ellipse{78}{78}}
\put(2993,1512){\ellipse{78}{78}}
\put(5018,2037){\ellipse{78}{78}}
\put(5018,1737){\ellipse{78}{78}}
\put(5018,1422){\ellipse{78}{78}}
\path(2281,1729)(2956,1954)
\path(2281,1729)(2956,1504)
\path(4306,1729)(4981,2029)
\path(4298,1724)(4973,1424)
\path(4306,1729)(4981,1729)
 
\put(3081,1504){\makebox(0,0)[lb]{\smash{{{\SetFigFont{10}{12.0}{\familydefault}{\mddefault}{\updefault}$\gamma_{ _2}$}}}}}
\put(3081,1954){\makebox(0,0)[lb]{\smash{{{\SetFigFont{10}{12.0}{\familydefault}{\mddefault}{\updefault}$\gamma_{ _1}$}}}}}
\put(5086,2029){\makebox(0,0)[lb]{\smash{{{\SetFigFont{10}{12.0}{\familydefault}{\mddefault}{\updefault}$\gamma_{ _1}$}}}}}
\put(5086,1729){\makebox(0,0)[lb]{\smash{{{\SetFigFont{10}{12.0}{\familydefault}{\mddefault}{\updefault}$\gamma_{ _2}$}}}}}
\put(331,1829){\makebox(0,0)[lb]{\smash{{{\SetFigFont{10}{12.0}{\familydefault}{\mddefault}{\updefault}$\gamma_{ _0}$}}}}}
\put(1081,1829){\makebox(0,0)[lb]{\smash{{{\SetFigFont{10}{12.0}{\familydefault}{\mddefault}{\updefault}$\gamma_{ _1}$}}}}}
\put(2206,1829){\makebox(0,0)[lb]{\smash{{{\SetFigFont{10}{12.0}{\familydefault}{\mddefault}{\updefault}$\gamma_{ _0}$}}}}}
\put(4231,1829){\makebox(0,0)[lb]{\smash{{{\SetFigFont{10}{12.0}{\familydefault}{\mddefault}{\updefault}$\gamma_{ _0}$}}}}}
\put(5081,1429){\makebox(0,0)[lb]{\smash{{{\SetFigFont{10}{12.0}{\familydefault}{\mddefault}{\updefault}$\gamma_{ _3}$}}}}}
\put(1771,425){\makebox(0,0)[lb]{\smash{{{\SetFigFont{12}{14.4}{\familydefault}{\mddefault}{\updefault}$B_2:$}}}}}
\put(3800,425){\makebox(0,0)[lb]{\smash{{{\SetFigFont{12}{14.4}{\familydefault}{\mddefault}{\updefault}$G_2:$}}}}}
\put(-100,425){\makebox(0,0)[lb]{\smash{{{\SetFigFont{12}{14.4}{\familydefault}{\mddefault}{\updefault}$A_2:$}}}}}
\put(-100,1700){\makebox(0,0)[lb]{\smash{{{\SetFigFont{12}{14.4}{\familydefault}{\mddefault}{\updefault}$A_2:$}}}}}
\put(1771,1700){\makebox(0,0)[lb]{\smash{{{\SetFigFont{12}{14.4}{\familydefault}{\mddefault}{\updefault}$A_3:$}}}}}
\put(3800,1700){\makebox(0,0)[lb]{\smash{{{\SetFigFont{12}{14.4}{\familydefault}{\mddefault}{\updefault}$D_4:$}}}}}
\put(331,600){\makebox(0,0)[lb]{\smash{{{\SetFigFont{10}{12.0}{\familydefault}{\mddefault}{\updefault}$\beta$}}}}}
\put(1081,600){\makebox(0,0)[lb]{\smash{{{\SetFigFont{10}{12.0}{\familydefault}{\mddefault}{\updefault}$\alpha$}}}}}
\put(2256,600){\makebox(0,0)[lb]{\smash{{{\SetFigFont{10}{12.0}{\familydefault}{\mddefault}{\updefault}$\beta$}}}}}
\put(3056,600){\makebox(0,0)[lb]{\smash{{{\SetFigFont{10}{12.0}{\familydefault}{\mddefault}{\updefault}$\alpha$}}}}}
\put(4281,600){\makebox(0,0)[lb]{\smash{{{\SetFigFont{10}{12.0}{\familydefault}{\mddefault}{\updefault}$\beta$}}}}}
\put(5031,600){\makebox(0,0)[lb]{\smash{{{\SetFigFont{10}{12.0}{\familydefault}{\mddefault}{\updefault}$\alpha$}}}}}

\end{picture}

}

\centerline{ {\bf Figure 3.} \ Vanishing cycles and associated  diagrams.}

\bigskip
\section{Fundamental group and monodromy representation}
We describe the fundamental group of the compliment of discriminant loci $S_{\mathrm{I_2}(p)}\setminus D_{\mathrm{I_2}(p)}$, and its monodromy representation in the first homology groups of $E_{\mathrm{I_2}(p),\underline{g}}$ and that of  its compactification $\overline{E}_{\mathrm{I_2}(p),\underline{g}}$.

First,  we sketch a geometric description of the fundamental group which is valid not only for types $\mathrm{A_2}, \mathrm{B_2}$ and $\mathrm{G_2}$ but also for the regular orbit space of any type finite reflection group by the authors \cite{Br},\cite{D},\cite{S3}. Since the formulation of \cite{S3} is closest to the present setting, we briefly explain it.

Consider a translation action $\tau_\varepsilon: S_{\mathrm{I_2}(p)} \to S_{\mathrm{I_2}(p)}$, $(g_s,g_l)\mapsto (g_s,g_l+\varepsilon)$ for $\varepsilon \in \mathbb{C}$.\footnote
{The action $\tau$ is naturally obtained by integrating the vector field $\partial_{g_l}$, called the primitive vector field, which is, up to constant factor well-defined (see \cite{S3}).
}
If $\varepsilon \in \mathbb{R}$, then the action preserves the real space $S_{\mathrm{I_2}(p)}^{\mathbb{R}}$. In Figure 1.,  we consider two shifted discriminant loci $\tau_{ \varepsilon}(D_{\mathrm{I_2}(p)})$ and $\tau_{ -\varepsilon}(D_{\mathrm{I_2}(p)})$ for some $\varepsilon\in \mathbb{R}_{>0}$. 
Then, 
the darkly shaded area $J$ in $\Gamma_{\mathrm{I_2}(p)}\subset S_{\mathrm{I_2}(p)}^{\mathbb{R}}$ exhibit the component  (which is homeomorphic to a rhombus) in $\Gamma_{\mathrm{I_2}(p)}$ cut by the shifted discriminant loci $\tau_{ \varepsilon}(D_{\mathrm{I_2}(p)})$ and $\tau_{ -\varepsilon}(D_{\mathrm{I_2}(p)})$, where its boundary edges $a$ and $b$ are segments on $\tau_{ \varepsilon}(D_{\mathrm{I_2}(p)})$ and $\tau_{ -\varepsilon}(D_{\mathrm{I_2}(p)})$, respectively.

Let $\tilde a$ and $\tilde b$ be the path in $S_{\mathrm{I_2}(p)}\setminus D_{\mathrm{I_2}(p)}$ defined as follows.
We choose the base point $\underline{g}_0\in \Gamma_{\mathrm{I_2}(p)}$ at the intersecting point of $a$ and $b$  as in Figure 1.

Let $a_{\mathbb{C}}$ and $b_\mathbb{C}$ be complexfication  of the real segments $a$ and $b$ embedded in the complex shifted discriminants $\tau_\varepsilon D_{\mathrm{I_2}(p)}$ and $\tau_{-\varepsilon} D_{\mathrm{I_2}(p)}$, respectively, such that $a_{\mathbb{C}}\cap D_{\mathrm{I_2}(p)}=\{\underline{g}_+\}$ and 
$b_{\mathbb{C}}\cap D_{\mathrm{I_2}(p)}=\{\underline{g}_-\}$. The following Figure 4.\ shows how to choose paths $\tilde a$ and $\tilde b$ in the complexification  $a_{\mathbb{C}}$ and in $b_{\mathbb{C}}$, and then embed them in $S_{\mathrm{I_2}(p)}\setminus D_{\mathrm{I_2}(p)}$.

\newpage
 \begin{figure}[h]
 \center
  \hspace{-0.3cm}\includegraphics[width=11.8cm]{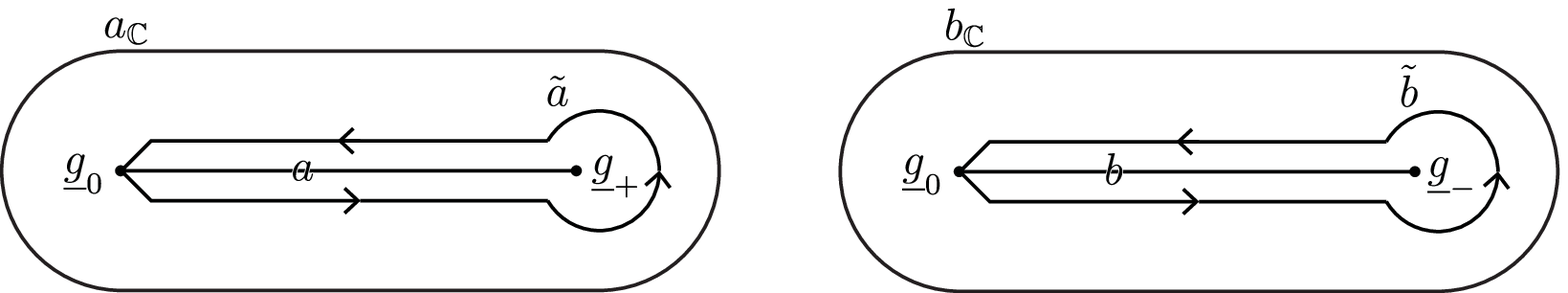}
  \vspace{0.2cm}
 \end{figure}
\centerline{ {\bf Figure 4.} {Homotopy classes $\tilde a$ and $\tilde b$}}

\bigskip
\bigskip
\noindent
{\bf Fact 7.}  {\it The fundamental group $\pi_1(S_{\mathrm{I_2}(p)}\setminus D_{\mathrm{I_2}(p)}, \underline{g}_0) =: G_{\mathrm{I_2}(p)}$ is generated by the homotopy classes of $\tilde a$ and $\tilde b$ and determined by the relation 
\begin{equation}
\label{eq:relation}
\tilde{a}\tilde{b}\cdots \ =\ \tilde{b}\tilde{a}\cdots 
\end{equation}
where the both hand sides are words of alternating sequences of $\tilde a,\tilde b$ of length $p$, which start either with $\tilde a$ or with $\tilde b$.
}

\medskip
\noindent
{\it Sketch of proof.} 1. A direct proof is given by Zariski-van Kampen method \cite{Z} w.r.t.\ the pencils \{$g_s\!=$const.\}, which intersects with discriminant loci by two points, giving two generators of the fundamental group easily identified  with $\tilde a, \tilde b$ (\cite{S3}\S4.3) and gives the relation \eqref{eq:relation} by turning $g_s=e^{i\theta}$ ($0\le \theta\le 2\pi$).

Let us sketch another approach in \cite{Br},\cite{D},\cite{S3}, where we regard the discriminant loci $D_{\mathrm{I_2}(p)}$ as the discriminant loci for the orbit space for the dihedral group action on a two dimensional vector space. Namely, let $V_{\mathbb{R}}$ be the real 2-space on which the dihedral group $W(\mathrm{I_2}(p))$ acts as the reflection group. We suitably (up to constant) identify the invariant ring $S(V_{\mathbb{R}}^*)^{W(\mathrm{I_2}(p))}$ with the polynomial ring generated by two variables $g_s$ and $g_l$ (recall \eqref{eq:equation}). So, we do $W(\mathrm{I_2}(p)\diagdown V_{\mathbb{C}}$ with $S_{\mathrm{I_2}(p)}$ and 
$(W(\mathrm{I_2}(p))\diagdown V_{\mathbb{C}})_{reg}$ $=W(\mathrm{I_2}(p))\diagdown (V_{\mathbb{C}}\setminus \cup H_\al)$ (=regular orbit space) with $S_{\mathrm{I_2}(p)}\setminus D_{I_2}$ (where $\cup H_\al$ is the union of complexified reflection hyperplanes). By this identification, the real region $\Gamma_{\mathrm{I_2}(p)}$ \eqref{eq:totallyreal} is homeomorphic to any chamber in $V_{\mathbb{R}}\setminus \cup H_\al$. Then, the inverse image in $V_{\mathbb{R}}$ of the closure of $*$ is a curved polygon which is dual to the chamber decomposition, namely, hexagon, octagon and dodecagon according as $\mathrm{I_2}(p)$ is of type $\mathrm{A_2}$, $\mathrm{B_2}$ and $\mathrm{G_2}$. The boundary of the polygon is given by $(\dot{a}\dot{b})^{p}$ (see \cite{S3} Fig.). 
Here, each $\dot{a}$ (resp.\ $\dot{b}$) is a double cover of the closed interval $\overline{a}$ (resp.\ $\overline{b}$) which is crossing the reflection hyperplane at its central point. We now consider the inverse image in $V_{\mathbb{C}}\setminus \cup H_\al$ of the union $\tilde{a}\cup\tilde{b}$. Actually, the inverse image is free homotpic to the boundary to the curved dual polygon shifted by $\sqrt{-1}\delta$ for some $\delta\in V_{\mathbb{R}}\setminus \cup H_\al$ (corresponding to breadth of $\tilde a$ and $\tilde b$, Figure 5.), where each component of the inverse image of $\tilde{a}$ (resp.\ $\tilde b$) is homotopic to a shifted edge {\small $\dot{a}+\sqrt{-1}\delta$} (resp.\  {\small $\dot{b}+\sqrt{-1}\delta$)}.  On the other hand, the shifted real vector space $V_{\mathbb{R}}+\sqrt{-1}\delta$ does not  intersect with any reflection hyperplane so that the shifted dual polygon is contractible in $V_{\mathbb{C}}\setminus \cup H_\al$. This creates the homotopy relation \eqref{eq:relation}. That there is no-more relations follows from a dimension argument, which we omit here.   \qquad $\Box$

The group $G_{\mathrm{I_2}(p)}$ with the presentation in Fact 7 is called the Artin group of type $\mathrm{I_2}(p)$ (\cite{B-S}). The element of the expression \eqref{eq:relation} is the least common multiple of the generators  $\tilde{a}$ and $\tilde{b}$ inside the positive monoid of the group and is called the fundamental element (\cite{G},\cite{B-S}), denoted by \footnote
{Here, we have an unfortunate coincidence of notation of the fundamental element  with that of the discriminant \eqref{eq:discriminant}. Since we use them in different places,  there shall be no confusions.
}
$$
\Delta \ = \ \Delta_{\mathrm{I_2}(p)}.
$$
 It is well-known (e.g.\ \cite{B-S}) that the center of the group $G_{\mathrm{I_2}(p)}$ is an infinite cyclic group generated by $\Delta$ (types $\mathrm{B_2}$ and $\mathrm{G_2}$) or by $\Delta^2$ (type $\mathrm{A_2}$).

\medskip
The next task is to determine the monodromy actions of the fundamental group $G_{\mathrm{I_2}(p)}$ on the lattices $L,\ L_{\mathrm{I_2}(p)}$ and $L/\mathrm{rad}(L)$.  Recall Fact 2, 3, 4 and 5 that the ``degeneration" of the curve $E_{\mathrm{I_2}(p),\underline{g}_0}$ moving the parameter $\underline{g}_0$ along the path $a$ (resp.\ $b$) pinches the cycles $\gamma_i$ {($i=1,\cdots,[p/2]$)}
 (resp.\ cycle $\gamma_0$) in $E_{\mathrm{I_2}(p),\underline{g}_0}$ to points $p_i$ (resp.\ $p_0$).  Then the actions $\rho=\rho_{\mathrm{I_2}(p)}$ of $\tilde a$ (resp.\ $\tilde b$) on the lattice $L=\mathrm{H}_1(E_{\mathrm{I_2}(p),\underline{g}_0},\mathbb{Z})$ is determined by Picard-Lefschetz formula (i.e. the transvections) of the vanishing cycles $\gamma_1,\cdots,\gamma_{[p/2]}$ (resp.\ $\gamma_0$), taking the intersection number \eqref{eq:intersection1} in account. That is,
$$
\rho(\tilde{a})(\gamma_j)=
\begin{cases}
\gamma_0 -\sum_{i=1}^{[p/2]}\gamma_i \text{\quad if } j=0\\
\gamma_j   \text{\qquad \qquad \ if } j=1,2,3
\end{cases}
\rho(\tilde{b})(\gamma_j)=
\begin{cases}
\gamma_0 \text{\qquad \qquad \ if } j=0\\
\gamma_j  +\gamma_0 \text{\ if } j=1,2,3
\end{cases}
$$
Then, the other actions on $ L_{\mathrm{I_2}(p)}$ and $L/\mathrm{rad}(L)$, denoted by the same $\rho$, are induced from this action  either by restriction to the sub-lattice or by  the quotient out the radical of the  lattice. 
In particular, the embedding \eqref{eq:sublattice} is equivariant with the monodromy action.
Explicit formulae are given as follows.

\medskip
\noindent
{\bf Fact 8.} 1)  {\it The action of $G_{\mathrm{I_2}(p)}$ on 
$L_{\mathrm{I_2}(p)}=\mathrm{H}_1(E_{\mathrm{I_2}(p),\underline{g}_0},\mathbb{Z})^{\sigma_{\mathrm{I_2}(p)}}$ is given as follows.} 
{\small
\begin{equation}
\label{eq:monodromy1}
\rho(\tilde{a} )
(\alpha,\beta)
=
(\alpha,\beta)
\begin{bmatrix}
1 & -1 \\
0 & 1
\end{bmatrix}
 \text{\quad and \quad}  
 \rho(\tilde{b}) (\alpha,\beta)
 =
 (\alpha,\beta) 
 \begin{bmatrix}
 1 & 0\\
 [p/2] & 1
 \end{bmatrix}.
\end{equation}
}

 2)  {\it The action of $G_{\mathrm{I_2}(p)}$ on 
 $L/\mathrm{rad}(L)=\mathrm{H}_1(\overline{E}_{\mathrm{I_2}(p),\underline{g}_0},\mathbb{Z})$ is given as follows. }
{\small
\begin{equation}
\label{eq:monodromy2}
\rho(\tilde{a} )
([\gamma_0],[\gamma_1])
\! = \!
([\gamma_0],[\gamma_1])
\begin{bmatrix}
1 \!\! & \! 0 \\
-[p/2] \!\! & \! 1
\end{bmatrix}
 \text{\ and \ }  
 \rho(\tilde{b}) ([\gamma_0],[\gamma_1])
\! = \!
 ([\gamma_0],[\gamma_1]) 
 \begin{bmatrix}
 1 \!\! & \! 1\\
 0 \!\! & \! 1
 \end{bmatrix}.\! \!\!\!
\end{equation}
}

\medskip
\noindent
Here, the representations \eqref{eq:monodromy1} and \eqref{eq:monodromy2} are conjugate by the basis change:
 \begin{equation}
 \label{eq:basechange}
 (\alpha,\beta)=([\gamma_0],[\gamma_1])
 \begin{bmatrix}
 0 & 1\\
 [p/2] & 0
 \end{bmatrix}
 \end{equation}
 \eqref{eq:sublattice}, and, hence, they are equivalent.\footnote
 {In the sequel,  we shall treat these two equivalent representations in parallel (e.g.\  \eqref{eq: localsystem}), which looks a bit redundant and cumbersome. This subtlety was caused since we, later on, want to study automorphic forms for congruence subgroups (see \eqref{eq:congruence1} and \S6). More precisely, from a view point for the period map of a primitive form, it is natural to consider the lattice $L_{\mathrm{I_2}(p)}$ of vanishing cycles. On the other hand, from a view point of classical elliptic integrals for a compact elliptic curve and to connect with classical elliptic modular function theory, it is natural to consider the lattice 
 $L/\mathrm{rad}(L)$.  
 Therefore, in the sequel, when we talk about the matrix expression of the representation $\rho$, we shall mean those matrices with respect to the basis $[\gamma_0]$ and $[\gamma_1]$ but not the basis $\alpha$ and $\beta$. 
 These cautious treatments are necessary for the cases of types  $\mathrm{B_2}$ and $\mathrm{G_2}$ but not for $\mathrm{A_2}$, since the embedding  \eqref{eq:sublattice} is already isomorphic.
 }
 We determine the image and the kernel of the representations as follows.
 
\medskip
\noindent
{\bf Fact 9.}  
{\it Let us identify $L/\mathrm{rad}(L)$ with $\mathbb{Z}^2$ by the use of the basis $[\gamma_0]$ and $[\gamma_1]$, and regard the representation \eqref{eq:monodromy2} as  a homomorphism $\rho : G_{\mathrm{I_2}(p)}\to \mathrm{SL}_2(\mathbb{Z})$.  Then, we have

1. The fundamental element $\Delta$ is represented by $\rho$ as follows.
\begin{equation}
\label{eq:fundamental}
\rho_{\mathrm{A_2}}(\Delta)=
\begin{bmatrix}
 0 & 1\\
 -1 & 0
 \end{bmatrix}  ,\quad
\rho_{\mathrm{B_2}}(\Delta)=
\begin{bmatrix}
 -1 & 0  \\
 0 & -1
 \end{bmatrix}  \quad {\text and} \quad 
 \rho_{\mathrm{G_2}}(\Delta)=
\begin{bmatrix}
 1 & 0\\
 0 & 1
 \end{bmatrix}  .
\end{equation}
In particular, the images  $Im(\rho_{\mathrm{A_2}})$ and $Im(\rho_{\mathrm{B_2}})$ contain $-id$, but $Im(\rho_{\mathrm{G_2}})$ does not.

2. The image of the representation $\rho$ in $\rm{SL}_2(\mathbb{Z})$  is equal to the subgroup 
\begin{equation}
\label{eq:congruence1}
\Gamma_1([p/2]) :=
\Big\{ 
\begin{bmatrix}
 a & b\\
 c & d
 \end{bmatrix} 
 \in \mathrm{SL}_2(\mathbb{Z}) \mid
\begin{bmatrix}
 a & b\\
 c & d
 \end{bmatrix} 
  \equiv 
  \begin{bmatrix}
 1 & b\\
 0 & 1
 \end{bmatrix} 
  \bmod [p/2] \}\Big\} ,
  \end{equation}
 called a congruence subgroup of $\mathrm{SL}_2(\mathbb{Z})$ of level $[p/2]$ (see e.g.\ \cite{Ko}):
  
3. The kernel of the representation $\rho$ is an infinite cyclic group generated by
\begin{equation}
\label{eq:center}
(\tilde{a}\tilde{b})^{k(\mathrm{I_2}(p))}
\end{equation}
 where $k(\mathrm{I_2}(p))\in \mathbb{Z}_{\ge2}$ is a number  attached to 
any Coxeter system (see \cite{S5} \S6, 6.1 ii)) such that  
\begin{equation}
\label{eq:k}
k(\mathrm{A_2})=6, \quad k(\mathrm{B_2})=4  \ \text{ and } \  k(\mathrm{G_2})=3. 
\end{equation}
\footnote{
The  numbers  $k(W)$ for any finite Weyl group $W$ is defined in \cite{S5} 6.1 depending  on Coxeter diagram of $W$. }

4. Thus, we have the short exact sequence
\begin{equation}
\label{eq:congruence2} 
1 \ \rightarrow \ \mathbb{Z} \ \rightarrow \ G_{\mathrm{I_2}(p)} \ \rightarrow \ \Gamma_1([p/2]) \ \rightarrow \ 1.
\end{equation}
That is, the Artin group $G_{\mathrm{I_2}(p)}$ is a central extension of an elliptic congruence modular group $\Gamma_1([p/2])$.
}

\begin{proof}
1. This is a direct calculation for the cases $[p/2] =1,2$ and $3$. The fact that $Im(\rho_{\mathrm{A_2}})$
does not contain $-id$ is a consequence of ii).

2. Let us show that $\rho$ is surjective to  $\Gamma_1([p/2])$. Set $A:=\rho(\tilde a)$ and $B:=\rho(\tilde b)$ in \eqref{eq:monodromy2}. Clearly $A,B\in \Gamma_1([p/2])$. We show that $\Gamma_1([p/2])$ is generated by $A$ and $B$. Consider an element $C\in \Gamma_1([p/2])$ whose $(2,1)$ (resp.\ (2,2)) entry is $c$ (resp.\ $d$). By definition, we set $c=\overline{c}[p/2]$ for some $\overline{c}\in \mathbb{Z}$. We also know by definition that $d\not=0$. Then the (2,1) entry (resp.\  (2,2) entry) of  $CA^k$ ($k\in\mathbb{Z}$) is equal to $c+kd[p/2]=(\overline{c}+kd)[p/2]$ (resp.\ is unchanged $d$). Then, by the Euclidean division algorithm, we can choose $k\in \mathbb{Z}$ such that $|\overline{c}-kd|\le |d|/2$. Next, let us consider $C\in  \Gamma_1([p/2])$ such that its (2,1) entry $c=\overline{c}[p/2]$ satisfies the condition $|\overline c |\le |d|/2$. If $c=0$ then the diagonal of $C$ is $\pm(1,1)$ where $(-1,-1) $ cannot occur for the case $[p/2]=3$ by the definition of $\Gamma_1(3)$. By multiplying $-id$ if necessary for the cases $[p/2]=1$ or $2$ (recall the result in 1.), we see that $C$ is already of the form $B^l$ for some $l\in \mathbb{Z}$. Suppose $c\not=0$. We consider $CB^l$ for $l\in\mathbb{Z}$. Then, its (2,1)-entry is unchanged $c$, but the (2,2) entry is given by $d':=d-lc$. Then  after a suitable choice of $l\in\mathbb{Z}$, we have $|d-lc|\le |c|/2$.  Combining the above two procedures, the (2,2) entry $d'$ of the matrix $CA^kB^l$ satisfies $|d'|\le |c|/2=|\overline{c}[p/2]|/2\le |d|[p/2]/4$. Since $[p/2]/4<1$ in our case, this means $|d'|<|d|$. That is, if $C$ is note generated by $A$ and $B$, there are integers $k$ and $l$ such that the (2,2) entry of 
$CA^kB^l$ has the absolute value strictly smaller than that of $C$. This give an induction proof of $ \Gamma_1([p/2])=\langle A,B\rangle$.

3.  As a result of 1., we observe
$\rho_{\mathrm{A_2}}(\Delta^4)$, $\rho_{\mathrm{G_2}}(\Delta^2)$ and $\rho_{\mathrm{G_2}}(\Delta)$ are identity matrices.  That is, $\Delta_{\mathrm{A_2}}^4$, $\Delta_{\mathrm{B_2}}^2$ and $\Delta_{\mathrm{G_2}}$ belong to the kernel, 
where we have relations $\Delta_{\mathrm{A_2}}^4 \! = \! (\tilde{a}\tilde{b})^{6}$, $\Delta_{\mathrm{B_2}}^2 \! = \! (\tilde{a}\tilde{b})^{4}$ and  $\Delta_{\mathrm{G_2}} \! = \! (\tilde{a}\tilde{b})^{3}$. Thus the element \eqref{eq:center} is contained in the $\ker \rho$.
Conversely, let us show that $\ker(\rho_{\mathrm{I_2}(p)}) \! \subset$center of $G_{\mathrm{I_2}(p)}$.  Then, we explicitly calculate $\Delta_{\mathrm{A_2}}^4$, $\Delta_{\mathrm{B_2}}^2$ and  $\Delta_{\mathrm{G_2}}$ generate the kernel. To show this, we use the fundamental domain of $\Gamma_1([p/2])$ (see \cite{S5} \S6 Assertion 5, 6.).

4.  This is only the rewriting of the results 2. and 3. 
\end{proof}
 
\begin{remark}
The following characterization of the congruence subgruop is well-known.
$$
\begin{array}{rl}
\!\!\! \Gamma_1([p/2]) \!\!\! & \!\! = \! \{ m\in  \mathrm{SL}_2(\mathbb{Z}) \! \mid   m \text{ preserves the subsets  $L /\rad(L) \! + \! \frac{1}{[p/2]} [\gamma_0]$}\}\\
 = & \!\! \{ m\in  \mathrm{SL}_2(\mathbb{Z}) \! \mid   m \text{ preserves the subsets $L_{\mathrm{I_2}(p)}$ and $L_{\mathrm{I_2}(p)} \! + \! [\gamma_1]$}\}\\
= &\!\!\! \{ \! m \! \in \!  \mathrm{SL}(L_{\mathrm{I_2}(p)})\! \mid \!  m \text{ extends to  $\mathrm{SL}_2(\mathbb{Z})$\! and preserves\! $L_{\mathrm{I_2}(p)}\! +\! [\gamma_1]$}\!\}.
\end{array}
$$
{\it A sketch of proof.} Reflections by the roots $\alpha$ and $\beta$, satisfy the conditions. Conversely, the conditions on $m =
\begin{bmatrix}
 a \!&\! b\\
 c \!&\! d
 \end{bmatrix} 
\in  \mathrm{SL}_2(\mathbb{Z})$ that it preserves $L$ (resp.\  $L_{\mathrm{I_2}(p)} + [\gamma_1]$) implies $c\equiv 0 \bmod [p/2]$ (resp.\ $d\equiv 1 \bmod [p/2]$).  $\Box$ 
\end{remark}

As a consequence of the description of the congruence subgroup $\Gamma_1([p/2])$ in Fact 9, we  introduce a character (see \cite{S5} \S6 (6.1.6)) on it, which shall be used to formulate the discriminant conjecture in \S11: 
\begin{equation}
\label{eq:character}
\vartheta_{\mathrm{I_2}(p)}:  \Gamma_1([p/2]) \to  \mathbb{C}^\times, \quad \tilde{a}, \tilde{b} \mapsto \exp{\Big(\frac{\pi\sqrt{-1}}{k(\mathrm{I_2}(p))}\Big)} .
\end{equation}
{\it Proof.} Obviously, the relation \eqref{eq:relation} is satisfied by the images of $\vartheta_{\mathrm{I_2}(p)}$. {\small On the other hand, the $\vartheta_{\mathrm{I_2}(p)}$-image of  \eqref{eq:center} is equal to  $\exp{\!(2\pi\sqrt{\!-1})}\!=\!1$}. $\Box$ 

\medskip
Note that the power 
\begin{equation}
\label{eq:anti-invariant}
\theta_{\mathrm{I_2}(p)}:=\vartheta_{\mathrm{I_2}(p)}^{k(\mathrm{I_2}(p))}:
 \Gamma_1([p/2]) \to \{\pm1\}, \ \ \tilde{a},\tilde{b} \mapsto -1
\end{equation}
 defines also a character for the anti-invariants, and $\vartheta_{\mathrm{I_2}(p)}^{2k(\mathrm{I_2}(p))}$ is  trivial.
Actually except for the type $\mathrm{G_2}$, $\theta_{\mathrm{I_2}(p)}$ factor through the sign morphism $W_{\mathrm{I_2}(p)} \to \{\pm1\}$ of the Weyl group associated with ${\mathrm{I_2}(p)}$.

\bigskip

For a use to describe the period map in the next \S5, we prepare some notations. 
It is well-known that monodromy data given in {\bf Fact 9.} is equivalent to data of the local systems of the homology groups $\mathrm{H}_1(E_{\mathrm{I_2}(p),\underline{g}},\mathbb{Z})^{\sigma_{\mathrm{I_2}(p)}}$  
and 
$\mathrm{H}_1(\overline{E}_{\mathrm{I_2}(p),\underline{g}},\mathbb{Z})$  
over $\underline{g} \in S_{\mathrm{I_2}(p)}\setminus D_{\mathrm{I_2}(p)}$. According to the two local systems, we consider 
\begin{equation}\begin{array}{ccl}
\label{eq: localsystem}
\tilde{L}_{\mathrm{I_2}(p)}  &\!\! := \! & \text{the lifting of the local system 
$\mathrm{H}_1(E_{\mathrm{I_2}(p),\underline{g}},\mathbb{Z})^{\sigma_{\mathrm{I_2}(p)}}$
to }   \widetilde{S}_{\mathrm{I_2}(p)}\!\! \\
\tilde{L}/\mathrm{rad}(\tilde{L}) \!&\!\! := \! & \text{the lifting of the local system 
$\mathrm{H}_1(\overline{E}_{\mathrm{I_2}(p),\underline{g}},\mathbb{Z})$
to }   \widetilde{S}_{\mathrm{I_2}(p)}.
\end{array}
\end{equation}
Here, $\widetilde{S}_{\mathrm{I_2}(p)}$ is the monodromy $\rho$ covering space of $S_{\mathrm{I_2}(p)}\setminus D_{\mathrm{I_2}(p)}$ defined by 
\begin{equation}
\label{eq:covering}
\widetilde{S}_{\mathrm{I_2}(p)} \ := \ \mathrm{ker}(\rho_{\mathrm{I_2}(p)}) \diagdown  \big(S_{\mathrm{I_2}(p)}\setminus D_{\mathrm{I_2}(p)}\big)^{\sim} ,
\end{equation}
where $(S_{\mathrm{I_2}(p)}\setminus D_{\mathrm{I_2}(p)})^{\sim} $ is the universal covering of $S_{\mathrm{I_2}(p)}\setminus D_{\mathrm{I_2}(p)}$,   on which the fundamental group $G_{\mathrm{I_2}(p)}$ acts from left properly and discontinuously  after choosing a copy  $\widetilde{\Gamma}_{\mathrm{I_2}(p)}$ in $(S_{\mathrm{I_2}(p)}\setminus D_{\mathrm{I_2}(p)})^{\sim} $ of  the base point loci $\Gamma_{\mathrm{I_2}(p)}\subset S_{\mathrm{I_2}(p)}\setminus D_{\mathrm{I_2}(p)}$ \eqref{eq:totallyreal} as the base point loci in  $(S_{\mathrm{I_2}(p)}\setminus D_{\mathrm{I_2}(p)})^{\sim} $. Then we take the quotient of the universal covering by the kernel of the monodromy representation.  The image of $\widetilde{\Gamma}_{\mathrm{I_2}(p)}$ in $\widetilde{S}_{\mathrm{I_2}(p)}$ is again denoted by $\widetilde{\Gamma}_{\mathrm{I_2}(p)}$ and called the base point loci of $\widetilde{S}_{\mathrm{I_2}(p)}$.

By definition, 
$\tilde{L}_{\mathrm{I_2}(p)}$ and $\tilde{L}/\mathrm{rad}(\tilde{L})$ are trivial local systems on the space $\widetilde{S}_{\mathrm{I_2}(p)}$ of homotopy type $S^1$.
\begin{equation}
\label{eq:trivial}
\tilde{L}_{\mathrm{I_2}(p)} \ \simeq \ L_{\mathrm{I_2}(p)}\times \tilde{S}_{\mathrm{I_2}(p)} \quad \text{and}\quad  
\tilde{L}/\mathrm{rad}(\tilde{L}) \ \simeq \ L/\mathrm{rad}(L) \times\  \tilde{S}_{\mathrm{I_2}(p)} 
\end{equation}
with natural inclusion relation $\tilde{L}_{\mathrm{I_2}(p)} \subset \tilde{L}/\mathrm{rad}(\tilde{L})$. For any element of $L_{\mathrm{I_2}(p)} \subset L/\mathrm{rad}(L)$, say $\gamma$, we shall denote by the 
 same $\gamma$ the global section of $\tilde{L}_{\mathrm{I_2}(p)} \subset \tilde{L}/\mathrm{rad}(\tilde{L})$ associated with it,  so far as there may be no-confusion. So, the free basis $\alpha$ and $\beta$ \eqref{eq:invariant} (resp.\ $[\gamma_1]$ and $[\gamma_0]$ (Fact 6)) are lifted to global basis $\alpha$ and $\beta$ (resp.\ $[\gamma_1]$ and $[\gamma_0]$) of $\tilde{L}_{\mathrm{I_2}(p)}$ (resp.\ $\tilde{L}/\mathrm{rad}(\tilde{L})$).
 
\medskip
Let us show in the next Fact 10 that not only those closed cycles $[\gamma_i]$ are lifted to global sections over $\widetilde{S}_{\mathrm{I_2}(p)} $ but also some ``special arcs"  $\delta_i$ on the curve $\overline{E}_{\mathrm{I_2}(p),\underline{g}}$, which originally defined only for $\underline{g}\in \Gamma_{\mathrm{I_2}(p)}$, have global sections over $\widetilde{S}_{\mathrm{I_2}(p)}$. Namely, let us first consider the pullback family 
$$
\tilde{\pi}_{\mathrm{I_2}(p)}\ : \ \tilde{X}_{\mathrm{I_2}(p)}\longrightarrow \tilde{S}_{\mathrm{I_2}(p)}
$$
of the family $\tilde{\pi}_{\mathrm{I_2}(p)}$ \eqref{eq:compact2} to $\tilde{S}_{\mathrm{I_2}(p)}$.
 It carries some additional structures: 
 
 (1) global sections $\{\infty_i \times \tilde{S}_{\mathrm{I_2(p)}}\} $ ($1\le i \le [p/2]$),
 
 (2) fiberwise automorphism $\tilde{\sigma}_{I_(p)}$ such that global sections are cyclically permutated: $\sigma_{\mathrm{I_2}(p)} (\infty_i)=\infty_{i+1}$ {\small ($i\in \mathbb{Z}/[p/2]\mathbb{Z}$)}. 
 
 \medskip
  Recall further that if the parameter $\underline{g}$ belongs to the base point locus 
 $\Gamma_{\mathrm{I_2}(p)}$, then the fiber curve $\overline{E}_{\mathrm{I_2}(p),\underline{g}}$ contains oriented arcs $\delta_i$ ($i=1,\cdots,[p/2]$) such that
\begin{equation}
\label{eq:delta}
\begin{cases}
a) \quad  \partial(\delta_i) \ = \ \infty_{i+1}-\infty_i   \quad \text{ for }  i\in \mathbb{Z}/[p/2]\mathbb{Z}, \\
b) \quad \ \sigma_{\mathrm{I_2}(p)}(\delta_i) \ = \ \delta_{i+1}    \qquad \text{ for }  i\in \mathbb{Z}/[p/2]\mathbb{Z}, \\ 
c) \quad \sum_{i\in \mathbb{Z}/[p/2]\mathbb{Z}} \delta_i \ \sim \ \gamma_0 \ \quad \text{homologous \ in  }\overline{E}_{\mathrm{I_2}(p),\underline{g}}.
\end{cases}
\end{equation}

We extend these $\delta_i$'s to global sections over $\tilde{S}_{\mathrm{I_2}(p)}$ as follows.

\medskip
\noindent
{\bf Fact 10.} 1. There exist global sections for $1\le i\le [p/2]$ 
\begin{equation}
\label{eq:deltai}
\delta_i \quad  : \quad  \underline{g} \ \in \ \tilde{S}_{\mathrm{I_2}(p)} \quad \mapsto \quad   \delta_i(\underline{g}) \ \in \  C_1(\overline{E}_{\mathrm{I_2}(p),\underline{g}},\mathbb{Z})
\end{equation}
which coincide with $\delta_i$ in {\bf Fact 1.} when $\underline{g}\in \Gamma_{\mathrm{I_2}(p)}$, and  satisfy the conditions \eqref{eq:delta}.
Here, $C_1(\overline{E}_{\mathrm{I_2}(p).\underline{g}},\mathbb{Z})$ is the module of singular 1-chains on $\overline{E}_{\mathrm{I_2}(p).\underline{g}}$.

2.    
 The sections are unique up to homologous zero. That is, if there exists other section $\delta'_i$, then $\delta_i(\underline{g})\! - \! \delta'_i(\underline{g}) \sim 0$ (homologous in $\overline{E}_{\mathrm{I_2}(p),\underline{g}}$) for all $\underline{g}\in \tilde{S}_{\mathrm{I_2}(p)}$. 

\smallskip The proof is left to the reader.

\section{Periods for Primitive forms}

We study integrals over paths and cycles in the curves $E_{\mathrm{I_2}(p),\underline{g}}$ of the family \eqref{eq:family}.  
Since we are interested in period integrals over the cycles in $\mathrm{H}_1(E_{\mathrm{I_2}(p),\underline{g}},\mathbb{Z})$ \eqref{eq:homology}, the integrant form should be holomorphic 1-form on $X_{\mathrm{I_2}(p)}$ relative to the base space $S_{\mathrm{I_2}(p)}$.\footnote{
Forms on $X_{\mathrm{I_2}(p)}$ relative to the base space $S_{\mathrm{I_2}(p)}$ means the equivalence classes of forms on $X_{\mathrm{I_2}(p)}$ modulo the $\mathcal{O}_{X_{\mathrm{I_2}(p)}}$-submodule generated by $dF_{\mathrm{I_2}(p)}, dg_s$ and $dg_l$.
}
Still, this condition is too weak to  fix an integrant  (e.g.\  the Betti number of the curves are larger than 1 so that just one choice of an integrant seems insufficient).

In  present paper, we shall focus  on the integrals only of the following form: 
\begin{equation}
\label{eq:primitiveform}
\zeta_{\mathrm{I_2}(p)}:= \Res\Big[\frac{dxdy}{F_{\mathrm{I_2}(p)}(x,y,\underline{g})}\Big]
\end{equation}
(see Footnote 16 for explicit descriptions of the residue \eqref{eq:primitiveform} \footnote{
For each fixed $\underline{g}\in S_{\mathrm{I_2}(p)}$, $dxdy/F_{\mathrm{I_2}(p)}$ may be considered as a top degree meromorphic 2-form on $(x,y)$-plane of simple pole along the curve $E_{\mathrm{I_2}(p),\underline{g}}$. Then, at the smooth point of $E_{\mathrm{I_2}(p),\underline{g}}$, the symbol \eqref{eq:primitiveform} defines the residue, which is a holomorphic one form on  $E_{\mathrm{I_2}(p),\underline{g}}$. Actually, using the $(x,y)$-coordinates, the form is explicitly given by the relative differential forms $\frac{dx}{\partial F_{\mathrm{I_2}(p)}/\partial y} \sim -\frac{dy}{\partial F_{\mathrm{I_2}(p)}/\partial x}$ (which are equivalent modulo $dF_{\mathrm{I_2}(p)}$). Using these expressions, we confirm that $\zeta_{\mathrm{I_2}(p)}$, as a 1-form on $\overline{E}_{\mathrm{I_2}(p),\underline{g}}$, is well-defined up to the singularity loci 
$\{(x,y) \!\in \! E_{\mathrm{I_2}(p),\underline{g}} \mid \partial F_{\mathrm{I_2}(p)}\! /\partial y=\partial F_{\mathrm{I_2}(p)}\!/\partial x =0\}=: \! Sing(E_{\mathrm{I_2}(p),\underline{g}})$ when $\underline{g}\ \in\! D_{\mathrm{I_2}(p)}$.
}),  
which has a characterization, up to a constant factor, as the unique  primitive form for the family \eqref{eq:equation}.
\footnote{
As we shall see immediately in Fact 11, the form $\zeta_{\mathrm{I_2}(p)}$ is, so called,  the elliptic integral of the first kind for the compact elliptic curves   $\overline{E}_{\mathrm{I_2}(p),\underline{g}}$ for $\underline{g} \in S_{\mathrm{I_2}(p)}$ (see, for instance, \cite{Si}). 
On the other hand, from a view point of integrals of vanishing cycles of the universal family \eqref{eq:family} of type $\mathrm{I_2}(p)$, the form $\zeta_{\mathrm{I_2}(p)}$ is, up to an ambiguity of a constant factor, called the primitive form of the family  (to be exact, this is a restriction of the primitive form $\zeta$ defined on the bigger family of type $\mathrm{A_2}$, $A_3$ or $D_4$. Since the action of $\sigma$ on the big family (recall Footnote 9) preserves the form $\zeta$, it induces the form $\zeta_{\mathrm{I_2}(p)}$ \eqref{eq:primitiveform}) (see \cite{S1,S5} for primitive forms). 

A primitive form in general has distinguished characterizations, which we shall implicitly (but not explicitly) use in the present paper (actually, we used already the action of the primitive vector field in \S4, which is an important ingredient of the primitive form theory), and which we recall briefly as follows.
Let us consider the twisted relative de-Rham cohomology group $\mathbb{R}\pi_{\mathrm{I_2}(p),*}(\Omega_{X_{\mathrm{I_2}(p)}/S_{\mathrm{I_2}(p)}} )$ of the family \eqref{eq:family}.It is a filtered $\mathcal{O}_s$-module equipped with the Gauss-Manin connection and higher residue pairings. Then, $\zeta_{\mathrm{I_2}(p)}$ is an element in the 0th filter satisfying 1) primitivity, 2) homogeneity, 3) orthogonality, 4) holonomicity (see, for instance, \cite{S1}). In particular, the primitivity means that the covariant differentiations of $\zeta_{\mathrm{I_2}(p)}$ generate all cohomology classes so that just a single choice of $\zeta_{\mathrm{I_2}(p)}$ is sufficient. 
}

\medskip
In the present section, we introduce the period map  \eqref{eq:periods} by integrating $\zeta_{\mathrm{I_2}(p)}$ over closed cycles in $E_{\mathrm{I_2}(p),\underline{g}}$. The description of its inversion map \eqref{eq:periodmap} by the use of Eisenstein series over the lattices $L_{\mathrm{I_2}(p)}\subset L/\rad(L)$ is the main subject of the present paper. In order to understand them, the study of integrals over the closed cycles are not sufficient, but we need to develop a study of integrals of $\zeta_{\mathrm{I_2}(p)}$ over non closed paths in  $\overline{E}_{\mathrm{I_2}(p),\underline{g}}$. Precisely, we study  integrals over the arcs $\delta_i$ in the present section and over indefinite paths in the next section.

Consider the integral of $\zeta_{\mathrm{I_2}(p)}$ over a horizontal family  of  cycles $\gamma \in \mathrm{H}_1(E_{\mathrm{I_2}(p),\underline{g}},\mathbb{Z})$ 
(where $\underline{g}$ runs over a simply connected open subset of $S_{\mathrm{I_2}(p)}\! \setminus \! D_{\mathrm{I_2}(p)}$)
\begin{equation}
\label{eq:period}
\omega_{\gamma}\ :=\ \oint_\gamma \zeta_{\mathrm{I_2}(p)}.
\end{equation}
 As is well-known that $\omega_\gamma$ is a holomorphic function in the parameter $\underline{g}$ (which can be confirmed by a use of Leray's residue formula) on the domain where the family $\gamma$ is defined.  On the other hand, $\omega_\gamma$ depends only on its equivalence class $[\gamma]\in L/\rad(L)$ due to the following fact.
 
\medskip
\noindent
{\bf Fact 11.} {\it The form $\zeta_{\mathrm{I_2}(p)}$ can be extended holomorphically on the compactified elliptic curves $\overline{E}_{\mathrm{I_2}(p),\underline{g}}$ for $\underline{g}\in S_{\mathrm{I_2}(p)}$. The extended form is nowhere vanishing on the smooth part $\overline{E}_{\mathrm{I_2}(p),\underline{g}}\setminus Sing(\overline{E}_{\mathrm{I_2}(p),\underline{g}})$ of the curve.} 
\begin{proof}
 If $(x,y)\in E_{\mathrm{I_2}(p),\underline{g}}$ is a non-singlar point, then either of  $\partial F_{\mathrm{I_2}(p)}(x,y,\underline{g})/\partial x$ or $\partial F_{\mathrm{I_2}(p)}(x,y,\underline{g})/\partial y$ is non-zero. Then, one of the two explicit expressions of $\zeta_{\mathrm{I_2}(p)}$ in Footnote 16 gives a non-vanishing  and holomorphic expression of $\zeta_{\mathrm{I_2}(p)}$ at the point $(x,y)$. At the infinity points, we have already observed in Footnote 3 that $\overline{E}_{\mathrm{I_2}(p),\underline{g}}$ is smooth. Then using that expressions, we check again that $\zeta_{\mathrm{I_2}(p)}$ is holomorphic and non-vanishing.  The details are left to the reader.  
\end{proof}

This means that  $\zeta_{\mathrm{I_2}(p)}$ may be regarded as a relative de Rham cohomology class of the family of compactified elliptic curves $\overline{E}_{\mathrm{I_2}(p),\underline{g}}$. The cycles $\gamma_1,\cdots,\gamma_{[p/2]}$ are homologous to each other in $\overline{E}_{\mathrm{I_2}(p),\underline{g}}$. This implies
$
\oint_{\gamma_1} \zeta_{\mathrm{I_2}(p)}=\cdots = \oint_{\gamma_{[p/2]}} \zeta_{\mathrm{I_2}(p)},
$
 and,  $\omega_\gamma$ depends only on the class $[\gamma]\in \widetilde{L}/\rad(\widetilde{L})$. Therefore, we regard $\omega_\gamma$ as a holomorphic function defined on $\widetilde{S}_{\mathrm{I_2}(p)}$. 
 In particular, in view of \eqref{eq:alpha-beta}, we have expressions
$$
\omega_\alpha:=\oint_{\alpha} \zeta_{\mathrm{I_2}(p)} \ = \  [p/2]\  \omega_{\gamma_1} \quad \ \text{and} \ \quad 
\omega_\beta:=\oint_{\beta}\zeta_{\mathrm{I_2}(p)} \ =  \ \omega_{\gamma_0} .
$$
 Thus, in the present paper, we shall integrate either over a cycle in $\widetilde{L}_{\mathrm{I_2}(p)}$ and a cycle $\gamma$ in $\tilde{L}/\rad(\tilde{L})$.\footnote
{It is rather restrictive view point to study  integrals only over $\sigma_{\mathrm{I_2}(p)}$-invariant cycles or over equivalent class of cycles. This is caused by the fact that our family \eqref{eq:equation} is already the  subfamily of the full families of type $A_3$ or $D_4$, and is fixed by a cyclic action $\sigma$ (recall Footnote 1 and 7) where the full ($A_3$ or $D_4$) lattice $L$ does not play role. The studies of period integrals of the primitive form over the full lattice $L$ in the big families of type $A_3$ and $D_4$ (unpublished) are by themselves interesting subject and should appear elsewhere.
} (in such situation, we shall say ``integrate a cycle $\gamma \in \widetilde{L}_{\mathrm{I_2}(p)} \subset \widetilde{L}/\rad (\widetilde{L})$"), or over ``special arcs" $\delta_i$ \eqref{eq:deltai}.

For each point $\tilde{\underline{g}}\in \widetilde{S}_{\mathrm{I_2}(p)}$ \eqref{eq:covering}, we consider the linear map 
\begin{equation}
\label{eq:periods}
\gamma \ \in \  \widetilde{L}_{\mathrm{I_2}(p)} \subset \widetilde{L}/\rad (\widetilde{L})  
\quad \mapsto \quad  \omega_\gamma:=\oint_\gamma \zeta_{\mathrm{I_2}(p)} \ \in \ \mathcal{O}_{\tilde{S}_{\mathrm{I_2}(p)}} 
\end{equation}
where we set $\mathcal{O}_{\tilde{S}_{\mathrm{I_2}(p)}}=\text{holomorphic functions on $\tilde{S}_{\mathrm{I_2}(p)}$}$.

Using the trivialization \eqref{eq:trivial} of $\widetilde{L}_{\mathrm{I_2}(p)} \subset \widetilde{L}/\!\sim$, we obtain a holomorpbic map
\begin{equation}
\label{eq:map}
\begin{array}{rccc}
\quad P_{\mathrm{I_2}(p)}  :  \!& \! \widetilde{S}_{\mathrm{I_2}(p)}  \! & \! \rightarrow\! &\!  \Hom_\mathbb{Z}(L_{\mathrm{I_2}(p)},\mathbb{C}) =  \Hom_\mathbb{Z}(L/\mathrm{rad}(L) ,\mathbb{C})\quad (\simeq \ \mathbb{C}^2) \!\! \!\!\\
\\
& \underline{g} & \! \mapsto\! &   \gamma \  \mapsto \ \omega_\gamma:= \oint_{\gamma(\underline{g})} \zeta_{\mathrm{I_2}(p)} \quad
\end{array}
\end{equation}
where the identification in RHS is canonically given by the change 
\begin{equation}
\label{eq:level}
\omega_\alpha= [p/2] \ \omega_{\gamma_1}  \quad \text{and} \quad  \omega_\beta =\omega_{\gamma_0} 
\end{equation}
of the basis in both vector spaces. We shall call $P_{\mathrm{I_2}(p)}$ the {\it period map} associated with the primitive form $\zeta_{\mathrm{I_2}(p)}$ \eqref{eq:primitiveform}.  

\medskip
\noindent
{\bf Fact 12.}  {\it The period map $P_{\mathrm{I_2}(p)}$ is locally bi-holomorphic.}
\begin{proof}
We show the Jacobian determinant of the map is no-where vanishing. We use an essential property: the {\it primitivity} of $\zeta_{\mathrm{I_2}(p)}$ \cite{S1}. Since the proof uses relative de-Rham cohomology theory for the family \eqref{eq:family} which is beyond the scope of present paper, a complete proof is left to the literature but we give here a brief sketch of its idea. 

There exists, so called, the Gauss-Manin covariant differentiation operator $\nabla$ on the module (over $\mathcal{O}_{S_{\mathrm{I_2}(p)}}$) of relative de Rham cohomology classes of  \eqref{eq:family}. Then, one basic property, called the primitivity, of a primitive form is that its covariant differentiations $\nabla_{\partial g_s}\zeta_{\mathrm{I_2}(p)}$ and $\nabla_{\partial g_l}\zeta_{\mathrm{I_2}(p)}$ generate the relative de-Rham cohomology module (here $\partial_{g_s}$ and $\partial_{g_l}$ stand for the partial derivatives w.r.t. the coordinate system $\underline{g}$). Then, standard duality between the de-Rham cohomology group and the ($\sigma_{\mathrm{I_2}(p)}$-invariant) singular homology group of the curve $E_{\mathrm{I_2}(p),\underline{g}}$ implies 
$\det
\begin{bmatrix}
\oint_{\gamma_0} \nabla_{\partial_{g_s}}\zeta_{\mathrm{I_2}(p)} &\!\! \oint_{\gamma_0} \nabla_{\partial_{g_l}}\zeta_{\mathrm{I_2}(p)}\\
\oint_{\gamma_1} \nabla_{\partial_{g_s}}\zeta_{\mathrm{I_2}(p)} &\!\! \oint_{\gamma_1} \nabla_{\partial_{g_l}}\zeta_{\mathrm{I_2}(p)} 
\end{bmatrix}
\not=0
$.
Since the integral $\oint$ commutes with the derivation action $\partial_g$ and the covariant differentiation $\nabla_{\partial_{g}}$, we see that the Jacobian determinant 
$\det
\begin{bmatrix}
\partial_{g_s}\oint_{\gamma_0} \zeta_{\mathrm{I_2}(p)} &\!\! \partial_{g_l}\oint_{\gamma_0} \zeta_{\mathrm{I_2}(p)}\\
\partial_{g_s} \oint_{\gamma_1} \zeta_{\mathrm{I_2}(p)} &\!\! \partial_{g_s} \oint_{\gamma_1} \zeta_{\mathrm{I_2}(p)} 
\end{bmatrix}
$
$\not=0
$
does not vanish.
\end{proof}

\smallskip
Let us now formulate the first main theorem of the present paper.

\begin{theorem}
\label{perioddomain}
{\it The period map \eqref{eq:map} induces a biholomorphic  map
\begin{equation}
\label{eq:periodmap}
P_{\mathrm{I_2}(p)}\ :\ \widetilde{S}_{\mathrm{I_2}(p)} 
\quad \simeq \quad
  \widetilde{\mathbb{H}} \ , \qquad
\end{equation}
where the RHS of \eqref{eq:periodmap}, so called the period domain, is given as
\begin{equation}
\begin{array}{rcl}
\label{eq:perioddomain}
\widetilde{\mathbb{H}} & := & \{\omega\in\Hom_\Z(L_{\mathrm{I_2}(p)},\C)\mid
\Im(\omega_\alpha/\omega_\beta)>0\} \\
\\
\\
\\
&= & \{\omega\in\Hom_\Z(L/\mathrm{rad}(L),\C)\mid
\Im(\omega_{\gamma_1}/\omega_{\gamma_0})>0\},
\end{array}
\end{equation}
(here we used again \eqref{eq:level} for the identification of the first and the second lines).
The map \eqref{eq:map} 
 is equivariant with the action of the group $\Gamma_1([p/2])\subset \mathrm{SL}_{\mathbb{Z}}(L_{\mathrm{I_2}(p)})\cap \mathrm{SL}_2(\mathbb{Z})$.}
\end{theorem}
\begin{proof} 
This result for the case of type $\mathrm{A_2}=\mathrm{I_2}(3)$, is the well-known classic (e.g.\ see \cite{Si}). We want to show the parallel world for the other types $\mathrm{B_2}$ and $\mathrm{G_2}$ exists.  
The complete proof of this theorem can be given after solving the Jacobi-inversion problem for those period maps in Theorem \ref{primitiveautomorphicform} in \S9.  
%
In the present section, we show only the following classical well-known fact:

\medskip
\noindent
{\bf Fact 13.}
{\it The image of the period map \eqref{eq:map} is contained in RHS of \eqref{eq:periodmap}. 
In particular, by the period integral \eqref{eq:period}, the lattice $\widetilde{L}/\rad(\widetilde{L})_{\underline{g}}$  for $\underline{g}\in \widetilde{S}_{\mathrm{I_2}(p)}$  is embedded into a discrete lattice, called the period lattice: 
\begin{equation}
\label{eq:periodlattice}
\Omega_{\widetilde{L}/\rad(\widetilde{L}),\underline{g}}:=\mathbb{Z}\omega_{\gamma_0} + \mathbb{Z}\omega_{\gamma_1}
\end{equation}
in the complex plane $\mathbb{C}$. }

\medskip
 Actually, this is well-known as a consequence of  (1) the Riemann's inequality:
$
\frac{1}{2\sqrt{-1}}\int_{\overline{E}_{\mathrm{I_2}(p),\underline{g}}} \overline{\zeta_{\mathrm{I_2}(p)}} \wedge \zeta_{\mathrm{I_2}(p)}>0,
$
due to the positivity of the real volume form  
$
\frac{1}{2\sqrt{-1}}\overline{\zeta_{\mathrm{I_2}(p)}} \wedge \zeta_{\mathrm{I_2}(p)}
$
and (2) the Stokes relation:
$
\int_{\overline{E}_{\mathrm{I_2}(p),\underline{g}}} \!\!\! \overline{\zeta_{\mathrm{I_2}(p)}} \! \wedge \! \zeta_{\mathrm{I_2}(p)} \! = \! \overline{\oint_{\gamma_1}\!\! \zeta_{\mathrm{I_2}(p)}} \oint_{\gamma_0} \!\! \zeta_{\mathrm{I_2}(p)}
\! - \! \overline{\oint_{\gamma_0} \!\! \zeta_{\mathrm{I_2}(p)}} \oint_{\gamma_1}\!\! \zeta_{\mathrm{I_2}(p)}
$, 
due to the fact that the cycles $\gamma_0$ and $\gamma_1$ give a canonical dissection of the real surface $\overline{E}_{\mathrm{I_2}(p),\underline{g}}$ where $ \zeta_{\mathrm{I_2}(p)}$ has no poles (i.e.\ it is only a topological but not analytical property.  c.f.\ the first row of Figure 3.).
\end{proof}

Before ending this section, let us consider also the integrals over the special arcs $\delta_i$ \eqref{eq:deltai} constructed in Fact 1 in \S2 and Fact 10 in  \S4.  Namely, for every $\underline{g}\in \tilde{S}_{\mathrm{I_2}(p)}$ and $1\le i\le [p/2]$, set 
\begin{equation}
\label{eq:deltaperiod}
\omega_{\delta_i} := \int_{\delta_i} \zeta_{\mathrm{I_2}(p)}.
\end{equation}
Recall that the $\delta_i$'s are cyclically permuted by $\sigma_{\mathrm{I_2}(p)}$ and their sum is homologous to $\gamma_0$ (see \eqref{eq:deltai} and Fact 10.1).  On the other hand, one sees immediately from the expression \eqref{eq:primitiveform} that the form $\zeta_{\mathrm{I_2}(p)}$ is invariant by the action of $\sigma_{\mathrm{I_2}(p)}$.  These together implies $\omega_{\delta_1}=\cdots= \omega_{\delta_{[p/2]}}$ and $\omega_{\delta_1}+\cdots+ \omega_{\delta_{[p/2]}}=\omega_{\gamma_0}$. Thus those arc integrals \eqref{eq:deltaperiod} are expressed in terms of a classical period of a vanishing cycle as follows.
\begin{equation}
\label{eq:deltaperiod2}
 \omega_{\delta_1}=\cdots= \omega_{\delta_{[p/2]}}= \frac{1}{[p/2]} \omega_{\beta}\ \big(= \frac{1}{[p/2]} \omega_{\gamma_0} \big)
\end{equation}

Finally in this section, let us notice some elementary but useful facts.

\medskip
\noindent
{\bf Fact 14.} {\it The periods $\omega_{\gamma_i}$ and $\omega_{\delta_i}$  ($i=1,\cdots,[p/2]$) are weighted homogeneous functions on $\tilde{S}_{\mathrm{I_2}(p)}$ of  weights given in Table 1. That is,  the $\mathbb{C}^\times$ action on $S_{\mathrm{I_2}(p)}$ is naturally lifted to that on $\tilde{S}_{\mathrm{I_2}(p)}$, which, for an abuse of notation, we shall denote $\underline{g}\mapsto t^{\wt(\underline{g})}\underline{g}$,
so that we have the following equivariance w.r.t.\ the action of $t\in \mathbb{C}^\times$.
$$
\begin{array}{c}
\omega_{\gamma_i}(t^{\wt(\underline{g})}\underline{g}) \ = \ t^{\wt(\omega_{\gamma_i} )}\ \omega_{\gamma_i}(\underline{g}) \text{\quad and \quad} \omega_{\delta_i}( t^{\wt(\underline{g})}\underline{g}) \ = \ t^{\wt(\omega_{\delta_i})} \ \omega_{\delta_i}(\underline{g})
\end{array}
$$
where $\wt(\omega_{\gamma_i})=\wt(\omega_{\delta_i})=\wt(z)$. Putting negatively graded structure on RHS's of \eqref{eq:map} and \eqref{eq:periodmap} by multiplication $t^{\wt(z)}$ for $t\in\mathbb{C}^\times$, the maps \eqref{eq:map} and \eqref{eq:periodmap} are equivariant with respect to the $\mathbb{C}^\times$-action.}

\medskip
 {\it Proof.} Recall that the $\mathbb{C}^\times$ action on the space $X_{\mathrm{I_2}(p)}$ extends to its partial compactification $\overline{X}_{\mathrm{I_2}(p)}$ in such manner that the restriction of the action on the divisors $\infty_i\times S_{\mathrm{I_2}(p)}$ are equivariant  with that on $S_{\mathrm{I_2}(p)}$.  Then, the action induces an action on the set of paths in $\overline{E }_{\mathrm{I_2}(p),\underline{g}}$ starting from  $\infty_1\times \underline{g}$, where the end point of the path is acted by the  $\mathbb{C}^\times$. We replace the integral \eqref{eq:Hamilton-time} over a path by the integral over the ``acted" path. Due to the expression \eqref{eq:primitiveform}, we have $\wt(z)=\wt(x)+\wt(y)-\wt(F_{\mathrm{I_2}(p)})$.
$\Box$.

\bigskip
\section{Jacobian variety}
 We 
study  integrals of the primitive form $\zeta_{\mathrm{I_2}(p)}$ over paths in the smooth part of the curve $\overline{E}_{\mathrm{I_2}(p), \underline{g}}$ for each fixed $\underline{g}\in S_{\mathrm{I_2}(p)}\! \setminus \! D_{\mathrm{I_2}(p)}$ which may not neccesarily be closed. That is, we study the Jacobian variety of the curve $\overline{E}_{\mathrm{I_2}(p),\underline{g}}$.\footnote
{We can study in parallel the case when $\underline{g}$ belongs to the discriminant $D_{\mathrm{I_2}(p)}$ by replacing $\overline{E}_{\mathrm{I_2}(p)\underline{g}}$ by $\overline{E}_{\mathrm{I_2}(p)\underline{g}}\! \setminus \! Sing(\overline{E}_{\mathrm{I_2}(p)\underline{g}})$ and $L_{\mathrm{I_2}(p)}$ by the 
lattices of rank 1 and 0. 
Details are left to the reader.}
\  Precisely, we focus on  the integrals from the point at infinity $\infty_1$: 
\begin{equation}
\label{eq:Hamilton-time}
\qquad \qquad  z:=\int_{\infty_1}^{(x,y)^{\sim}}\zeta_{\mathrm{I_2}(p)}  \qquad \in \ \mathbb{C}
\end{equation} 
where $(x,y)^{\sim}$ is a point in the universal covering (with respect to the base point $\infty_1$) of  the curve $\overline{E}_{\mathrm{I_2}(p)\underline{g}}$ which lies over a point $(x,y) \in \overline{E}_{\mathrm{I_2}(p)\underline{g}}$, or, equivalently, a homotopy class of rectifiable paths in the curve  $\overline{E}_{\mathrm{I_2}(p)\underline{g}}$ from $\infty_1$ to $(x,y) \in \overline{E}_{\mathrm{I_2}(p)\underline{g}}$ (the notation ${(x,y)^{\sim}}$ is ambiguous and we use it only here).

It is a classic that the integral \eqref{eq:Hamilton-time} induces a biholomorphic map from the universal covering of $\overline{E}_{\mathrm{I_2}(p), \underline{g}}$ to the complex plane $\mathbb{C}$ ({\it Proof.}  The map is locally bi-regular (Fact 11) and is equivariant with the covering transformation of $\pi_1(\overline{E}_{\mathrm{I_2}(p)\underline{g}},\infty_1)$ on $\overline{E}_{\mathrm{I_2}(p), \underline{g}}^\sim$ with the translation action by the full period lattice $\Omega_{\widetilde{L}/\rad(\widetilde{L})\underline{g}}\!\! =\!  \mathbb{Z}\omega_{\gamma_0}\! \oplus \! \mathbb{Z}\omega_{\gamma_1}$ on $\mathbb{C}$  (Fact 13)).


\smallskip
For a later use in the study of inversion problem, let us confirm the direction of the Hamiltonian $F_{\mathrm{I_2}(p)}$ at the base point $\underline{g}_0\in \Gamma_{\mathrm{I_2}(p)}$ as follows.
 

\medskip
\noindent
{\bf Fact 15.} {\it By the map \eqref{eq:Hamilton-time} for $\underline{g}_0\in \Gamma_{\mathrm{I_2}(p)}$, the cycle \eqref{eq:realcycle} is mapped to the real interval $[0, \omega_\beta]$ in the $z$-complex plane.}

In particular, one has the correspondence:  $\infty_i   \leftrightarrow   \frac{i}{[p/2]} \omega_\beta \bmod \mathbb{Z} \omega_\beta$.

 \begin{figure}[h]
 \center
  \hspace{-0.3cm}\includegraphics[width=6.0cm]{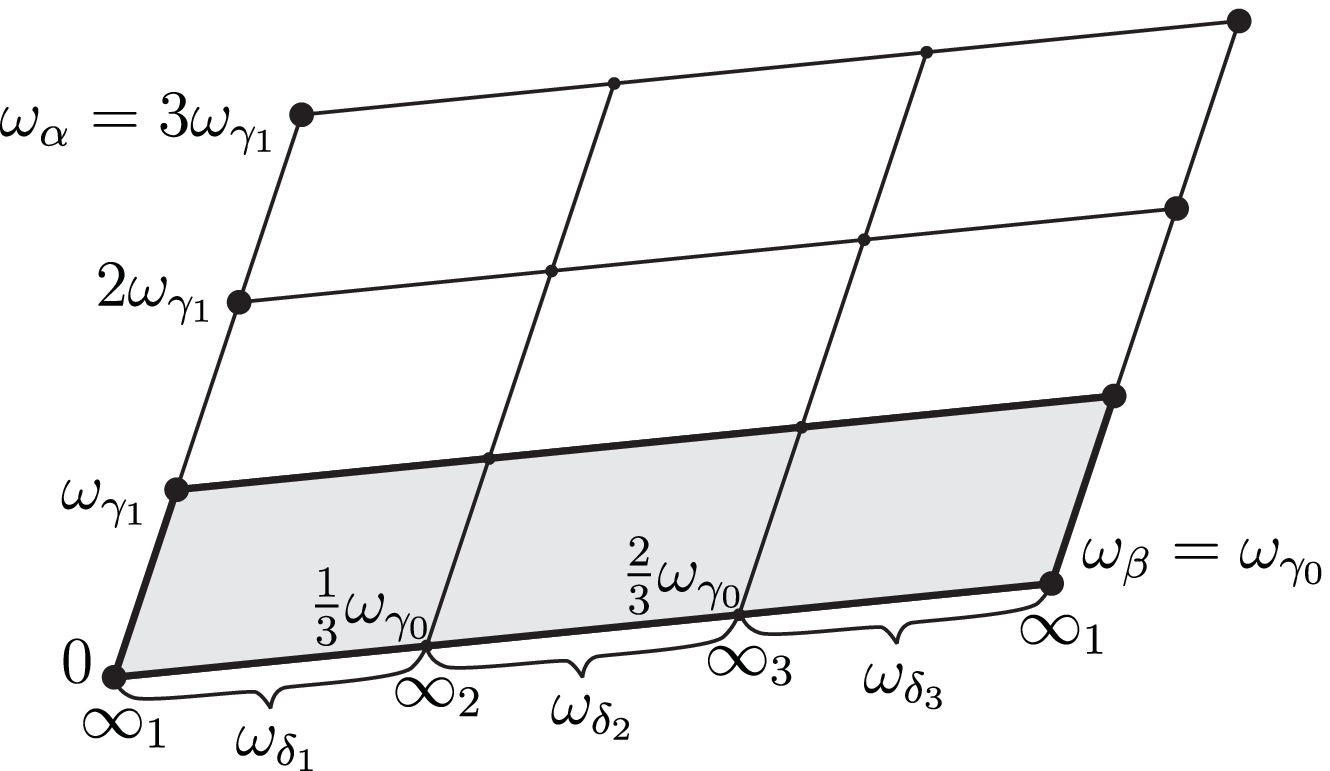}
  \vspace{0.2cm}
 \end{figure}
\centerline{ {\bf Figure 5.} {Period lattice $\Omega_{\mathrm{G_2},\underline{g}}$ and its fundamental domain}}

\noindent
 {\footnotesize  The big black spots indicate points in   $\Omega_{\mathrm{G_2},\underline{g}}$, and other small spots indicate other places of the poles of the functions $x_{\mathrm{I_2}(p)}(z,\underline{g})$ and  $y_{\mathrm{I_2}(p)}(z,\underline{g})$. Shaded area is  a fundamental domain (=period parallelogram) for translation action of  $\Omega_{\mathrm{G_2},\underline{g}}$.}

\begin{proof} 
Let's first observe that the integral \eqref{eq:Hamilton-time} over the arc $\delta_1$ takes real increasing values. For the purpose, we only need to notify that $\delta_1$ is a real path as in Figure 2. and the integrant $\zeta_{\mathrm{I_2}(p)}=\frac{dx}{\partial F_{\mathrm{I_2}(p)}/\partial y}$ takes positive real values on the arc $\delta_1$, or equivalently, the function $\partial F_{\mathrm{I_2}(p)}/\partial y$ takes negative real values  when $x$ is decreasing and takes positive real values when x is increasing on the path $\delta_1$ (near at $\infty_1$). This can be confirmed directly depending on cases using the condition  \eqref{eq:delta1}. The same argument works for the integral over the arcs $(\delta_2,\infty_2),\cdots,(\delta_{[p/2]},\infty_{[p/2]})$. Then, recall that the cycle \eqref{eq:realcycle} is homologous to $\gamma_0=\beta$.
\end{proof}

We consider the inverse map:  $z= \int_{\infty_1}^{(x,y)^{\sim}}\zeta_{\mathrm{I_2}(p)} \in \mathbb{C} \mapsto (x,y) \in \overline{E}_{\mathrm{I_2}(p),\underline{g}} $. More precisely, we associate to $z\in \mathbb{C}$ the coordinate values of the corresponding point $(x,y)\in \overline{E}_{\mathrm{I_2}(p), \underline{g}}$ by the relation \eqref{eq:Hamilton-time} and denote it by
\begin{equation}
\label{eq:pfunction}
x_{\mathrm{I_2}(p)}(z,\underline{g})  \quad \text{and}  \quad  y_{\mathrm{I_2}(p)}(z,\underline{g}), 
\end{equation}
respectively.\footnote{
Inversion expression of the coordinate $(x,y)$ with respect to  the integral value $w= \int_{\infty}^{(x,y)} \omega$ over the elliptic integral of the first kind $\omega=dx/${\tiny $\sqrt{4x^3-g_{s}x+g_l}$} is well-known to be given by Weierstrass $\mathfrak{p}$-function and its derivative as $x=\mathfrak{p}(w),y=\mathfrak{p}'(w)$.  Since there exists a factor relation $\zeta_{\mathrm{A_2}}=\omega/2$, we have the relation $w=2z$ and the period lattice gets half size. Then, $x_{\mathrm{A_2}}(z) =\frac{1}{4}\mathfrak{p}(z), y_{\mathrm{A_2}}(z)=\frac{1}{8}\mathfrak{p}'(z)$.
}
 To be careful,  the values of the ``functions" $x_{\mathrm{I_2}(p)}$ and $y_{\mathrm{I_2}(p)}$ may not be defined when $(x,y)$ represents a point at infinity \eqref{eq:compact1}, however, the function is obviously holomorphic at other points and we see easily those undefined points are removable to a holomorphic or meromorphic function. \footnote
{Here, one should be slightly cautious that $x_{\mathrm{I_2}(p)}$ and $y_{\mathrm{I_2}(p)}$ are (at present stage) as  functions in  $\underline{g} \in S_{\mathrm{I_2}(p)}\setminus D_{\mathrm{I_2}(p)}$ only pointwise.  Their holomorphic dependence on $\underline{g}$ can be shown again by a use of Leray's residue formula (details are omitted). However,  their  extendability to $D_{\mathrm{I_2}(p)}$ is not a priori obvious). Actually combining with the  discussions in Footnote 15, it is possible to show that they actually are extendable to the functions on the whole $S_{\mathrm{I_2}(p)}$. However, without using that logic, we show directly in Lemma \ref{eq:LaurentSolution}  in \S7.
}

Recalling the two local expressions  in Footnote 16 of the de Rham class of $\zeta_{\mathrm{I_2}(p)}$, we obtain the following Hamilton's equation of motion:
\begin{equation}
\label{eq:Hamiltonean}
\frac{\partial x_{\mathrm{I_2}(p)}(z,\underline{g})}{\partial z}  =\frac{\partial F_{\mathrm{I_2}(p)}(x,y,\underline{g})}{\partial y}  \quad \text{and} \quad 
\frac{\partial y_{\mathrm{I_2}(p)}(z,\underline{g})}{\partial z}  = -\frac{\partial F_{\mathrm{I_2}(p)}(x,y,\underline{g})}{\partial x}. 
\end{equation}
However, this equation of the motion (which depends only on $g_s$ but not on $g_l$) alone does not determine the solution \eqref{eq:pfunction} uniquely. 
In order to recover the functions \eqref{eq:pfunction} as a function in $z$, we need to put the following constraint on the energy level (depending on $g_l$) of the motion:
\begin{equation}
\label{eq:energy}
F_{\mathrm{I_2}(p)}(x_{\mathrm{I_2}(p)},y_{\mathrm{I_2}(p)},\underline{g})=0
\end{equation}
We remark also that the functions $x_{\mathrm{I_2}(p)}$ and $y_{\mathrm{I_2}(p)}$, which are no-longer polynomials but meromorphic in $z$, are still weighted homogeneous functions if we give the weights to $x,y, z$ and $\underline{g}$ as given in the Table 1. That  is, we have the following equivariance w.r.t. the action of $t\in \mathbb{C}^\times$.
\begin{equation}
\label{eq:weightedhomo}
\begin{array}{c}
x_{\mathrm{I_2}(p)}(t^{\wt(z)}z, t^{\wt(\underline{g})}\underline{g}) \ = \ t^{\wt(x)} \ x_{\mathrm{I_2}(p)}(z,\underline{g}),\\ y_{\mathrm{I_2}(p)}(t^{\wt(z)}z, t^{\wt(\underline{g})}\underline{g}) \ = \ t^{\wt(y)} \ y_{\mathrm{I_2}(p)}(z,\underline{g})
\end{array}
\end{equation}
 ({\it Proof of \eqref{eq:weightedhomo}.}) Recall that the $\mathbb{C}^\times$ action on the space $X_{\mathrm{I_2}(p)}$ extends to its partial compactification $\overline{X}_{\mathrm{I_2}(p)}$ in such manner that the point $\infty_1$ at infinity of a curve stays at infinity $\infty_1$ of the curve whose parameter $\underline{g}$ is acted by the $\mathbb{C}^\times$ action.  Then, the action induces an action on the set of paths  starting from  $\infty_1$ to a point in the curve, where the end point of the path is acted by the  $\mathbb{C}^\times$. We replace the integral \eqref{eq:Hamilton-time} over a path by the integral over the ``acted" path. Due to the expression \eqref{eq:primitiveform}, we have $\wt(z)=\wt(x)+\wt(y)-\wt(F_{\mathrm{I_2}(p)})$.
$\Box$). 

One crucial fact here is that the only variable $z$ has the negative weight (recall Table 1) and the functions can be (and, actually, is)  transcendental in $z$.

\section{Laurent series solutions at infinities}

We study formal Laurent series solutions of the equation of the motion \eqref{eq:Hamiltonean} together with the constraint \eqref{eq:energy} of the energy level and the weight condition \eqref{eq:weightedhomo}. The solutions are exactly in one to one correspondence with the set of points at infinity of the curve $\widetilde{E}_{\mathrm{I_2}(p),\underline{g}}$.  

More exactly, we do the following shift of center of Laurent expansion. 
Namely, if $x_{\mathrm{I_2}(p)}(z)$ or $y_{\mathrm{I_2}(p)}(z)$ has a non trivial pole at $z=\omega(\underline{g})$ (where $\omega(\underline{g})$ is a function  of $\underline{g}\in \widetilde{S}_{\mathrm{I_2}(p)}$  of weight  $=\wt(z)$, then we consider the Laurent expansion of the pair $({\mathbf x}(\z),{\mathbf y}(\z)):=(x_{\mathrm{I_2}(p)}(\z+\omega(\underline{g})),y_{\mathrm{I_2}(p)}(\z+\omega(\underline{g})))$ with respect to the local formal coordinate $\z$ at $0$. The pair $({\mathbf x}(\z),{\mathbf y}(\z))$ satisfies the pair of  equations 
$$
\leftline{(5.3)* \quad\quad  {\large $
\frac{\partial {\mathbf x}(\z,\underline{g})}{\partial \z}  =\frac{\partial F_{\mathrm{I_2}(p)}({\mathbf x},{\mathbf y},\underline{g})}{\partial {\mathbf y}}  \quad \text{\normalsize and} \quad 
\frac{\partial {\mathbf y}(\z,\underline{g})}{\partial \z}  = -\frac{\partial F_{\mathrm{I_2}(p)}({\mathbf x},{\mathbf y},\underline{g})}{\partial {\mathbf x}} $}}
$$
since the equations \eqref{eq:Hamiltonean} are invariant by the shift of the center of the expansion. Next Lemma classify all formal solutions of the equations with non-trivial poles. 

\medskip
\begin{lem} 
\label{eq:LaurentSolution} 
{\it Consider the system of the equations \eqref{eq:Hamiltonean} together with the constraint \eqref{eq:energy} and the weight condition \eqref{eq:weightedhomo}. Then it has exactly $[p/2]$-pairs of formal Laurent series solutions having non-trivial pole, which are in one to one  correspondence with the set  $\{\infty_1,\cdots,\infty_{[p/2]}\}$ of the points at infinity \eqref{eq:compact1}  of the curve $\overline{E}_{\mathrm{I_2}(p)}$. 
\begin{equation}
\label{eq:Solution}
\begin{array}{ccc}
 \text{\{Solutions with non-trivial pole 
 \} } & \simeq & \text{  \{Points at infinity\} }\!\!\!\\
({\mathbf x}(\z),{\mathbf y}(\z))  & \mapsto  & \ \underset{\z \downarrow 0}{ \lim }\  ({\mathbf x}(\z),{\mathbf y}(\z)) \!\!
\end{array}
\end{equation}
In particular, the linear transformation $\sigma_{\mathrm{I_2}(p)}$ \eqref{eq:autom} acts on the set of solutions cyclically equivariant with the bijection \eqref{eq:Solution}. The coefficients of the Laurent series are $\mathbb{Q}$-coefficients weighted homogeneous polynomials in $\underline{g}$ so that the solution is a pair of weighted homogeneous functions of weight $(\wt(x),\wt(y))$.

}
\end{lem}
\begin{proof}
The proof is divided into 4 steps.

\smallskip
\noindent
Step 1.  For each type $\mathrm{I_2}(p)$, consider the pair of Laurent series:\begin{equation}
\label{eq:Laurent}
\begin{array}{cccccccc}
{\bf x}_{\mathrm{A_2}}(\z,\underline{g}) &\!=\!& \sum_{n=-a}^\infty A_n\z^{2n} & \text{and} & {\bf y}_{\mathrm{A_2}}(\z,\underline{g}) &\!=\! &\sum_{n=-b}^\infty B_n\z^{2n-1} \\
{\bf x}_{\mathrm{B_2}}(\z,\underline{g}) &\!=\!& \sum_{n=-a}^\infty A_n\z^{2n+1} & \text{and} & {\bf y}_{\mathrm{B_2}}(\z,\underline{g}) & \!=\! &\sum_{n=-b}^\infty B_n\z^{2n} \\
{\bf x}_{\mathrm{G_2}}(\z,\underline{g}) & \!=\!& \sum_{n=-a}^\infty A_n\z^n & \text{and} & {\bf y}_{\mathrm{G_2}}(\z,\underline{g}) &\!=\! &\sum_{n=-b}^\infty B_n\z^n \\
\end{array}
\end{equation} 
of indeterminate coefficients $A_n$ and $B_n$ with non vanishing leading terms $A_{-a}B_{-b}\not=0$ \footnote
{Here, we, unfortunately, use the notation $A_n$ and $B_n$ for the coefficients of the Laurent series, which have nothing to do with the classification names $A_n$ and $B_n$ for root systems. Since they are used only inside present proof, one should cautiously read them.
}
(according to the even property of $x_{\mathrm{A_2}}$ and $y_{\mathrm{B_2}}$ or odd property of $y_{\mathrm{A_2}}$ and $x_{\mathrm{B_2}}$ caused by the $\mathbb{Z}/2$-symmetry $(x,y)\! \to\! (x,-y)$, the sum consist either of even or odd powers in $\z$). 
We assume that the pair has non-trivial pole, i.e.\ at least one of $a$ or $b$ is positive. Here the coefficients are weighted homogeneous functions in $\underline{g}\in \widetilde{S}_{\mathrm{I_2}(p)}$, a priori not necessarily polynomials, of weight 
$$
\wt(A_n)=\wt(B_n)=(1+n)/d
$$ 
where $d=3,2$ or 3 according as $\mathrm{I_2}(p)=\mathrm{A_2},\mathrm{B_2}$ or $\mathrm{G_2}$ (use Table 1 for weights for $x,y$ and $z$).
Therefore, for each type $\mathrm{I_2}(p)$, there exists a positive integer $n_0$ such that 
$$
\wt(A_{n_0})=\wt(B_{n_0}) =1=\wt(g_l).
$$ 
(actually, $n_0=d-1=2,1$ or $2$ according as $\mathrm{I_2}(p)=\mathrm{A_2},\mathrm{B_2}$ or $\mathrm{G_2}$).

\smallskip 
\noindent
Step 2.  
Using only the equations (5.3)$^*$ double inductively  on $A_n$ $(n\ge -a)$ and $B_n$ $(n\ge -b)$, one can determine the coefficients until the degree  $n<n_0$, where we observe two basic facts.

(1)  The initial term $(A_{-a},B_{-b})$ are constants independent of the parameter, and we have exactly $[p/2]$-number of solutions. 
More exactly, according to each initial direction condition listed in the following table \eqref{eq:initial},  there exists a unique solution satisfying it, where,  in the last case, ${\bf x}_{\mathrm{I_2}(p)}$ does not have non-trivial pole.
 The list of explicit solutions is given in  \eqref{eq:Intermediate}, where one confirm that the coefficients are $\mathbb{Q}$-coefficient  polynomials in $g_s$. The calculation is case by case and we omit details.  
\begin{equation}
\label{eq:initial}
\begin{array}{rccl}
\mathrm{A_2} & \infty_1 &:&  A_{-a} >0   \quad \text{and} \quad  B_{-b} <0  \\
\\
\mathrm{B_2}& \infty_1 &:&  A_{-a} >0   \quad \text{and} \quad  B_{-b} <0  \\
\mathrm{B_2}& \infty_2 &:&  A_{-a} <0   \quad \text{and} \quad  B_{-b} >0  \\
\\
\mathrm{G_2} & \infty_1  &:&  A_{-a} >0   \quad \text{and} \quad  B_{-b} <0  \\
\mathrm{G_2} & \infty_2  &:&  A_{-a} <0   \quad \text{and} \quad  B_{-b} <0  \\
\mathrm{G_2} & \infty_3 &:&  \quad\!\! a\le 0\ \   \quad \text{and} \quad  B_{-b} <0
\end{array}
\end{equation}
{\small } 


Already in this initial solutions level (under the assumption that they shall later extend to full solution), we can confirm the bijection \eqref{eq:Solution} by the use of Figure 2, where the ``roots" of paths $\delta_i$'s are the infinity points $\infty_i$.  
Therefore, we indicated the point at infinity in the table \eqref{eq:initial} of initial conditions and in the table \eqref{eq:Intermediate} of (partial) solutions. 

(2) {\it The second nontrivial term ($(A_1,B_1)$ for type $\mathrm{A_2}$ and $(A_0,B_0)$ for types $\mathrm{B_2}$ and $\mathrm{G_2}$) contains the variable $g_s$ non-trivial linearly.}

\bigskip
\noindent
Step 3.  
We next use the energy condition  \eqref{eq:energy} to determine the coefficients  $A_{n_0}, B_{n_0}$ (actually, the equation \eqref{eq:Hamiltonean} alone cannot determine the energy level). We confirm that $g_l$ appears non-trivially in $A_{n_0}, B_{n_0}$.  
According to the 6 initial conditions listed in \eqref{eq:initial}, the results are given below. In particular, we confirm that {\it the variable $g_l$ appears as a non-trivial linear term in $A_{n_0}$ and/or in $B_{n_0}$}.
\begin{equation}
\label{eq:Intermediate}
\begin{array}{lll}
{\bf x}_{\mathrm{A_2},\infty_1}(\z)  =  \ \ \frac{1}{4}\z^{-2} + \frac{1}{5}g_{s}\ z^2 + \frac{4}{7} g_l \z^4+ \sum_{n=3}^\infty A_nz^{2n} 
\\
{\bf y}_{\mathrm{A_2},\infty_1}(\z)   =   -\frac{1}{4}\z^{-3} + \frac{1}{5}g_{s}\z + \frac{8}{7} g_l \z^3+ \sum_{n=3}^\infty B_n \z^{2n-1}\\
\\
\vspace{0.2cm}
\\
{\bf x}_{\mathrm{B_2},\infty_1}(\z)   =  \ \  \frac{1}{2} \z^{-1}+ \frac{1}{3}g_s \z  +\big(\frac{1}{18}g_s^2-\frac{4}{5}g_l\big) \z^3 +  \sum_{n=2}^\infty A_n \z^{2n+1}  \\
\\
{\bf y}_{\mathrm{B_2},\infty_1}(\z)   = -\frac{1}{4} \z^{-2} + \frac{1}{6}g_s +\big(\frac{1}{12}g_s^2-\frac{6}{5}g_l\big) \z^2+ \sum_{n=2}^\infty B_n \z^{2n}
\\
\vspace{0.2cm}
\\
{\bf x}_{\mathrm{B_2},\infty_2}(\z) = -\frac{1}{2} \z^{-1} - \frac{1}{3} g_s \z -(\frac{1}{18}g_s^2-\frac{4}{5}g_l) \z^3 +   \sum_{n=2}^\infty A_n \z^{2n+1}   \\
{\bf y}_{\mathrm{B_2},\infty_2}(\z)  = \quad \frac{1}{4} \z^{-2} - \frac{1}{6} g_s - (\frac{1}{12} g_s^2- \frac{6}{5}g_l) \z^2 +  \sum_{n=2}^\infty B_n \z^{2n}   \\
\vspace{0.2cm}
\\
{\bf x}_{\mathrm{G_2},\infty_1}(\z)   =  \ \ \frac{1}{2} \z^{-1} +\frac{1}{2}g_s + \frac{3}{2}g_s^2 \z + (g_s^3-\frac{1}{2}g_l) \z^2 + \sum_{n=3}^\infty A_n \z^n \\
\\
{\bf y}_{\mathrm{G_2},\infty_1}(\z)   =   -\frac{1}{2} \z^{-1}  + \frac{3}{2} g_s -  \frac{3}{2} g_s^2 \z + (3g_s^3-\frac{3}{2}g_l) \z^2 + \sum_{n=3}^\infty B_n \z^n \\
\\
\vspace{0.2cm}
\\
{\bf x}_{\mathrm{G_2},\infty_2}(\z) = -\frac{1}{2} \z^{-1} + \frac{1}{2} g_s  - \frac{3}{2}g_s^2 \z + (g_s^3-\frac{1}{2}g_l) \z^2 + \sum_{n=3}^\infty A_n \z^n   \\
{\bf y}_{\mathrm{G_2},\infty_2}(\z)  =  -\frac{1}{2} \z^{-1} - \frac{3}{2} g_s  - \frac{3}{2} g_s^2 \z^2 -  (3g_s^3-\frac{3}{2}g_l) \z^2 + \sum_{n=3}^\infty B_n \z^n 
\\
\vspace{0.2cm}
\\
{\bf x}_{\mathrm{G_2},\infty_3}(\z) =     - g_s  \ \ + (-2 g_s^3 + g_l) \z^2 + \ \sum_{n=2}^\infty A_{2n} \z^{2n}  \\
{\bf y}_{\mathrm{G_2},\infty_3}(\z) = \quad \z^{-1}  + \qquad 3g_s^2\z  \qquad + \ \sum_{n=2}^\infty B_{2n-1} \z^{2n-1}
\end{array}
\end{equation}

\smallskip 
\noindent
Step 4.
To determine the coefficients $A_n$ and $B_n$ for $n>n_0$, we use again the equation \eqref{eq:Hamiltonean}. By inserting \eqref{eq:Intermediate} in (5.3)$*$,  compare coefficients of the Laurent expansions in BHS. Let $C_n$ (resp.\ $D_n$) be the coefficient polynomial of the power of $\z$ in the RHS of \eqref{eq:Hamiltonean} whose degree coincides with the term $A_n$ (resp.\ $B_n$) in LHS.  The $C_n$ and $D_n$ are rational coefficients weighted homogenous polynomials in $g_s$, $g_l$ and $A_m,B_m$ ($m\in\mathbb{Z}_{>n_0}$). Since the total weight of $C_n$ (resp.\ $D_n$) is equal to $\wt(A_n)=\wt(B_n)$, the coefficient $C_n$ (resp.\ $D_n$) cannot contain $A_m,B_m$ for $m>n$ and $A_n$ and $B_n$ appear only linearly with constant coefficients. The linear coefficients are independent of $n$, since such terms appear in the expansion of RHS of (5.3)$*$ only when the term $A_n$ or $B_n$ multiplied with the constant coefficient terms of $x_{\mathrm{I_2}(p)}$ and $y_{\mathrm{I_2}(p)}$, that is lowest degree terms $A_{-a}  z^{-a}$ and $B_{-b}z^{-b}$ (see \eqref{eq:Intermediate}). But, such pattern does not depends on $n\in \mathbb{Z}_{>n_0}$. Calculating explicitly the linear coefficients, we obtain the equations:
\begin{equation}
\label{eq:recurence}
\begin{array}{clll}
\mathrm{A_2} \text{ type } : &  2n A_n =   2 B_n  +C'_n,   &  (2n-1)B_n =   6 A_n  + D'_n \!\!\! \\
\\
\mathrm{B_2} \text{ type }  : &  (2n+1) A_n =   2  B_n  +C'_n,   & 2n B_n =  3 A_n + D'_n  \\
\\
\mathrm{G_2} \text{ type }  : &  n A_n =   -A_n +  B_n  +C'_n,  &  nB_n =   3A_n + B_n + D'_n  \\
\end{array}
\end{equation}
where $C'_n$ and $D'_n$ are the remaining part of $C_n$ and $D_n$ after subtracting the linear terms in $A_n$ and $B_n$.
We observe immediately that the determinant of coefficients of $A_n$ and $B_n$ in the two equation for the three types are given by
{\small $$
\begin{array}{rcrcrl}
\mathrm{A_2}:\quad &\!\! \det\!
\begin{bmatrix}
 2n \!\! & \!\! -2\\
 -6 \!\!&\! 2n\!-\!1
 \end{bmatrix}  &
\mathrm{B_2}:\quad &\!\! \det\!
\begin{bmatrix}
 2n\!+\!1\! \!&\!\! -2\\
 -3 \!\!&\!\! 2n
 \end{bmatrix} &
\mathrm{G_2}: \quad &\!\!
\det\!
\begin{bmatrix}
 n\!+\!1\! \!\!&\!\! -1\\
 -3 \!\!&\! \! n\!-\!1
 \end{bmatrix}  \\
\qquad =\!\! &\!2(2n\!+\!3)(n\!-\!2), &
\qquad =\!\! &\! 2(2n\!-\!3)(n\!+\!2), &
\qquad = \!\! &\! (n\!+\!2)(n\!-\!2)
\end{array}
$$
}
which takes positive values for $n>n_0$. Thus $A_n$ and $B_n$ are uniquely  expressed as rational coefficients polynomials in $A_m, B_m$ ($n_0\le m<n$) and $g_s$. This gives the inductive construction of the coefficients $A_n$ and $B_n$ ($n\in \mathbb{Z}_{n_0}$).

This completes the proof of Lemma \ref{eq:LaurentSolution} .
\end{proof}

As a result of Lemma \ref{eq:LaurentSolution}, we can determine the principal parts of the Laurent expansions of the meromorphic functions $x_{\mathrm{I_2}(p)}(z,\underline{g})$ and $y_{\mathrm{I_2}(p)}(z,\underline{g})$ \eqref{eq:pfunction}.  

\medskip
\noindent
{\bf Fact 16.}  1.{ \it The formal Laurent series solutions in Lemma 6.1 are convergent.}

2. {\it The following substitutions of $\z$ in the formal series solution \eqref{eq:Intermediate}: 
\begin{equation}
\label{eq:principalpart}
\begin{array}{ccc}
{\bf x}_{\mathrm{I_2}(p),\infty_i}(z-(\frac{i-1}{[p/2]} \omega_{\gamma_0} + \omega_\gamma)) , \quad 
{\bf y}_{\mathrm{I_2}(p),\infty_i}(z-(\frac{i-1}{[p/2]} \omega_{\gamma_0} + \omega_\gamma))
\end{array}
\end{equation}
give the Laurent  expansions of the meromorphic functions $x_{\mathrm{I_2}(p)}(z,\underline{g})$ and $y_{\mathrm{I_2}(p)}(z,\underline{g})$ at the place $\frac{i-1}{[p/2]} \omega_{\gamma_0} + \omega_\gamma$ for any $1\le i\le [p/2]$ and $\gamma\in \tilde{L}/\rad(\tilde{L})$.}
\begin{proof} 1. can be shown as a consequence of the next 2. 

2.  We know already from geometry (recall a discussion after the definition \eqref{eq:pfunction} and the description of Fact 15) that the functions 
$x_{\mathrm{I_2}(p)}(z,\underline{g})$ and $y_{\mathrm{I_2}(p)}(z,\underline{g})$ may have poles only at the places 
$\frac{i-1}{[p/2]} \omega_0 + \omega_\gamma$ for $i=1,\cdots,[p/2]$ and $\gamma \in  \tilde{L}/\rad(\tilde{L})$. In view of the asymptotic behavior of paths $\delta_i$'s at their starting points  in the first row of Figure 1,  we observe that all of them become poles except that the function  $x_{\mathrm{G_2}}$ at the places $\frac{2}{3} \omega_0 + \omega_\gamma$ for any $\gamma\in  \tilde{L}/\rad(\tilde{L})$) does not have a pole

Obviously, the Laurent expansions at those places should satisfy the equations \eqref{eq:Hamiltonean} together with the constraint \eqref{eq:energy}, which further satisfy the initial constraint \eqref{eq:initial} according to its location. Then, the uniqueness of the solution of the  equations under the initial constraint  implies that the formal solution should coincide with the expansion of $x_{\mathrm{I_2}(p)}(z,\underline{g})$ or $y_{\mathrm{I_2}(p)}(z,\underline{g})$. 
\end{proof}

\medskip
We  note that the proof of Lemma 6.1  actually covers also the Laurent series expansions of the cases for $\underline{g} \in D_{\mathrm{I_2}(p)}$. The \eqref{eq:initial} express the Laurent series expressions of the coordinate $(x,y)$ for the cases of $\underline{g} \in D_{\mathrm{I_2}(p)}$ (recall Footnote 15 and 17.) as trigonometric or rational functions in $z$.

\smallskip
Let us give explicitly the first few terms of the Laurent expansion of  $x_{\mathrm{I_2}(p)}(z,\underline{g})$ and $y_{\mathrm{I_2}(p)}(z,\underline{g})$ at the origin $z=0$ as follows. 

\bigskip
\noindent
{$\mathrm{A_2}$ case.}\vspace{-0.3cm}
$$
\begin{array}{lll}
x_{\mathrm{A_2}}(z)  =  \frac{1}{4}z^{-2} + \frac{1}{5}g_{s}z^2 + \frac{4}{7} g_lz^4+\frac{4}{75}g_s^2z^6+ \frac{48}{385} g_sg_l z^8 + 
 \cdots 
\\
 y_{\mathrm{A_2}}(z)   =   -\frac{1}{4}z^{-3} + \frac{1}{5}g_{s}z + \frac{8}{7} g_lz^3+ \frac{4}{25} g_s^2 z^5 + \frac{192}{385} g_s g_l z^7 + \cdots\\
\end{array}
$$ 

\noindent
{$\mathrm{B_2}$ case.}
\vspace{-0.3cm}
$$
\label{eq:pzeta}
\begin{array}{lcl}
x_{\mathrm{B_2}}(z)   =    \frac{1}{2}z^{-1}+ \frac{1}{3}g_s z  +\big(\frac{1}{18}g_s^2-\frac{4}{5}g_l\big)z^3 +  \big(\frac{1}{27}g_s^3-\frac{8}{35}g_s g_l\big) z^5 +\cdots  \\
y_{\mathrm{B_2}}(z)   = -\frac{1}{4} z^{-2} + \frac{1}{6}g_s +\big(\frac{1}{12}g_s^2-\frac{6}{5}g_l\big)z^2+ \big(\frac{5}{54}g_s^3-\frac{4}{7}g_s g_l\big) z^4 + \cdots
\end{array}
$$ 

\noindent
{$\mathrm{G_2}$ case.}
\vspace{-0.3cm}
$$
\label{eq:derzeta}
\begin{array}{lcl}
x_{\mathrm{G_2}}(z)   =   \frac{1}{2}z^{-1} +\frac{1}{2}g_s + \frac{3}{2}g_s^2 z + (g_s^3-\frac{1}{2}g_l)z^2 + (\frac{3}{2}g_s^4-\frac{6}{5}g_lg_s)z^3 +\cdots \\
y_{\mathrm{G_2}}(z)   =   -\frac{1}{2} z^{-1}  + \frac{3}{2} g_s -  \frac{3}{2} g_s^2 z + (3g_s^3-\frac{3}{2}g_l)z^2 - (\frac{3}{2}g_s^4-\frac{6}{5}g_lg_s)z^3 +\cdots \\
\end{array}
$$ 

\medskip
{\rm\large \qquad Table 2:\quad {\normalsize }}Laurent expansions of $x_{\mathrm{I_2}(p)}$ and 
 $y_{\mathrm{I_2}(p)}$ at $z=0$.\!\!}

\section{Partial fractional expansions}
 We come back to  the global study of the meromorphic functions  $ x_{\mathrm{I_2}(p)}(z,\underline{g})$ and $y_{\mathrm{I_2}(p)}(z,\underline{g})$ in $z$ \eqref{eq:pfunction}. The goal of this section is to give the partial fractional expansion of them. The study belongs to classical elliptic function theory. In particular, $\mathrm{A_2}$-type case is well-known as Weierstrass $\mathfrak{p}$-function theory. We generalize it for  the other two types $\mathrm{B_2}$ and $\mathrm{G_2}$, since those descriptions in the present section, lead to ``generalized Eisenstein series" expression of the modular forms for the congruence subgroups $\Gamma_1(2)$ and $\Gamma_1(3)$ in the next section, which seems to have be unknown (see also Remark \ref{exceptionalEisenstein} in \S9).  

We recall the classical Weierstrass's $\mathfrak{p}$-function and $\zeta$-function associated with any point $(\omega_0,\omega_1)\in \widetilde{\mathbb{H}}$ as meromorphic functions on the $z$-plane with double or simple poles (see, e.g.\ \cite{H-C}).
$$
\begin{array}{rcl}
\mathfrak{p}(z) & = & \frac{1}{z^2} + \underset{\omega\not=0 \in \Omega}{\sum} \Big(\frac{1}{(z-\omega)^2} -\frac{1}{\omega^2} \Big)\\
\zeta(z)  &= & \frac{1}{z} + \underset{\omega \not=0 \in \Omega}{\sum} \Big(\frac{1}{z-\omega} +\frac{1}{\omega} +\frac{z}{\omega^2} \Big),\\
\end{array}
$$
where we set 
$
\Omega  := \mathbb{Z} \omega_0  +  \mathbb{Z} \omega_1
$. 
Since they are compact uniform convergent on $\mathbb{C}\times\widetilde{\mathbb{H}}\setminus \cup_{m,n\in\mathbb{Z}}\{z-(m\omega_1+n\omega_0)=0\}$ , one may derivate them termwisely.   In particular, one has the well-known relation:  
$
\zeta'(z) \ = \ -\frak{p}(z) . 
\footnote{The notation ``\ $'$ \ " or ``\ $''$ \ " shall mean  single or  twice derivative with respect to $z$.
}
$

\begin{theorem}
\label{Fractional}
{\it The meromorphic functions  $ x_{\mathrm{I_2}(p)}(z,\underline{g})$ and $y_{\mathrm{I_2}(p)}(z,\underline{g})$  have the following partial fractional expansions.

\begin{equation}
\label{eq:fractionalexpansion}
\begin{array}{lll}
x_{\mathrm{A_2}}(z,\underline{g})=  \frac{1}{4} \mathfrak{p}(z)  \\
y_{\mathrm{A_2}}(z,\underline{g})=   \frac{1}{8} \mathfrak{p}'(z)  \\
 \\
  x_{\mathrm{B_2}}(z,\underline{g})= -\frac{1}{2}\zeta(\frac{1}{2}\omega_0) + \frac{1}{2} \zeta(z) - \frac{1}{2} \zeta(z-\frac{1}{2}\omega_{\gamma_0}) \\
 y_{\mathrm{B_2}}(z,\underline{g})=  -\frac{1}{4} \mathfrak{p}(z)  + \frac{1}{4} \mathfrak{p}(z-\frac{1}{2}\omega_{\gamma_0}) \\
 \\
  x_{\mathrm{G_2}}(z,\underline{g}) = - \frac{1}{6} \zeta(\frac{1}{3}\omega_{\gamma_0}) -  \frac{1}{6} \zeta(\frac{2}{3}\omega_{\gamma_0}) + \frac{1}{2} \zeta(z) - \frac{1}{2} \zeta(z-\frac{1}{3}\omega_{\gamma_0})  \\
  y_{\mathrm{G_2}}(z,\underline{g}) = \frac{1}{2} \zeta(\frac{1}{3}\omega_{\gamma_0}) + \frac{1}{2}  \zeta(\frac{2}{3}\omega_{\gamma_0}) -\frac{1}{2} \zeta(z) -\frac{1}{2} \zeta(z-\frac{1}{3}\omega_{\gamma_0}) +   \zeta(z-\frac{2}{3}\omega_{\gamma_0}) \!\! \!\! \!\! \!\! \!\!
 \end{array}
 \end{equation}
where the $\mathfrak{p}$-function and the $\zeta$-function in RHS are those associated with the period $(\omega_{\gamma_0},\omega_{\gamma_1})\in \widetilde{\mathbb{H}}$, and, hence, with the period lattice $\Omega=\Omega_{\widetilde{L}/\rad(\widetilde{L}),\underline{g}}$ \eqref{eq:periodlattice}.}
\end{theorem}

\begin{proof}
Owing to \eqref{eq:Intermediate} and  Fact 16, we know already the principal parts of poles of the functions  $ x_{\mathrm{I_2}(p)}(z,\underline{g})$ and $y_{\mathrm{I_2}(p)}(z,\underline{g})$ in $z$. Since
the principal parts of poles of $\mathfrak{p}(z)$, $\mathfrak{p}'(z)$ and $\zeta(z)$ are
$1/(z-\omega)^2$,  $-2/(z-\omega)^3$ and $1/(z-\omega)$ for $\omega \in \Omega_{\mathrm{I_2}(p),\underline{g}}$,  respectively, we see that the sum in the bracket of RHS of  the following \eqref{eq:principal} give meromorphic functions on $z$, whose principal parts of poles coincide with those of $ x_{\mathrm{I_2}(p)}(z,\underline{g})$ and $y_{\mathrm{I_2}(p)}(z,\underline{g})$ in $z$, respectively (here, we denote by $\big[f(z)\big]$ the set of all principal parts of a meromorphic function $f(z)$ defined on $z$-plane).
\begin{equation}
\label{eq:principal}
\begin{array}{lll}
\big[x_{\mathrm{A_2}}(z,\underline{g})\big]=  \big[ \frac{1}{4} \mathfrak{p}(z) \big]  \\
\big[ y_{\mathrm{A_2}}(z,\underline{g}) \big] =   \big[\frac{1}{8} \mathfrak{p}'(z) \big] \\
 \\
  \big[  x_{\mathrm{B_2}}(z,\underline{g}) \big] = \big[ \frac{1}{2} \zeta(z) - \frac{1}{2} \zeta(z-\frac{1}{2}\omega_{\gamma_0}) \big]\\
\big[ y_{\mathrm{B_2}}(z,\underline{g}) \big] =  \big[ -\frac{1}{4} \mathfrak{p}(z)  + \frac{1}{4} \mathfrak{p}(z-\frac{1}{2}\omega_{\gamma_0}) \big]
 \\
 \\
\big[ x_{\mathrm{G_2}}(z,\underline{g}) \big] = \big[ \frac{1}{2} \zeta(z) - \frac{1}{2} \zeta(z-\frac{1}{3}\omega_{\gamma_0}) \big] \\
  \big[ y_{\mathrm{G_2}}(z,\underline{g}) \big] =  \big[ -\frac{1}{2} \zeta(z) -\frac{1}{2} \zeta(z-\frac{1}{3}\omega_{\gamma_0}) +   \zeta(z-\frac{2}{3}\omega_{\gamma_0}) \big] \\
 \end{array}
 \end{equation}

On the other hand, we remark that the functions in the bracket of the RHS of \eqref{eq:principal} are periodic functions w.r.t.\ the period lattice  $\Omega_{\mathrm{I_2}(p),\underline{g}}$, since (1) the functions $\mathfrak{p}(z)$ and $\mathfrak{p}'(z)$ are already periodic, and (2) the sum of coefficients of the linear combinations of the functions of the form $\zeta(z+*)$ in each formula is equal to zero and, then, it is well known that the linear combination is a periodic function (see, e.g.\ \cite{H-C} Ch.1\S12, this follows from an elementary property of zeta function that $\zeta(z+m\omega_0+n\omega_1)-\zeta(z)=m2\zeta(\omega_0/2)+n2\zeta(\omega_1/2)$ for $\omega \in \Omega$). Thus, due to Liouville's Theorem, the difference of meromorphic functions in the brackets of both hand sides of \eqref{eq:principal} are constants. 

Actually, the data of the principal parts of poles are not sufficient to control  the ambiguity of adding constant terms except for the case of type $\mathrm{A_2}$. Namely, the Laurent expansion for the type $\mathrm{A_2}$ at $z=0$ (see Table 2 at the end of \S7) does not have constant terms but those for the other types $\mathrm{B_2}$ and $\mathrm{G_2}$ contain non-trivial constant terms, which are linear in $g_s$ and which are still to be determined from the data of the lattice $\Omega_{\mathrm{I_2}(p)}$. In order to overcome this issue, we use the Hamilton equation \eqref{eq:Hamiltonean} of the motion. This is essentially new feature to be cautious  compared with the classical case of type $\mathrm{A_2}$.

Let us determine the constants depending on the type separately.

\medskip 
\noindent
{\bf $\mathrm{A_2}$ type case:}
Since the constant terms of the Laurent expansions at the origin of BHS are zero, the difference is zero, and we obtain already \eqref{eq:fractionalexpansion} for type $\mathrm{A_2}$. 

\medskip 
\noindent
{\bf $\mathrm{B_2}$ type case:}  
Set, for suitable constants (w.r.t.\ $z$) $A$ and $B$,
\vspace{-0.1cm}
$$
\begin{array}{rcl}
 x_{\mathrm{B_2}}(z,\underline{g}) & = & A+ \frac{1}{2} \zeta(z) - \frac{1}{2} \zeta(z-\frac{1}{2}\omega_{\gamma_0}) \\
 y_{\mathrm{B_2}}(z,\underline{g}) & = & B -\frac{1}{4} \mathfrak{p}(z)  + \frac{1}{4} \mathfrak{p}(z-\frac{1}{2}\omega_{\gamma_0}) 
\end{array}
$$
In the first equality, since the constant terms of the Laurent expansions of  $x_{\mathrm{B_2}}(z,\underline{g})$ and $ \zeta(z)$ are zero (recall \eqref{eq:Intermediate} and the fact that $\zeta(z)$ is an odd function) the sum of the remaining terms $A- \zeta(0-\frac{1}{2}\omega_{\gamma_0})$ is equal to zero. This determine $A=- \zeta(\frac{1}{2}\omega_{\gamma_0})$. 

For the second row of the equality, recall the Hamilton's equation of the motion \eqref{eq:Hamiltonean} {\normalsize $\frac{\partial  x_{\mathrm{B_2}}}{\partial z}$} $= 2 y_{\mathrm{B_2}}$. The LHS is equal to $-\frac{1}{2} \mathfrak{p}(z)  + \frac{1}{2} \mathfrak{p}(z-\frac{1}{2}\omega_{\gamma_0}) $, and substituting $y_{\mathrm{B_2}}$ in the RHS, we see that $2B=0$. These already gives \eqref{eq:fractionalexpansion} for type $\mathrm{B_2}$. Then, by comparing the constant terms of the Laurent expansion of the BHS of \eqref{eq:fractionalexpansion} in view of \eqref{eq:Intermediate}, we obtain $\frac{1}{6}g_s=\frac{1}{4}\mathfrak{p}(0-\frac{1}{2}\omega_{\gamma_0})$. Hence,
\begin{equation}
\begin{array}{rcl}
\label{eq:gsforB2}
A &= & -\zeta(\frac{1}{2}\omega_{\gamma_0})\\
B & = & 0 \\
g_s & = & \frac{3}{2}\ \mathfrak{p}(\frac{1}{2}\omega_{\gamma_0}).
\end{array}
\end{equation}

\medskip 
\noindent
{\bf $\mathrm{G_2}$ type case:}
Set, for suitable constants w.r.t.\ $z$, $A$ and $B$,
\vspace{-0.1cm}
$$
\begin{array}{rcl}
 x_{\mathrm{G_2}}(z,\underline{g}) & = & A+  \frac{1}{2} \zeta(z) - \frac{1}{2} \zeta(z-\frac{1}{3}\omega_{\gamma_0}) \\
 y_{\mathrm{G_2}}(z,\underline{g}) & = & B -\frac{1}{2} \zeta(z) -\frac{1}{2} \zeta(z-\frac{1}{3}\omega_{\gamma_0}) +   \zeta(z-\frac{2}{3}\omega_{\gamma_0}) .
\end{array}
$$
Recall the Hamilton's equation of the motion \eqref{eq:Hamiltonean} 
$$
\frac{\partial  x_{\mathrm{G_2}}}{\partial z}= 2 x_{\mathrm{G_2}}y_{\mathrm{G_2}} + 2 g_s y_{\mathrm{G_2}}. 
$$
The LHS is equal to $-\frac{1}{2} \mathfrak{p}(z)  + \frac{1}{2} \mathfrak{p}(z-\frac{1}{3}\omega_{\gamma_0}) $ so the residue at any pole of LHS is equal to zero. Thus, we obtain two relations that the residues at  at $z=\frac{2}{3}\omega_{\gamma_0}$ and at 
 at $z=\frac{1}{3}\omega_{\gamma_0}$  of the meromorphic function
$$
\begin{array}{cl}
&\!\!\! 2\ \big(A+  \frac{1}{2} \zeta(z) - \frac{1}{2} \zeta(z-\frac{1}{3}\omega_{\gamma_0})\big)
\times
\big(B -\frac{1}{2} \zeta(z) -\frac{1}{2} \zeta(z-\frac{1}{3}\omega_{\gamma_0}) +   \zeta(z-\frac{2}{3}\omega_{\gamma_0}) \big)\\
+&\!\!\!
 2\ g_s \big(B -\frac{1}{2} \zeta(z) -\frac{1}{2} \zeta(z-\frac{1}{3}\omega_{\gamma_0}) +   \zeta(z-\frac{2}{3}\omega_{\gamma_0}) \big)
\end{array}
$$
are zero.

\noindent
(1) Residue at $z=\frac{2}{3}\omega_{\gamma_0}$: \ 
$
2A +   \zeta(\frac{2}{3}\omega_{\gamma_0})  -  \zeta(\frac{2}{3}\omega_{\gamma_0}-\frac{1}{3}\omega_{\gamma_0}) +2g_s=0
$

\noindent
(2) Residue at $z=\frac{1}{3}\omega_{\gamma_0}$: \ 
$ -A -\frac{1}{2}  \zeta(\frac{1}{3}\omega_{\gamma_0}) -B+ \frac{1}{2}  \zeta(\frac{1}{3}\omega_{\gamma_0})  -  \zeta(\frac{1}{3}\omega_{\gamma_0}-\frac{2}{3}\omega_{\gamma_0}) -g_s =0
$

In addition to them, let us consider two more relations: 

\noindent
(3) \quad $ \frac{1}{2}g_s = A - \frac{1}{2} \zeta(0-\frac{1}{3}\omega_{\gamma_0}) $

\noindent
(4) \quad $ \frac{3}{2}g_s = B -\frac{1}{2} \zeta(0-\frac{1}{3}\omega_{\gamma_0}) +   \zeta(0-\frac{2}{3}\omega_{\gamma_0})$

\noindent
obtained by comparing the constant terms of Laurent expansion of the equalities: 
$x_{\mathrm{G_2}} =  A+  \frac{1}{2} \zeta(z) - \frac{1}{2} \zeta(z-\frac{1}{3}\omega_{\gamma_0}) $ 
and 
$y_{\mathrm{G_2}} =B - \frac{1}{2} \zeta(z) -\frac{1}{2} \zeta(z-\frac{1}{3}\omega_{\gamma_0}) +   \zeta(z-\frac{2}{3}\omega_{\gamma_0})$. 

Recalling the fact that $\zeta$ is an odd function, we see that  (1), (2), (3) and (4) are overdetermined system for $A,B$ and $g_s$, and we obtain the solution:  
\begin{equation}
\begin{array}{rcl}
\label{eq:gsforG2}
A & = &- \frac{1}{6}\zeta(\frac{1}{3}\omega_{\gamma_0})  - \frac{1}{6}  \zeta(\frac{2}{3}\omega_{\gamma_0})\\
B & = &  \ \ \frac{1}{2}\zeta(\frac{1}{3}\omega_{\gamma_0}) + \frac{1}{2}  \zeta(\frac{2}{3}\omega_{\gamma_0}) \\
g_s  & = & \ \  \frac{2}{3}  \zeta(\frac{1}{3}\omega_{\gamma_0}) - \frac{1}{3}  \zeta(\frac{2}{3}\omega_{\gamma_0})
\end{array}
\end{equation}

This completes the proof of Theorem 7.1.

\end{proof}

\begin{rem}
\label{abstractinverse}
 In order to get the equality \eqref{eq:fractionalexpansion}, we have substituted the lattice $\Omega$ in the RHS by the period lattice $\Omega_{\mathrm{I_2}(p),\underline{g}}$. However, the expression in the RHS of \eqref{eq:fractionalexpansion} is defined in a self-contained manner for any point $(\omega_0,\omega_1)$ in $\widetilde{\mathbb{H}}$. Therefore, we shall hereafter regard  RHS of \eqref{eq:fractionalexpansion} as meromorphic functions in $z$ which are holomorphically parametrized by $\widetilde{\mathbb{H}}$, where the holomorphicity follows from the compact uniform convergences of the series $\mathfrak{p}$ and $\zeta$ also in the variable $(\omega_0,\omega_1)\in \widetilde{\mathbb{H}}$, regardless whether it is in the image of the period map or not. 
 \end{rem}

\section{ Eisenstein series of type  $\mathrm{A_2, B_2}$ and $\mathrm{G_2}$
\qquad\qquad \qquad\qquad \qquad       {\small - (Primitive automorphic forms)}
}

We come back to the solve the inversion problem posed at Theorem \ref{perioddomain}. For the purpose, we use  some  generalizations of Eisenstein series to obtain inversion maps (see Theorem \ref{primitiveautomorphicform}). For type $\mathrm{A_2}$, this is classically well established theory. Our interest is to show that a generalization of the theory works for types $\mathrm{B_2}$ and $\mathrm{G_2}$  (which is the first main goal of the present paper).

\begin{definition} 
\label{Eisenstein}
For each type $\mathrm{I_2}(p)$, the coefficients of the Laurent series expansion at $z=0$ of the meromorphic functions in RHS of \eqref{eq:fractionalexpansion}, as a weighted homogeneous holomorphic functions on $(\omega_0,\omega_1)\in \widetilde{\mathbb{H}}$, shall be called  Eisenstein series of type $\mathrm{I_2}(p)$. 
\end{definition}

In the following, we determine explicitly all Eisenstein series of type $\mathrm{I_2}(p)$. However,  such explicit description is un-necessary to solve the inversion problem. So, some readers may skip the present paragraph till Theorem \ref{primitiveautomorphicform}. However, the explicit description are unavoidably important in \S10, when we study the Fourier expansions of the polynomials in $\mathbb{C}[g_s,g_l]$ as modular forms.

Set $\Omega:=\mathbb{Z}\omega_0+\mathbb{Z}\omega_1$ for $(\omega_0,\omega_1)\in \widetilde{\mathbb{H}}$. Depending on  $m\in \mathbb{Z}_{\ge3}$ and $a\in \mathbb{R}\omega_0+ \mathbb{R} \omega_1=\mathbb{\R}\otimes_{\mathbb{Z}}\Omega$, let us consider series: 
\begin{equation}
\label{eq:classicalEisenstein}
G_{m}(a)
:=
{\small
\left\{
\begin{matrix}
 \underset{\omega \in \Omega\setminus\{0\}}{ \sum} \omega^{-m}
 &= 
 & \frac{1}{(m-1)!}\frac{d^{m-2}(\mathfrak{p}-z^{-2})}{dz^{m-2}}(0)   
 &\\
&= &-\frac{1}{(m-1)!}\frac{d^{m-1}(\zeta-z^{-1})}{dz^{m-1}}(0) &(\text{if } a\in \Omega)\\
 \underset{\omega \in \Omega}{ \sum}( \omega+a)^{-m} 
 &=
 & \frac{1}{(m-1)!}\frac{d^{m-2}\mathfrak{p}}{dz^{m-2}}(-a) &\\
 &= 
 &-\frac{1}{(m-1)!}\frac{d^{m-1}\mathfrak{\zeta}}{dz^{m-1}}(-a)
 & (\text{if }a\not\in \Omega) \ .\\
\end{matrix}
\right. }
\end{equation}
The first series for $a\in\Omega$ are the classical well-known classical {\it Eisenstein series of weight $m$} (see, e.g. [H-C,E-Z]). However, the second series for $a\notin \Omega$ seem to have not appeared in literature. As we shall see, since both behaves in parallel to the classical series, we shall call the  latter case {\it shifted classical Eisenstein series of weight $m$}.

It is absolute and locally uniformly convergent so that  defines a holomorphic function on $\widetilde{\mathbb{H}}$ of weight $-m\cdot \wt(z)$\footnote
{We should be cautious about the use of the terminology ``weight".  The weight $-m\cdot \wt(z)$ of $G_m(a)$ as a function on $\widetilde{\mathbb{H}}$ comes from the $\mathbb{C}^\times$-action (recall Fact 14).  It is  proportional to the weight $m$ as the Eisenstein series, but depends on the factor $\wt(z)$ which 
depends on type $\mathrm{I_2}(p)$ (recall Table 1) 
 (c.f.\ \eqref{eq:identification} and Table 2).
}
parametrized by $a\in (\mathbb{R}\omega_0+\mathbb{R}\omega_1)/\Omega$, such that
 $G_m(a)=(-1)^mG_m(-a)$.  In particular, we have the relations: 
$$
\begin{array}{rll}
G_m(a)  = 0 \qquad    \text{ for } a\in\frac{1}{2}\Omega \text{ and  } m=odd.  \\
\end{array}
$$

Using \eqref{eq:classicalEisenstein}, one get the following Laurent and Tayler expansions. The first two lines are standard (e.g.\ \cite{H-C}), and the latter two for $a\in\mathbb{R}\otimes\Omega\setminus \Omega$ can be shown similarly.\footnote
{Actually, the last equality has meaning for the periodic variable $a$, even though the zeta function is not periodic, since $\zeta$ is  still ``semi-periodic" (see \cite{H-C}\S11).
}
\begin{equation}
\begin{array}{ccl}
\mathfrak{p}(z) &=&  z^{-2}+\sum_{n=1}^\infty (2n+1) z^{2n}G_{2n+2}(0) \\
\\
\zeta(z) &=&  z^{-1}-\sum_{n=1}^\infty z^{2n+1}G_{2n+2}(0) \\
\\
\mathfrak{p}(z-a) &=& \mathfrak{p}(a)+\sum_{m=1}^\infty (m+1) z^{m}G_{m+2}(a) \\
\\
\mathfrak{\zeta}(z-a) &=&- \zeta(a)-\mathfrak{p}(a)z-\sum_{m=1}^\infty  z^{m+1}G_{m+2}(a) \\

\end{array}
\end{equation}


Now, let us describe Eisenstein series for each type $\mathrm{I_2}(p)$ separately. The calculation is  straight forward from the formula \eqref{eq:fractionalexpansion}, and we omit details of them.

\medskip
\noindent
{\bf $\mathrm{A_2}$ type:}  \  
Set
\begin{equation}
\label{eq:EisensteinA21}
\begin{array}{rcl}
x_{\mathrm{A_2}}(z)  &=  & \ \frac{1}{4}z^{-2} + \sum_{n=1}^\infty A_n z^{2n} \\
 y_{\mathrm{A_2}}(z)  & = &  -\frac{1}{4}z^{-3} +  \sum_{n=1}^\infty B_n z^{2n-1}\\
\end{array}
\end{equation}
Then, we have
\begin{equation}
\label{eq:EisensteinA22}
\begin{array}{ccl}
A_n &= & \frac{2n+1}{4}G_{2n+2}(0)  \quad (n\ge1)\\
B_n &=&  \frac{(2n+1)n}{4}G_{2n+2}(0)  \quad (n\ge1)\\
\end{array}
\end{equation}

\medskip
\noindent
{\bf $\mathrm{B_2}$ type: }\    
Set
\begin{equation}
\label{eq:EisensteinB21}
\begin{array}{rcl}
x_{\mathrm{B_2}}(z)  & =  &  \frac{1}{2} z^{-1}+  \sum_{n=0}^\infty A_n z^{2n+1}  \\
y_{\mathrm{B_2}}(z)  & =  & -\frac{1}{4} z^{-2} + \sum_{n=0}^\infty B_n z^{2n}
\end{array}
\end{equation}
Then, we have
\begin{equation}
\label{eq:EisensteinB22}
\begin{array}{ccl}
A_0 & = &\frac{1}{2}\mathfrak{p}(\frac{1}{2}\omega_0)\\  
\\
A_n & = & -\frac{1}{2}G_{2n+2}(0)+\frac{1}{2}G_{2n+2}(\frac{1}{2}\omega_0) \ \ (n\ge1) \\
\\
B_0 & = & \frac{1}{4}\mathfrak{p}(\frac{1}{2}\omega_0)  \\
\\
B_n & =  &-\frac{2n+1}{4}G_{2n+2}(0)+\frac{2n+1}{4}G_{2n+2}(\frac{1}{2}\omega_0)  \ \ (n\ge1) \\
\end{array}
\end{equation}

\medskip
\noindent
$\mathrm{G_2}$ type:   \ 
Set
\begin{equation}
\label{eq:EisensteinG21}
\begin{array}{rcl}
x_{\mathrm{G_2}}(z)  & = & \ \ \frac{1}{2}z^{-1} +\sum_{n=0}^\infty A_n z^n \\
y_{\mathrm{G_2}}(z) &  = &  -\frac{1}{2} z^{-1}  +  \sum_{n=0}^\infty B_n z^n \\
\end{array}
\end{equation}
Then, we have
\begin{equation}
\label{eq:EisensteinG22}
\begin{array}{ccl}
A_0 &= & \frac{1}{3} \zeta(\frac{1}{3}\omega_{\gamma_0}) -  \frac{1}{6} \zeta(\frac{2}{3}\omega_{\gamma_0}) \ =\ \frac{1}{2} \zeta(\frac{1}{3}\omega_{\gamma_0}) - \frac{1}{3} \zeta(\frac{1}{2}\omega_{\gamma_0})\\
\\
 A_1  &= &\frac{1}{2}\mathfrak{p}(\frac{1}{3}\omega_0)  \\
 \\
 A_n& = & -\frac{1}{2}G_{n+1}(0) + \frac{1}{2}G_{n+1}(\frac{1}{3}\omega_0)  \qquad (n\ge2)\\
 \\
B_0 &= & \zeta(\frac{1}{3}\omega_{\gamma_0}) -  \frac{1}{2} \zeta(\frac{2}{3}\omega_{\gamma_0}) 
\ = \ \frac{3}{2} \zeta(\frac{1}{3}\omega_{\gamma_0}) - \zeta(\frac{1}{2}\omega_{\gamma_0}) \\
\\
B_1 & = & \frac{1}{2}\mathfrak{p}(\frac{1}{3}\omega_0)-\mathfrak{p}(\frac{2}{3}\omega_0) 
\ = \  -\frac{1}{2}\mathfrak{p}(\frac{1}{3}\omega_0)\\
\\
  B_n &=&\frac{1}{2}G_{n+1}(0)+\frac{1}{2}G_{n+1}(\frac{1}{3}\omega_0)-G_{n+1}(\frac{2}{3}\omega_0)  \\
  & = & \frac{1}{2}G_{n+1}(0)+(\frac{1}{2}+(-1)^n)G_{n+1}(\frac{1}{3}\omega_0) 
   \qquad (n\ge2) \\
\end{array}
\end{equation}

\begin{remark}
\label{exceptionalEisenstein}
1.  The infinite sequence  of Eisenstein series for each type $\mathrm{I_2}(p)$ are not algebraically independent.  More precisely, they are obeying recurrence relation \eqref{eq:recurence} in Step 4.\ of the proof of Lemma \ref{eq:LaurentSolution} (which leads to the isomorphism \eqref{eq:Eisenstein}, describing relations directly). The relations may be considered as the $\mathrm{B_2}$-type and $\mathrm{G_2}$-type generalizations of the classically well-known $\mathrm{A_2}$-type relations. However, in the present paper, we do not go into details of the relations.

2. We note that the first Eisenstein series $A_0,B_0$ in case of type $\mathrm{B_2}$ and the first and the second Eisenstein series   $A_0,B_0$ and $A_1,B_1$ in case of type $\mathrm{G_2}$ do not have the description using the classical series \eqref{eq:classicalEisenstein}. These exceptional behavior was caused by the fact that the classical Eisenstein series \eqref{eq:classicalEisenstein} do not converge absolutely in those low weights so that one need to make conditional  convergent series by a help of $\mathfrak{p}$-function or $\zeta$-function. This was made possible by the determination of the constant terms of fractional expansions in Theorem \ref{Fractional} using the energy condition  \eqref{eq:energy} of the Hamilton's equations of the motion.  

In a forthcoming paper \cite{A-S}, we shall study systematically those ``exceptional" Eisenstein series from a view point of modular forms.
\end{remark}

Now we are to formulate the second main theorem of the present paper. The proof is essentially  done already in previous sections so that we have only to coordinate them.

\begin{theorem}
\label{primitiveautomorphicform}
 {\it Consider the pull-back homomorphism $P_{\mathrm{I_2}(p)}^*: \mathcal{O}_{\widetilde{\mathbb{H}}} \to \mathcal{O}_{\widetilde{S}_{\mathrm{I_2}(p)}}$ from the ring of holomorphic functions on the period domain $\widetilde{\mathbb{H}}$ to that on the monodromy covering space $\widetilde{S}_{\mathrm{I_2}(p)}$ of the base space $S_{\mathrm{I_2}(p)}$ of the family \eqref{eq:family} (recall \eqref{eq:covering} for the definition of $\widetilde{S}_{\mathrm{I_2}(p)}$ and  \eqref{eq:map} for the definition of $P_{\mathrm{I_2}(p)}$). 

Then, it induces the ring isomorphism: 
\begin{equation}
\label{eq:Eisenstein}
\mathbb{Q}[\text{Eisenstein series of type $\mathrm{I_2}(p)$}] \quad  \simeq \quad  \mathbb{Q}[g_s,g_l],
\end{equation}
where LHS is the ring over $\mathbb{Q}$ generated by all Eisenstein series of type $\mathrm{I_2}(p)$ (recall Definition \ref{Eisenstein}) and RHS is the coordinate ring of the space $S_{\mathrm{I_2}(p)}$ (recall \eqref{eq:family}) generated by the flat coordinates $g_s$ and $g_l$ over  $\mathbb{Q}$.

In particular, the generators $g_s$ and $g_l$ are expressed by Eisenstein series of type $\mathrm{I_2}(p)$  as follows

\medskip
\noindent
{\bf $\mathrm{A_2}$ type:}  \  
\begin{equation}
\label{eq:A2inversion}
\begin{array}{rcl}
g_s  & = &   \frac{15}{4} \ G_4(0) \\
 g_l  & = &  \frac{35}{16}\ G_6(0) \qquad\qquad  
\end{array}
\end{equation}

\medskip
\noindent
{\bf $\mathrm{B_2}$ type:}  \  
\begin{equation}
\label{eq:B2inversion}
\begin{array}{rcl}
g_s & = & \frac{3}{2} \mathfrak{p}(\frac{1}{2}\omega_0) 
  \\
g_l & = &  \frac{5}{32}\mathfrak{p}^2(\frac{1}{2}\omega_0)+\frac{5}{8} G_4(0)- \frac{5}{8} G_4(\frac{1}{2}\omega_0)     \!\!\!\!\!\!  \!\!\!\!\!\!  \!\!\!\!\!\!  \!\!\!\!\!\!  \!\!\!\!\!\!  \!\!\!\!\!\! \\
\end{array}
\end{equation}

\medskip
\noindent
{\bf $\mathrm{G_2}$ type:}  \  
\begin{equation}
\label{eq:G2inversion}
\begin{array}{rcl}
g_s & = & \zeta(\frac{1}{3}\omega_0)-\frac{2}{3}\zeta(\frac{1}{2}\omega_0) 
  \\
  g_s^2& = & \frac{1}{3}\mathfrak{p}(\frac{1}{3}\omega_0)\\
g_l & = &  2g_s^3-G_3(\frac{1}{3}\omega_0) \\
\end{array}
\end{equation}
}
\end{theorem}
\begin{proof}
The explicit description of the function on $\widetilde{S}_{\mathrm{I_2}(p)}$ corresponding to an Eisenstein series by the pull-back morphism $P_{\mathrm{I_2}(p)}^*$  is obtained by the corresponding coefficient of the Laurent expansion at $z=0$ of the meromorphic functions $x_{\mathrm{I_2}(p)}(z,\underline{g})$ or $y_{\mathrm{I_2}(p)}(z,\underline{g})$ \eqref{eq:fractionalexpansion}. Then, it was already shown in Lemma \ref{eq:LaurentSolution} that they are rational coefficient polynomials in $g_s$ and $g_l$ (c.f.\ Table 2 at the end of \S7). This defines the homomorphism \eqref{eq:Eisenstein} from left to right.

The morphism is injective, since the period map $P_{\mathrm{I_2}(p)}$ \eqref{eq:map} is an open map between connected manifolds (Fact 12). 

The morphism is surjective, since (1) the generator $g_s$ is, up to a constant factor, given by the lowest weight Eisenstein series for each type (recall Step 2. of the proof of Lemma \ref{eq:LaurentSolution} and \eqref{eq:Intermediate}), and (2) the generator $g_l$ appear non-trivially and linearly in the coefficients $A_{n_0}$ and $B_{n_0}$ of the Laurent expansions of $x_{\mathrm{I_2}(p)}(z,\underline{g})$ and $y_{\mathrm{I_2}(p)}(z,\underline{g})$ (recall Step 3. of the proof of Lemma \ref{eq:LaurentSolution} and \eqref{eq:Intermediate}). 
\end{proof}

After the isomorphism \eqref{eq:Eisenstein}, we shall sometimes identify the ring of Eisenstein  series and the polynomial ring in $g_s$ and $g_l$.

\bigskip
\noindent
{\bf Proof of  Theorem \ref{perioddomain}}

We show by 4 steps that the period map $P_{I_2(p)}$  \eqref{eq:periodmap} is bi-holomorphic.

\smallskip
\noindent
Step 1. 
Regardless, whether an element $(\omega_0,\omega_1)\in \widetilde{\mathbb{H}}$ belongs to the image of the period map or not, let us use the Eisenstein series of type $\mathrm{I_2}(p)$ expressions \eqref{eq:A2inversion},\eqref{eq:B2inversion} and \eqref{eq:G2inversion} 
to define a holomorphic map
\begin{equation}
\label{eq:Inverse}
E=(E_s,E_l) \ : \ \widetilde{\mathbb{H}} \ \longrightarrow \ S_{\mathrm{I_2}(p)}.
\end{equation}
The equality \eqref{eq:fractionalexpansion} in Theorem \ref{Fractional} implies that the following diagram is commutative
 \begin{equation}
 \label{eq:Inversion}
 \begin{array}{rccc}
 & \widetilde{S}_{\mathrm{I_2}(p)} \quad & \!\!\!\! \overset{P_{\mathrm{I_2}(p)}}{-\!\!\!\!-\!\!\!\!-\!\!\!\!-\!\!\!\!-\!\!\!\!-\!\!\!\!\longrightarrow} \!\!\!\! &  \widetilde{\mathbb{H}}  \\
                    &\qquad \searrow      \!\!\!\!\!\!      &                     &\!\!\!\!  \swarrow _E  \qquad \\
                    &                                  &      S_{\mathrm{I_2}(p)}       &
\end{array}          
\end{equation}      
We remark  that pull back of the polynomial ring on $S_{\mathrm{I_2}(p)}$ by the morphism $E$ \eqref{eq:Inverse} induces the same isomorphism \eqref{eq:Eisenstein}, since (i) the period map $P_{\mathrm{I_2}(p)}$ is a non-trivial open map, and (ii) any algebraic dependence relation among Eisenstein series on an open domain in $\widetilde{\mathbb{H}}$ automatically extends on the whole
$\widetilde{\mathbb{H}}$ by analytic continuation, since $\widetilde{\mathbb{H}}$  is  {\it connected}.
    
\medskip
\noindent
Step 2.  Let us  show that the image of $E$ is contained in the compliment of the discriminant:  $E( \widetilde{\mathbb{H}})\subset  S_{\mathrm{I_2}(p)} \setminus  D_{\mathrm{I_2}(p)}$.  For $(\omega_1,\omega_2)\in \widetilde{\mathbb{H}}$, using RHS of \eqref{eq:fractionalexpansion}, we define global meromorphic functions $x_{\mathrm{I_2}(p)}$ and $y_{\mathrm{I_2}(p)}$, which are periodic w.r.t. the lattice $\Omega=\mathbb{Z}\omega_0+\mathbb{Z}\omega_1$. Let us see that the pair satisfies the relation \eqref{eq:Hamiltonean} together with \eqref{eq:energy}, where the parameter $\underline{g}$ is given by \eqref{eq:Inverse}.  Actually, the both hand sides give doubly periodic function of the period $\Omega=\mathbb{Z}\omega_0+\mathbb{Z}\omega_1$, where we can check they have the same principal parts of poles, and the constant term of Laurent expansions at 0 coincides.

This means that the time coordinate $z$ is given by the integral \eqref{eq:Hamilton-time} (up to a shift of a constant). That is, the image of the map $(x_{\mathrm{I_2}(p)},y_{\mathrm{I_2}(p)}$ satisfies the equation \eqref{eq:energy}.  However, if $\underline{g}$ belonged to the discriminant, then the associated curve defined by the equation \eqref{eq:equation} is a singular rational curve. The integral  \eqref{eq:Hamilton-time} (avoiding the singularity of the curve but admitting to go through points at infinity) cannot be doubly periodic (either one periodic for $\underline{g}\in D_{\mathrm{I_2}(p)}\! \setminus\! \{0\}$, or no-periodic for $\underline{g}=0$), where as the starting $(\omega_0,\omega_1)\in \widetilde{\mathbb{H}}$ generates rank 2 lattices and $x_{\mathrm{I_2}(p)}$ and $y_{\mathrm{I_2}(p)}$ are doubly periodic. A contradiction! 

\medskip
\noindent
Step 3. 
Let us show that the period map is surjective. Since $\widetilde{\mathbb{H}}$ is connected, for any point $\omega\in \widetilde{\mathbb{H}} $, consider any path, say $p$, in $\widetilde{\mathbb{H}} $ connecting $\omega$ with the image $P_{\mathrm{I_2}(p)}(\widetilde{\Gamma}_{\mathrm{I_2}(p)})$ in $\widetilde{\mathbb{H}}$ of the base point loci of $ \widetilde{S}_{\mathrm{I_2}(p)}$ (recall the definition \eqref{eq:covering}).  Then, the projection image $E(p)$ is a path in $S_{\mathrm{I_2}(p)}\setminus D_{\mathrm{I_2}(p)}$ connecting the base point loci $\Gamma_{\mathrm{I_2}(p)}$ with $E(\omega)$ (recall the commutative diagram \eqref{eq:Inversion}). Then, the monodromy lifting $\widetilde{E(p)}$ of the path $E(p)$ in the covering space $ \widetilde{S}_{\mathrm{I_2}(p)}$ is a path connecting the base point loci $\widetilde{\Gamma}_{\mathrm{I_2}(p)}$ to a point $\widetilde{E(\omega)}$ which lies over $E(\omega)$. Then, the image $P_{\mathrm{I_2}(p)}(\widetilde{E(p)})$ is a monodromy covering  in $\widetilde{\mathbb{H}}$ of the path $E(p)$  connecting $P_{\mathrm{I_2}(p)}(\widetilde{\Gamma}_{\mathrm{I_2}(p)})$ to a point $P_{\mathrm{I_2}(p)}(\widetilde{E(\omega)})$. Since $p$ is also is a monodromy covering  in $\widetilde{\mathbb{H}}$ of the same path $E(p)$  connecting the base point loci $P_{\mathrm{I_2}(p)}(\widetilde{\Gamma}_{\mathrm{I_2}(p)})$ to the point $p$. So the two end points $P_{\mathrm{I_2}(p)}(\widetilde{E(\omega)})$ and $p$ of the paths should coincide each other. In particular, $p$ is in the image of the period map. This shows also that the modular group
$\Gamma_1([p/2])$, which is the monodromy representation $\rho$-image of the fundamental group of $S_{\mathrm{I_2}(p)} \setminus  D_{\mathrm{I_2}(p)}$ (Fact 9, 2.),  acts on any fiber of the map $E$ transitively.  That is, the modular group action quotient of $\widetilde{\mathbb{H}}$ is isomorphic to the discriminant compliment:
$$
 \Gamma_1([p/2]) \diagdown \widetilde{\mathbb{H}} \quad \simeq \quad S_{\mathrm{I_2}(p)}\setminus D_{I_(2)}
 $$

\noindent
Step 4.  
Finally, let us show that the period map is injective. Since the period map is equivariant with the modular group action, it is sufficient that the modular group action on the period domain  $\widetilde{\mathbb{H}} $ is (generically) fixed point free. But this is trivially true, since the modular group is a subgroup of $\mathrm{GL}_2(\mathbb{Z})$ so that its fixed points set is thin and $\widetilde{\mathbb{H}} $ is an open subset of $\mathbb{C}^2$. Since $E$ is  a covering map, if the action modular group is fixed point free in one fiber, it is fixed point free for all fibers.

\medskip
This completes a proof of Theorem \ref{perioddomain}. 
\qquad  \qquad \qquad \qquad \qquad        {$\Box$}
 
 \begin{rem} 1.
\begin{equation}
\label{eq:}
\frac{\partial(g_s,g_l)}{\partial(\omega_0,\omega_1)} \ = \ c \Delta_{\mathrm{I_2}(p)}^{red}
\end{equation}
\end{rem}

2. In \cite{S4,S5} , we posed a general question to describe the inversion morphism to the period map defined by a primitive form.  If a function on the parameter space of the family is described in terms of the coordinates of the period domain, we call the function (and its description on the period domain) a {\it primitive automorphic form}.  In that sense, the generalized Eisenstein series of type $\mathrm{B_2}$  and $\mathrm{G_2}$ in this sections are are the first examples of primitive automorphic forms beyond the classical case of type $\mathrm{A_2}$.

\section{Ring of modular forms and Discriminant}

We identify the ring of Eisenstein series of type $\mathrm{I_2}(p)$ with the ring of modular forms of the congruence group  $\Gamma_1([p/2])$ (see \cite{A-I} for $M_*(\Gamma_1([p/2]))$). Then we confirm that the set of irreducible components of the discriminant of the family \eqref{eq:family} is in one to one correspondence with the set of cusps of the congruence group $\Gamma_1([p/2])$.

\noindent
\begin{theorem}
\label{Eisen-Modular}
 {\it The ring of Eisenstein series of type $\mathrm{I_2}(p)$ is identified with the ring of holomorphic modular forms of the congruence group $\Gamma_1([p/2])$, where the identification is given in following \eqref{eq:identification}. \footnote
{This is naturally an expected result. However, this should have been proven, since the Eisenstein series appeared in the context of the geometry of the period mapping, whereas the modular forms are defined independently by themselves. So, their coincidence is a non-trivial marvelous fact, which we need to work cautiously. 
}
\begin{equation}
\label{eq:EisensteinModular}
\mathbb{C}[\text{Eisenstein series of type $\mathrm{I_2}(p)$}] \quad  \simeq \quad  
M_*(\Gamma_1([p/2])).
\end{equation}
The correspondences of generators are given in \eqref{eq:A2theta}, \eqref{eq:B2theta} and  \eqref{eq:G2theta}.
}
\end{theorem}
\begin{proof} Proof of the theorem is divided into Steps 1-5.

\noindent
{Step 1.} 
We explain the meaning of ``identification", and fix notation. 

Recall the period domain $\widetilde{\mathbb{H}}:=\{(\omega_0,\omega_1)\in\mathbb{C}^2\mid \mathrm{Im}(\omega_1/\omega_0)>0\}$ with its homogenous coordinates $(\omega_0,\omega_1)$. We introduce the inhomogeneous coordinate
\begin{equation}
\label{eq:tau}
 \tau \quad  :=\quad \omega_1/\omega_0
 \end{equation}
 Then the natural projection $\widetilde{\mathbb{H}}\to \mathbb{H}:=\{\tau\in\mathbb{C}\mid \mathrm{Im}(\tau)>0\}$, $(\omega_0,\omega_1)\mapsto \tau:=\omega_1/\omega_0$ gives a principal $\mathbb{C}^\times$-bundle, say $(L^{\times})^{ -1}$, over $\mathbb{H}$,  which we trivialize by the morphism $\widetilde{\mathbb{H}}\simeq \mathbb{C}^\times \mathbb{H}, \ (\omega_0,\omega_1)\mapsto (\omega_0,\tau)$. The modular group $\Gamma_1([p/2])$ acts from the left on $\widetilde{\mathbb{H}}$ and hence on $L$. For $k\in\mathbb{Z}_{>0}$, a holomorphic section of the $L^k$, say $s=s(\tau)$, such that {\footnotesize $\gamma^*(s):=s\cdot \frac{a\tau+b}{c\tau+d}$} is equal to $(c\tau+d)^ks(\tau)$ for 
$\gamma=
\begin{bmatrix}
a & b\\
c & d
\end{bmatrix}
\in \Gamma_1([p/2])
$,
is called a modular form of weight $k$ of $\Gamma_1([p/2])$ in a wide sense. Then, the correspondence $s(\tau) \ \mapsto \ \omega_0^{-k} \cdot s(\tau)$ defines the ``identification": 

{\footnotesize
\begin{equation}
\label{eq:identification}
\begin{array}{cl}
&  \{\text{modular forms of weight $k$ of the  group $\Gamma_1([p/2])$ in a wide sense}\}  \\
\leftrightarrow  & \{\text{holomorphic functions on $\widetilde{\mathbb{H}}$ of weight $-k\cdot \wt(z)$ invariant by $\Gamma_1([p/2])$}\} 
\end{array}
  \end{equation}
}

Actually, we study more restricted class of modular forms which are holomorphic and taking finite values at cusps, as we explain now.

Recall \cite{Ko} that a point $x\in\mathbb{R}\cup\{\sqrt{-1}\infty\}\ (=\partial\mathbb{H})$ is called a cusp of $\Gamma_1(N)$ if it is fixed by a hyperbolic element of $\Gamma_1(N)$. The isotropic subgroup of $\Gamma_1(N)$ fixing a cusp is an infinite unipotent group, and the set of all cusps are invariant under the action of $\Gamma_1(N)$. 
%
%
For $\gamma =
\begin{bmatrix}
a & b\\
c & d
\end{bmatrix}\in \mathrm{SL_2}(\mathbb{Z})$, 
set $(s|_k\gamma)(\tau):=(\gamma^*s)(\tau)(c\tau+d)^{-k}$ (so that the modularity property of $s$ is equivalent to $s|_k\gamma=s$ for all $\gamma\in \Gamma_1(N)$). 
Let $\gamma(x)=\sqrt{-1}\infty$ for a cusp $x$ and $\gamma\in \mathrm{SL_2}(\mathbb{Z})$. For a modular form $s$ (in the wide sense), $s|_k\gamma$, as a periodic function in $\tau$, develops into a Fourier series in $\tau$. 
Then, $s$ is called a {\it holomorphic} modular form if the Fourier series consists only of non-negative powers of $q=\exp{(2\pi\sqrt{-1}\tau)}$ at all cusps of $\Gamma_1(N)$. The constant term of the Fourier series is called the {\it value} of $s$ at the cusp and denoted by $(s|_k\gamma)(\sqrt{-1}\infty)$.

\smallskip
\noindent
{Step 2.} We recall the results by  Aoki and Ibukiyama on the ring of modular forms.  In \cite{A-I}, Aoki and Ibukiyama gave a simple unified description of the graded ring of holomorphic   modular forms of $\Gamma_0(N)$ for $N=1,2,3,4$. From that description, we recover easily the ring $M_*(\Gamma_1(N))$ of holomorphic modular forms of $\Gamma_1(N)$.  Namely, according as $N=1,2$ and $3$, the ring is generated by two (algebraically independent) modular forms $e_4, e_6$ of weight 4 and 6, $\alpha_2, \beta_4$ of weight 2 and 4,  and $\alpha_1, \beta_3$ of weight 1 and 3, respectively. That is, 
\begin{equation}
\label{eq:MGamma}
M_*(\Gamma_1(1))=\mathbb{C}[e_4,e_6],\ M_*(\Gamma_1(2))=\mathbb{C}[\alpha_2,\beta_4],\ M_*(\Gamma_1(3))=\mathbb{C}[\alpha_1,\beta_3]
\end{equation}
where explicit descriptions of the generators as theta-function, in particular, the first few Fourier coefficients at the cusps are given in \cite{A-I}. 

\medskip
\noindent
{Step 3.}
We define a morphism from the left hand side to the right hand side of \eqref{eq:EisensteinModular}.
This is achieved by showing that the Eisenstein series are holomorphic at all cusps.

More explicitly, recall that, according as $N=1,2$ and 3, the number of $\Gamma_1(N)$-equivalence classes of cusps are 1, 2 and 2, whose representatives are given as follows.
{\small 
$$
\text{$\mathrm{A_2}$ type: } \sqrt{-1}\infty,\ \ 
\text{$\mathrm{B_2}$ type: } \sqrt{-1}\infty \text{ and }  0, \ \ 
\text{$\mathrm{G_2}$ type: } \sqrt{-1}\infty \text{ and }  0.
$$ 
}
Let us consider the ring of Eisenstein series of type $\mathrm{I_2}(p)$. Recall that the identification \eqref{eq:Eisenstein} was given by the composition with the period map $P_{\mathrm{I_2}(p)}$, where the period map is $\Gamma_1([p/2])$-equivariant. This means that any Eisenstein series is a $\Gamma_1([p/2])$-invariant function on $\widetilde{\mathbb{H}}$. So, by the identification \eqref{eq:identification}, it gives arise a modular form.
If we, further, show that the Eisenstein series are holomorphic at all cusps, we obtain a graded ring homomorphism from the left hand side to the right hand side of \eqref{eq:EisensteinModular}. 

\begin{lem}
{\it The modular forms associated with Eisenstein series of type $\mathrm{I_2}(p)$ ($p=1,2,3$) are holomorphic at their cusp(s).}
\end{lem}
\begin{proof} 
Let $E=$
{\tiny
$\begin{bmatrix}
1 & 0 \\
0 & 1
\end{bmatrix}
$}
and
$S=$
{\tiny
$\begin{bmatrix}
0 & 1\\
-1 & 0
\end{bmatrix}
$
}
$\! \in \mathrm{SL_2}(\mathbb{Z})$, which transforms {\small $\sqrt{-1}\infty$} and 0 to {\small $\sqrt{-1}\infty$}. 
We  show that the series associated with $s|_kE$ and/or $s|_kS$ for an Eisenstein series $s$ with the weight $k$ converges absolute uniformly in a ``neighborhood" of $\sqrt{-1}\infty$.  This is classical for the case of Eisenstein series $s=G_m(0)$ for $m\ge3$ (see \cite{F-B}), and similar proof works for Eisenstein series of the form $s=G_m(a)$ for a suitable $a\in \Omega_{\mathbb{Q}}$. In case of Eisenstein series of the form $s=\mathfrak{p}(a)$ for a suitable $a\in \Omega_{\mathbb{Q}}$ in type $\mathrm{B_2}$ and $\mathrm{G_2}$, we may either show directly the  convergence in \cite{A-S}, or alternatively, we use the expression of the $\mathfrak{p}$-function as a proportion of Jacobi forms (\cite{E-Z} Theorem 3.6) to show that the Fourier expansions at infinity consists only of positive powers. In case of  $s=\zeta(\frac{1}{3}\omega_0)-\frac{2}{3}\zeta(\frac{1}{2}\omega_0)$, it will be shown in \cite{A-S}.
\end{proof}

\medskip
\noindent
{Step 4.}
In the following table, we give the values at cusps of the additive summands in \eqref{eq:A2inversion}, \eqref{eq:B2inversion} and \eqref{eq:G2inversion}. 
Calculations depend on cases: In case of the form $G_m(0)$, it is classical (e.g.\ \cite{B-G-H-Z, F-B}).  In case of the form $\mathfrak{p}(a)$ for some $a\in \Omega_{\mathbb{Q}}$ (see \cite{A-S}). The cases of shifted series $G_4(\frac{1}{2}\omega_0)$ and 	$G_3(\frac{1}{3}\omega_0)$ at the cusp $\sqrt{-1}\infty$ are reduced to Riemann's zeta function $\zeta_R(4)$ or Dirichlet's L-function $L(3,\chi)$ of the character $\chi$ given by quadratic residues, respectively (see \cite{A-K-I} for explicit values). Their values at the cusp 0 are directly shown to be zero. In case of $\zeta(\frac{1}{3}\omega_0)-\frac{2}{3}\zeta(\frac{1}{2}\omega_0)$, 
 it will be shown in \cite{A-S}.

\smallskip
\noindent
{\bf $\mathrm{A_2}$ type:}   
\vspace{-0.1cm}
{\footnotesize
$$
\label{eq:A2cusp}
\begin{array}{lll}
\big(\omega_0^4 \cdot G_4(0)|_4E\big)(\sqrt{-1} \infty)  =  \frac{\pi^4}{45}&&
 \big(\omega_0^6\cdot G_6(0)|_6E\big)(\sqrt{-1} \infty) =  \frac{2 \pi^6}{945}
\end{array}
$$
}
\noindent
{\bf $\mathrm{B_2}$ type:}   
\vspace{-0.1cm}
{\footnotesize
$$
\begin{array}{lll}
\big(\omega_0^2\cdot  \mathfrak{p}(\frac{1}{2}\omega_0)|_2E\big)(\sqrt{-1} \infty)  =   \frac{2\pi^2}{3}
&&
\big(\omega_0^2 \cdot \mathfrak{p}(\frac{1}{2}\omega_0)|_2S\big)(\sqrt{-1} \infty)  =   - \frac{\pi^2}{3} 
\\
\\
\big(\omega_0^4 \cdot \mathfrak{p}^2(\frac{1}{2}\omega_0)|_4E\big)(\sqrt{-1} \infty)  =   \frac{4\pi^4}{9} 
&&
\big(\omega_0^4\cdot \mathfrak{p}^2(\frac{1}{2}\omega_0)|_4S\big)(\sqrt{-1} \infty)  =    \frac{\pi^4}{9}\\
\\
\big(\omega_0^4\cdot G_4(0)|_4E\big)(\sqrt{-1} \infty)  =  \frac{\pi^4}{45}
&&
\big(\omega_0^4\cdot G_4(0)|_4S\big)(\sqrt{-1} \infty)  =   \frac{\pi^4}{45}
\\
\\
\big(\omega_0^4\cdot G_4(\frac{1}{2}\omega_0)|_4E\big)(\sqrt{-1} \infty)  =   \frac{\pi^4}{3} 
&&
\big(\omega_0^4\cdot G_4(\frac{1}{2}\omega_0)|_4S\big)(\sqrt{-1} \infty)  =   0 \\
\end{array}
$$
}

\noindent
{\bf $\mathrm{G_2}$ type:}  
\vspace{-0.1cm}
{\footnotesize
$$
\begin{array}{ll}
\big(\omega_0 \cdot (\zeta(\frac{1}{3}\omega_0)-\frac{2}{3}\zeta(\frac{1}{2}\omega_0))|_1E\big)(\sqrt{-1} \infty) 
&\!
\big(\omega_0 \cdot (\zeta(\frac{1}{3}\omega_0)-\frac{2}{3}\zeta(\frac{1}{2}\omega_0))|_1S\big)(\sqrt{-1} \infty) \\
 \vspace{-0.2cm}
 =\  \frac{\pi}{\sqrt{3}} & =  \  - \sqrt{-1} \frac{\pi}{3}
\\
\\
\big(\omega_0^2\cdot \mathfrak{p}(\frac{1}{3}\omega_0)|_2E\big)(\sqrt{-1} \infty)  =    \pi^2
&
\big(\omega_0^2\cdot \mathfrak{p}(\frac{1}{3}\omega_0)|_2S\big)(\sqrt{-1} \infty)  =    -\frac{\pi^2}{3}
\\
\\
\big(\omega_0^3\cdot G_3(\frac{1}{3}\omega_0)|_3E\big)(\sqrt{-1} \infty)  =   \frac{2^2\pi^3}{3\sqrt{3}}
&
\big(\omega_0^3\cdot G_3(\frac{1}{3}\omega_0)|_3S\big)(\sqrt{-1} \infty)  =    0\\ 
\end{array}
$$
}
\vspace{-0.15cm}
\centerline{\rm\large  Table 2:\quad {\normalsize Values of Eisenstein series of type $\mathrm{I_2}(p)$ at cusps, I}}

\bigskip
This completes a proof of Theorem \ref{Eisen-Modular}
\end{proof}

\begin{remark}
In case of type $\mathrm{A_2}$, Fourier coefficients of Eisenstein series are well-known to be given by the divisor sum function $\sigma_p(n)$ with suitable constant factors given by special values of Riemann's zeta function (e.g.\ \cite{F-B}VII.1.3). It is natural to ask for similar expressions of the Fourier coefficients of Eisenstein series for the types  $\mathrm{B_2}$ and  $\mathrm{G_2}$.
\end{remark}

Combining above calculations with the expressions \eqref{eq:A2inversion}, \eqref{eq:B2inversion} and \eqref{eq:G2inversion}, we obtain the following table,

\medskip
\noindent
{\bf $\mathrm{A_2}$ type:}  \  
{\footnotesize
$$
\begin{array}{ll}
 \big(\omega_0^4\cdot g_s|_4E \big)(\sqrt{-1} \infty)=\frac{1}{2^2 3} \pi^4,  \\

 \big( \omega_0^6 \cdot g_l |_6E \big)(\sqrt{-1} \infty)=\frac{1}{2^3 3^3} \pi^6,  \\
\end{array}
$$
}

\noindent
{\bf $\mathrm{B_2}$ type:}  \  
{\footnotesize
$$
\begin{array}{ll}
 \big( \omega_0^2\cdot g_s|_2E \big)(\sqrt{-1} \infty)=\pi^2,  &\big(\omega_0^2\cdot  g_s|_2S \big)(\sqrt{-1}\infty)=-\frac{1}{2}\pi^2, \\

 \big( \omega_0^4\cdot g_l|_4E \big)(\sqrt{-1} \infty)=-\frac{1}{2^3}\pi^4, & \big( \omega_0^4\cdot g_l|_4S \big)(\sqrt{-1}\infty)=\frac{1}{2^5}\pi^4, \\
\end{array}
$$
}

\noindent
{\bf $\mathrm{G_2}$ type:}  \  
{\footnotesize
$$
\begin{array}{ll}
\big( \omega_0 \cdot g_s|_1E \big)(\sqrt{-1} \infty)= \frac{\pi}{\sqrt{3}},  & \big(\omega_0 \cdot  g_s|_1S \big)(\sqrt{-1}\infty)= - \frac{\sqrt{-1}}{3}\pi, \\

 \big( \omega_0^3\cdot g_l|_3E \big)(\sqrt{-1} \infty)= \frac{2}{3\sqrt{3}}\pi^3, &\big( \omega_0^3\cdot g_l |_3S \big)(\sqrt{-1}\infty)= \frac{ 2\sqrt{-1}}{3^3}\pi^3. \\
\end{array}
$$
}

\centerline{\rm\large  Table 3:\quad {\normalsize Values of Eisenstein series of type $\mathrm{I_2}(p)$ at cusps, II}}

\smallskip
\noindent
{Step 5.} This is the final step to obtain the isomorphism \eqref{eq:EisensteinModular}. We determine the linear relations between the generators of both hand sides by comparing their values at cusps.

First, we compare the weights (of the free generators of) the rings in both hand sides of \eqref{eq:EisensteinModular} by the use of ``weight factor" $\wt(z)$ (recall \eqref{eq:identification}). Comparing Table 1 in \S2 and the description of \eqref{eq:MGamma}, we obtain the following ``{\it coincidences}" of weights!
\vspace{0.3cm}

\begin{tabular}{r| c c  | cc}
         &  $-\wt(g_s)/\wt(z)$, & $-\wt(g_l)/\wt(z)$ &\text{weights of } $M_*(\Gamma_1([p/2]))$  \\
\hline         
$\mathrm{A_2}$  \! &       -(2/3)/(-1/6)= 4   ,   &      -1/(-1/6)=6       & 4 \ , \quad  6 \\
$\mathrm{B_2}$  \! &         -(1/2)/(-1/4)=2  ,   &   -1/(-1/4)=4        &  2\ , \quad 4  \\
$\mathrm{G_2} $ \! &       -(1/3)/(-1/3)=1  ,  &\!    -1/(-1/3)=3     &   1\ , \quad 3 \\
\end{tabular}

\smallskip
\centerline{\rm\large  Table 4:\quad {\normalsize Weights of the generators of the ring of modular forms}}

 \smallskip
This means that  each generators $g_s$ and $g_l$ are mapped into the graded vector subspace of $M_*(\Gamma_1([p/2]))$ of the same degree as the corresponding generators $\alpha_i$, $\beta_j$ etc., respectively. For $g_s$, the dimension of the graded subspace containing it is equal to 1 so that we only need to fix the constant factor. In case of $g_l$, the dimension is either 1 for type $\mathrm{A_2}$ or 2, spanned by $g_l$ and a power of $g_s$, for types $\mathrm{B_2}$ and $\mathrm{G_2}$. 

Let us recall the values of the  generators of the modular forms at cusps. The following values are taken from \cite{A-I} Internat J. Math. 16-3(2005) 249-279, $B_2$: p.270
$B_3$ pp.271-272.

$$
\begin{array}{cc}
(e_4|_2E)(\sqrt{-1} \infty)=1,  \\
(e_6|_4E)(\sqrt{-1} \infty)=1,  \\
\\
(\alpha_2|_2E)(\sqrt{-1} \infty)=1,  &(\alpha_2|_2S)(\sqrt{-1}\infty)=-1/2, \\
(\beta_4|_4E)(\sqrt{-1} \infty)=0, & (\beta_4|_4S)(\sqrt{-1}\infty)=1/256, \\
\\
(\alpha_1|_1E)(\sqrt{-1} \infty)=1,  & (\alpha_1|_1S)(\sqrt{-1}\infty)=-\sqrt{-1} /\sqrt{3}, \\
(\beta_3|_3E)(\sqrt{-1} \infty)=1, &(\beta_3|_3S)(\sqrt{-1}\infty)=-\sqrt{-1} /3\sqrt{3}. \\
\end{array}
$$
\centerline{\rm\large  Table 5:\quad {\normalsize Values of Modular forms of $\Gamma_1([p/2])$ at cusps}}

\medskip
Comparing Table 3 with Table 5, we obtain the following expressions of the isomorphism \eqref{eq:EisensteinModular}.

\medskip
\noindent
{\bf $\mathrm{A_2}$ type:}  \  
\begin{equation}
\label{eq:A2theta}
\begin{array}{rcl}
g_s  & = &  \  \frac{1}{2^23} \pi^4 \ e_4\ \omega_0^{-4} \quad  \\
 g_l  & = &  \frac{1}{2^33^3} \pi^6 \ e_6 \  \omega_0^{-6}  
\end{array}
\end{equation}

\noindent
{\bf $\mathrm{B_2}$ type:}  \  
\begin{equation}
\label{eq:B2theta}
\begin{array}{rcl}
\qquad \qquad g_s & = & \ \  \pi^2 \ \al_2 \ \omega_0^{-2} 
  \\
g_l & = &   2^4 \pi^4 \ \beta_4 \ \omega_0^{-4}  \ - \  \frac{1}{2^3} \pi^4 \ (\alpha_2\ \omega_0^{-2} )^2
\end{array}
\end{equation}

\noindent
{\bf $\mathrm{G_2}$ type:}  \  
\begin{equation}
\label{eq:G2theta}
\begin{array}{rcl}
g_s & = &  \ \ \  \frac{1}{\sqrt{3}} \pi \ \ \alpha_1 \ \omega_0^{-1}
  \\
g_l & = &  -\frac{2}{3\sqrt{3}} \pi^3 \ \beta_3 \ \omega_0^{-3} 
\end{array}
\end{equation}

\begin{remark}
As a consequence of the identification \eqref{eq:EisensteinModular}, we observe that the holomorphicity condition at cusps on modular forms are equivalent to the holomorphic extendability condition on functions on $S_{\mathrm{I_2}(p)}\setminus D_{\mathrm{I_2}(p)}$ to holomorphic functions on $S_{\mathrm{I_2}(p)}$. Then the cusp form condition ``should be" equivalent to the condition to vanish on the discriminant $D_{\mathrm{I_2}(p)}$. This is the subject discussed in the following lemma. 
\end{remark}

\bigskip
\noindent
\begin{lem}
\label{cusp-disc}
{\it The set of  $\Gamma_1([p/2])$-equivalence classes of cusp points   corresponds naturally in one to one  with the set of irreducible components of the discriminant loci $D_{\mathrm{I_2}(p)}$ \eqref{eq:discriminantloci}.
}
\end{lem}

We show that a generator of the ideal of the ring $M_*(\Gamma_1([p/2]))$ of modular forms vanishing at each equivalence class of cusp points, which can be easily found from Table 4, is, by the pull back by the period map \eqref{eq:periodmap},  up to a constant factor, identified with an irreducible component of the equation $\Delta_{\mathrm{I_2}(p)} \in \mathbb{C}[g_s,g_l]$ \eqref{eq:discriminant} of  the discriminant. 

Here, we may recall Footnote 27 again. 


\medskip
\noindent
{\bf $\mathrm{A_2}$ type:}  \ 
Recall that there is a unique equivalence class of cusps which is represented by $\sqrt{-1}\infty$. The ideal vanishing at the class is generated by $e_4^3-e_6^2$. We also recall the equation of the discriminant \eqref{eq:discriminant} of type $\mathrm{A_2}$.  Then, the identification \eqref{eq:A2theta} induces the following identity:
\begin{equation}
\label{eq:A2DiscriminantModular}
-27g_l^2+g_s^3 \quad = \quad \frac{\pi^{12}}{1728}(e_4^3-e_6^2)\ \omega_0^{-12} 
\end{equation}

\medskip
\noindent
{\bf $\mathrm{B_2}$ type:}  \ 
Recall that there are two equivalence classes of cusps, which are represented by $\sqrt{-1}\infty$ and by 0. The ideal vanishing at the class $\sqrt{-1}\infty$ is generated by $\beta_4$, and 
the ideal vanishing at the class $0$ is generated by $\al_2^2-64 \beta^4$. We  recall the irreducible  factors of the equation of the discriminant \eqref{eq:discriminant} of type $\mathrm{B_2}$ are $-8g_l+g_s^2$ and $8g_l+g_s^2$. 
Then, the identification \eqref{eq:B2theta} induces the following identities.
\begin{equation}
\label{eq:B2DiscriminantModular}
\begin{array}{lcl}
\ \ \ 8g_l+g_s^2 &\ =  \  &  128 \pi^4 \beta_4\ \omega_0^{-4}  \\
-8g_l+g_s^2\ (=-G_2(\frac{1}{2}\omega_0))& \ =  \  &   2\pi^4 (\alpha_2^2 -64\beta_4)\ \omega_0^{-4} 
\end{array}
\end{equation}

\medskip
\noindent
{\bf $\mathrm{G_2}$ type:}  \ 
Recall that there are two equivalence classes of cusps, which are represented by $\sqrt{-1}\infty$ and by 0. The ideal vanishing at the class $\sqrt{-1}\infty$ is generated by $\frac{1}{54}(\alpha_1^3-\beta_3)$, and the ideal vanishing at the class $0$ is generated by $\frac{1}{2}(\alpha_1^3+\beta_3)$.
We recall the irreducible  factors of the equation of the discriminant \eqref{eq:discriminant} of type $\mathrm{G_2}$ are $g_l+2g_s^3 $ and $g_l -2g_s^3$. 
Then,  the identification \eqref{eq:G2theta} induces the following identities.
\begin{equation}
\label{eq:G2DiscriminantModular}
\begin{array}{lcl}
g_l+2g_s^3 & = \ \ \ \frac{2}{3\sqrt{3}} \pi^3 (\alpha_1^3-\beta_3) \ \omega_0^{-3} &   \\
g_l -2g_s^3 \ (=-G_3(\frac{1}{3}\omega_0))& \ = \ -\frac{2}{3\sqrt{3}} \pi^3 (\alpha_1^3+\beta_3) \ \omega_0^{-3} & 
\end{array}
\end{equation}

\bigskip
\section{Discriminant Conjecture}

This is the last section of the study of the period map of types  $\mathrm{A_2}$, $\mathrm{B_2}$ and $\mathrm{G_2}$.
We recover the classical {\it modular discriminant formula} for all types, and answer in  {\bf Theorem} \ref{cuspform} to the discriminant conjecture posed in \cite{S5} \S6.
Let us recall the conjecture in the original form. 

\medskip
\noindent
{\bf Conjecture 6.} {\it Let $W$ be a crystallographic finite reflection group.  Is the $k(W)$th power root of 
$\delta_W$, say $\lambda_W$ (up to a constant factor), an automorphic form for the group $\Gamma(W)$ with the character $\vartheta_W$? {\small Can one find  an infinite product expression for $\lambda_W$ compatible with Conjecture 4?}}

\medskip
We first explain notation in the conjecture. In \cite{S5}, $W$ is a Weyl group for an irreducible finite root system of any type, but in the present paper, we restrict ourselves only to the types $\mathrm{A_2}$, $\mathrm{B_2}$ or $\mathrm{G_2}$. The number $k(W)$ and the character $\vartheta_W$ are defined in \cite{S5} \S6, but, in the present paper, in \eqref{eq:k} and \eqref{eq:character}, respectively. The $\delta_W$ is a  generator of anti-invariants whose square is the discriminant $\Delta_{{\mathrm{I_2}(p)}}$. We shall explain about $\delta_W$ and $\lambda_W$ again in Theorem \ref{cuspform}.
The modular group $\Gamma(W)$ is, in \cite{S5} \S6, 6.4 Example, unfortunately, wrongly stated to be equal to the congruence group $\Gamma_0([p/2])$.  However, according to a result \eqref{eq:congruence1} in the present paper, the group  $\Gamma(W)$ should be corrected to be $\Gamma_1([p/2])$.\footnote{
In fact,  $\Gamma_0(N)=\Gamma_1(N)$ for $N=1,2$, but $\Gamma_0(3)$ is a double extension of  $\Gamma_1(3)$ by $\pm id$. Since in \cite{S5}, one use the same generators as \eqref{eq:monodromy2} in the present paper as for the generators of $\Gamma(I_2(p))$, practically the calculations in \cite{S5} are still meaningful, and we shall use them in the present paper.
}
The Conjecture 4 mentioned in the Conjecture 6 is  something about the ``liftablity" of the ring of Eisenstein series to the ring of  the ``Galois" covering by the Weyl group of the type $\mathrm{I_2}(p)$. However, we will not discuss on this in the present paper.

\medskip
The answer to Conjecture is given by a use of Dedekind eta function $\eta(\tau):=q^{1/24}\prod_{n=1}^\infty (1-q^n)$ ($q=\exp{(2\pi\sqrt{-1}\tau)}$). Recall that a function in $\tau\in\mathbb{H}$ is called an eta-quotient if it has a finite product/quotient expression $\prod_{i=1}^s\eta^{r_i}(m_i\tau)$ where $s, m_i\in\mathbb{Z}_{>0}$, $r_i\in\mathbb{Z}$.
A holomorphic modular form that is non-vanishing on $\mathbb{H}$ and has integer Fourier coefficients at infinity,  is an integer multiple of an eta-quotient (see \cite{R-W}).

Let us come back to the equalities \eqref{eq:A2DiscriminantModular}, \eqref{eq:B2DiscriminantModular} and \eqref{eq:G2DiscriminantModular}, and show that they admit eta-quotient expressions.  LHSs, as  defining equations of irreducible components of the discriminants,  do not vanish on $S_{\mathrm{I_2}(p)}\setminus D_{\mathrm{I_2}(p)}$. So, after the identification \eqref{eq:identification},  they do not vanish on $\widetilde{\mathbb{H}}$. On the other hand, those generators $e_4,e_6$, $\alpha_2,\beta_4$, $\alpha_1,\beta_2$ in RHS, are described by theta-functions, and, therefore, have integral Fourier coefficients at infinity (see \cite{A-I}). Then, we have the following expressions.

\vspace{-0.1cm}
$$
\begin{array}{cl}
\Gamma_1(1): &  \frac{e_4^3-e_6^2}{1728}=\eta(\tau)^{24}, \\
\Gamma_1(2): & \beta_4=\frac{\eta(2\tau)^{16}}{\eta(\tau)^8}, \quad
\alpha_2^2-64\beta_4 =\frac{\eta(\tau)^{16}}{\eta(2\tau)^8}, \\
\Gamma_1(3): & \frac{1}{54}(\alpha_1^3-\beta_3)=\frac{\eta(3\tau)^9}{\eta(\tau)^3}, \quad
\frac{1}{2}(\alpha_1^3+\beta_3)=\frac{\eta(\tau)^9}{\eta(3\tau)^3}
\end{array}
$$
\smallskip
\centerline{\rm\large  Table 5:\quad {\normalsize eta-quotients of irreducible components of the disriminant}}

\smallskip
({\it Proof.} Case of type $\mathrm{A_2}$ is classical (e.g.\  \cite{B-G-H-Z,F-B}). 

Case of type $\mathrm{B_2}$: Because of level 2 condition, candidates of eta quotients are of the form $c\eta(\tau)^p\eta(2\tau)^q$ for unknowns $c\in \mathbb{C}$ and $p,q\in \mathbb{Z}$, satisfying {\small $p+q=2\cdot weight=8$}. Recall Table 5, so that
{\small $\beta_4|4E (\sqrt{-1}\infty)=0$} (simple zero in $q$), {\small $\beta_4|4S (\sqrt{-1}\infty)=1/256$} and
{\small $\alpha_2^2-64\beta_4|4E (\sqrt{-1}\infty)=1$, $\alpha_2^2-64\beta_4|4S (\sqrt{-1}\infty)=0$}  (simple zero in $q$). Posing these constraints on the eta-quotients, we determine $c, p$ and $q$, and obtain the expression.  

Similar proof works for the type $\mathrm{G_2}$. $\Box$)

\medskip
Applying above expressions for  \eqref{eq:A2DiscriminantModular}, \eqref{eq:B2DiscriminantModular} and \eqref{eq:G2DiscriminantModular}, we are now able to express the discriminant form $\Delta_{\mathrm{I_2}(p)}$ \eqref{eq:discriminant} and its reduced  form $\Delta^{red}_{\mathrm{I_2}(p)}$ by some products of eta quotients. As is expected (since discriminant vanishes on both cusps), the results are no-longer {\it eta-quotients}, but are {\it eta-products}, i.e.\ they don't have denominators. 
 
 \begin{equation}
\label{eq:discriminant-eta}
\begin{array}{rccr}
\Delta_{\mathrm{A_2}}\ =&\!\! -27g_l^2+g_s^3 \!\!&=&  \ \ \pi^{12}\ \ \eta(\tau)^{24} \ \omega_0^{-12} \\
\\
\Delta_{\mathrm{B_2}}\ =&\!\! (8g_l+g_{s}^2)(-8g_l+g_{s}^2)^2&=&512\ \pi^{12}\ \eta(\tau)^{24} \ \omega_0^{-12} \\
\\
\Delta_{\mathrm{G_2}}\ =&\!\! 
  (g_l+2g_{s}^3)(-g_l+2g_{s}^3)^3&=&-\frac{27}{4} \ \pi^{12}\ \eta(\tau)^{24} \ \omega_0^{-12} \\
\end{array}
\end{equation}
\centerline{\bf Table 6.  Eta-product expressions of discriminants}

\medskip
 \begin{equation}
\label{eq:reduced-discriminant-eta}
\begin{array}{rcr}
\Delta_{\mathrm{A_2}}^{red}\ =&\!\! -27g_l^2+g_s^3 \ =&\!  \pi^{12}\ \eta(\tau)^{12}\eta(\tau)^{12} \ \omega_0^{-12} \\
\\
\Delta_{\mathrm{B_2}}^{red}\ =&\!\! -64g_l^2+g_s^4 \  =&\! 256 \pi^8\ \eta(\tau)^8 \eta(2\tau)^8 \ \omega_0^{-8} \ \\
\\
\Delta_{\mathrm{G_2}}^{red}\ =&\!\! -g_l^2+4g_s^6  \ \ =&\! -4 \pi^6\ \eta(\tau)^6\eta(3\tau)^6 \ \omega_0^{-6} \ \\
\end{array}
\end{equation}
\centerline{\bf Table 7.  Eta-product expressions of reduced discriminants}

\bigskip
We formulate final Theorem of the present paper, where both formulae \eqref{eq:reduced-discriminant-eta} and \eqref{eq:discriminant-eta} give answers to the first half and the latter half of Conjecture, respectively.

\begin{theorem}
\label{cuspform}
1. {\it For all three types $\mathrm{A_2}$, $\mathrm{B_2}$ and $\mathrm{G_2}$, set 
\begin{equation}
\label{eq:lambda}
\lambda_{\mathrm{I_2}(p)}(\tau) := \eta(\tau)\eta([p/2]\tau) .
\end{equation}
Then, }
(i)  {\it $\lambda_{\mathrm{I_2}(p)}(\tau)$ is a modular form of weight 1 of the group $\Gamma_1([p/2])$ with respect to the character $\vartheta_{\mathrm{I_2}(p)}$  \eqref{eq:character},  }

(ii) {\it The power  $\delta_{\mathrm{I_2}(p)}:= \lambda_{\mathrm{I_2}(p)}(\tau)^{k(\mathrm{I_2}(p))}$ is a generator of the module anti-invariant modular forms (recall \eqref{eq:k} for $k(\mathrm{I_2}(p))$),  }

(iii) {\it The power $\lambda_{\mathrm{I_2}(p)}(\tau)^{2k(\mathrm{I_2}(p))}$, up to a non-zero constant factor, corresponds by \eqref{eq:identification} to the reduced discriminant form $\Delta^{red}_{\mathrm{I_2}(p)}$.}

\smallskip
2. {\it  For all three types $\mathrm{A_2}$, $\mathrm{B_2}$ and $\mathrm{G_2}$, the discriminant form $\Delta_{\mathrm{I_2}(p)}$ \eqref{eq:discriminant}, up to a non-zero constant factor, corresponds by \eqref{eq:identification} to 
$$
\begin{array}{c}
q\prod_{n=1}^\infty(1-q^n)^{24} 
\end{array}
$$

called the modular discriminant (or discriminant function).}
\end{theorem}
\vspace{-0.2cm}
\footnote{
\vspace{-0.1cm}
All results are well-known for the type $\mathrm{A_2}$. However, for the other types  $\mathrm{B_2}$ and $\mathrm{G_2}$, at present, there is no geometric explanation why such power roots of the reduced  discriminant $\Delta_{\mathrm{I_2}(p)}^{red}$ form modular forms with character, and why all three geometric  discriminants  $\Delta_{\mathrm{I_2}(p)}$ take the same modular discriminant form. What are  ${\lambda_{\mathrm{I_2}(p)}}' s$?  Does mirror symmetric interpretation  give a hint to understand them? 
}
\begin{proof}
1.  (i) The calculations given in \cite{S5} to show the modularity of $\lambda_{\mathrm{I_2}(p)}$ with the character $\vartheta$  are still valid. For a sake of completeness of the present paper, we recall it by adjusting notations.

For simplicity, we introduce a number $N\in\mathbb{Z}_{>0}$ called level, where $N=[p/2]$ in case of type $\mathrm{A_2, B_2}$ and $\mathrm{G_2}$. 
Set $\zeta:=\exp{(\pi\sqrt{-1} /12)}$ so that $\zeta^{N+1}=\exp{( \pi \sqrt{-1}/k(\mathrm{I_2}(p)))}$  ($N=1,2,3$ and $p=3,4,6$).  In view of \eqref{eq:monodromy2}, \eqref{eq:congruence1} and \eqref{eq:character},  it is sufficient to show the following.

\begin{lem} {\it The  $\lambda_N:=\eta(\tau)\eta(N\tau)$  for $N\in \mathbb{Z}_{>0}$ is a modular form  of $\Gamma_1(N)$  with a character  $\vartheta_{N}$, where the character  satisfies 
$$\vartheta_{N}: \
\tilde{a}_N
:=
\begin{bmatrix}
1 & 0 \\
-N & 1
\end{bmatrix}, 
 \
 \tilde{b}_N
 :=
 \begin{bmatrix}
 1 & 1\\
 0 & 1
 \end{bmatrix}
\quad
\mapsto
\quad
\zeta^{N+1}\in \mathbb{C}^\times
$$}
\end{lem}
\noindent
{\it Proof of Lemma 10.2.} By definition,  $\lambda_N$ is automatically a modular form with a character of the group $\Gamma_1(N)$ (see, e.g.\  \cite{R-W}). Recalling \eqref{eq:tau} and \eqref{eq:identification}, we have only to show $(\tilde{a}_N^*)^{-1}(\lambda_N\omega_0^{-1})= \zeta^{-N-1}(\lambda_N\omega_0^{-1})$ and $\tilde{b}_N^*(\lambda_N\omega_0^{-1})= \zeta^{N+1}(\lambda_N\omega_0^{-1})$. In the following, we shall use a  sign convention on the monodromy of the eta-function from \cite{Ko} p121.

Recall \eqref{eq:tau},  $\tilde{b}_N^*(\tau)=\tau+1$ and $\tilde{b}_N^*(\omega_0)=\omega_0$. Then, using the transformation formula $\eta(\tau+1)=\zeta \eta(\tau)$, we obtain: 
$$
\begin{array}{rcl}
\tilde{b}_N^*(\lambda_N\omega_0^{-1})& := &\eta(\tau+1)\eta(N(\tau+1))\omega_0^{-1}\\
 &\ =&   \zeta \eta(\tau) \ \zeta^N\eta(N\tau) \omega_0^{-1} \quad =\quad    \zeta^{N+1} (\lambda_N\omega_0^{-1})
 \end{array}
 $$
 Recall \eqref{eq:tau} so that $\tilde{a}_N^*(\tau)=\frac{\tau}{1-N\tau}$ and $\tilde{a}_N^*(\omega_0)=\omega_0-N\omega_1= \omega_0(1-N\tau)$. Then, using the transformation formula $\eta(-/\tau)=\sqrt{\frac{\tau}{\sqrt{-1}}} \eta(\tau)$, we obtain: 
$$
\begin{array}{rcl}
(\tilde{a}_a^*)^{-1}(\lambda_N\omega_0^{-1})& := &\eta(\frac{\tau}{1+N\tau})\eta(\frac{N\tau}{N\tau+1}))\frac{\omega_0^{-1}}{1+N\tau}\\
 &\ =&   \sqrt{\frac{-(N\tau+1)/\tau}{\sqrt{-1}} } \eta(-N-1/\tau) \ \zeta\eta(-\frac{1}{N\tau+1}) \frac{\omega_0^{-1}}{1+N\tau}\\
 &\ =&\sqrt{\frac{-(N\tau+1)/\tau}{\sqrt{-1}}\frac{\tau}{\sqrt{-1}}\frac{N\tau+1}{\sqrt{-1}}} \zeta^{-N}\eta(\tau) \ \zeta\eta(N\tau+1) \frac{\omega_0^{-1}}{1+N\tau} \\
 &\ =& \zeta^{-N+2}\sqrt{\frac{1}{\sqrt{-1}}} \eta(\tau) \eta(N\tau) \omega_0^{-1}\\
&\ =&\zeta^{-N-1}\zeta(\tau)\eta(N\tau)\omega_0^{-1} \quad =\quad    \zeta^{-N-1} (\lambda_N\omega_0^{-1})
 \end{array}
 $$
 End of Proof of Lemma 10.2. $\Box$
 
\medskip
 (ii)  Recalling the character $\theta$ \eqref{eq:anti-invariant}, $\delta_{\mathrm{I_2}(p)}$ is obviously an anti-invariant of the group $\Gamma_1(N)$. 
 
 (iii) This a paraphrase of \eqref{eq:reduced-discriminant-eta}.
 
 \smallskip
 2.  This is only a paraphrase of \eqref{eq:discriminant-eta}.
\end{proof}

\begin{remark}
In \cite{S5} \S6,  we formulated 6 conjectures. The present paper gives positive answers to all conjectures for the types  $\mathrm{A_2, B_2}$ and $\mathrm{G_2}$. 
The  conjectures seem to be  still valid for all types of crystallographic reflection groups. That is,  we ask for a construction of Eisenstein series for all types to answer to Jacobi inversion problem, where we may need special consideration for  low weight cases as in the present paper. 
\end{remark}


\section{Concluding Remarks}


The findings of the  duplication formula for the Lemniscate arc length integral due to Fagnano (1718) and its generalization to the addition formula due to Euler (1751) were naturally understandable by inverting the variables, i.e.\ by parametrizing the Cartesian coordinates of the curve by the arc length. This led to the finding of new periodic functions, i.e.\ the  elliptic functions, beyond the trigonometric or exponential functions. It is impressive to see the historical developments caused by the finding of the elliptic integrals,  from the classical Abel-Jacobi theory through the modern mixed Hodge theory.

I, however, was attracted by other aspects of elliptic integrals. Namely, covariant differentiations of the elliptic integral of the first kind give other  kinds of elliptic integrals.  
 That is, the elliptic integral of the first kind is a {\it potential for all other periods}. I called this property the {\it primitivity} of the elliptic integral of the first kind.

The primitivity is combined with  another remarkable property of the theory. Namely, the 2 of the {\it rank} of the lattice of cycles used for the elliptic integrals is equal to the 2  of the {\it dimension} of the unfolding parameters $(g_s,g_l)$ of the elliptic curve. 
That is, the map from the space of parameters to the space of periods defined by the period integrals of  first kind becomes a morphism between the spaces of the same dimension. The Schottky type problem,  an unsolved problem in classical Abelian integral theory to determine the image set of periods (c.f.\ \cite{Si}), is resolved automatically in this setting! I called this property the {\it equi-dimensionality} of the period map of elliptic integrals.  

Then, it was again natural to ask for a description of the inverse morphism from the space of periods to the original parameter space.\! {\small Actually,} it was done by  {\it theta-series} by Jacobi and later by {\it Eisenstein series}.  We call this procedure  ``solutions to {\it Jacobi's inversion problem}". 

Following those works, I was inspired to look for (higher dimensional) analogs of the elliptic integral of the first kind, which carries the primitivity and the equi-dimensionality. That is the theory of integrals of {\it primitive forms} over Lefschetz vanishing cycles \cite{S1}.  If we look back some historical works from the view point of primitive forms,  the works by E. Picard \cite{Pic1883} (1883) and by G. Shimura \cite{Shi1963,Shi1964} (1963,1964) can be regarded already as some part of period integrals of primitive forms of type $\mathrm{E_6}$ and $\mathrm{E_8}, \mathrm{E_8^{(1,1)}}$ and some others, respectively.  

Nowadays, primitive forms become  a driving force for constructing new integrable hierarchies, and play a  role in  mirror symmetry from complex geometric side. 
However, this is one half of the primitive form theory, i.e.\ algebraic analytic aspects. The transcendental aspects, i.e.\ the period integral theory over Lefschetz vanishing cycles is 
missing still. Primitive forms can play their full original power only after they are integrated to   period maps, and the solutions to the Jacobi inversion problem  leads us to the study of  new transcendental functions \cite{S6}. 

However, the integral theory over closed cycles is a quite hard subject, since they form a closed world which is rigid and inflexible. We first need to embed them in a big ocean of integrals over open cycles, where we have wide freedom of making new pictures and theories, as was done in the original works of arc-length integrals by Fagnano and Euler. Also, the classical abelian integral theory by Riemann was successful through integrals over open intervals, called the Jacobian variety theory, where the inversion maps are described by theta functions \cite{Si}.

Thus, it was my pleasure  to reinterpret the classical elliptic integrals of the families for Weierstrass, Legendre-Jacobi and Hesse in terms of  integrals of primitive forms over vanishing cycles of types $\mathrm{A_2}, \mathrm{B_2}$ and $\mathrm{G_2}$ as in the present paper (types $D_4$ and $A_3$ are ongoing).  We tried to make clear the importance of integrals over open paths  by showing that their inversion functions are the solutions of {\it Hamilton's equation of motion} in \S6. That is, open integrals are inverse to certain  Dynamical systems.
Actually, this fact was the key reason, why we could determine the inversion map in \S7-9 to solve Jacobi's inversion problem by introducing  {\it generalized Eisenstein series} for each type.

We do not know yet what are the higher dimensional analog of them:  how to invert the period map to answer Jacobi's inversion problem,  how to generate the inversion functions (\cite{S4}). When the vanishing cycles form a finite root system, there are some conjectural descriptions  (\cite{S5}). 
{\small The positive answer to the conjectures for types $\mathrm{A_2}, \mathrm{B_2}$ and $\mathrm{G_2}$ in the present} paper by a use of the Dedekind {\it eta-function} is a toy model and did not give any new transcendental function. However, we may get enough hope to expect that the conjectures still hold for all other types of root systems.

%
 One may expect the mirror symmetry and the study of d-branes on open string theory in high energy physics may give some suggestions for the understanding of the period maps for primitive forms, since the study of the power roots of the discriminant, as done in the present paper,  is explained not from the primitive form side by itself but from the structure of the root system of the  vanishing cycles \cite{S5} which belongs to the mirror side, i.e.\ the symplectic geometry side.

\medskip
\noindent
{\bf Acknowledgements.}
The author expresses his gratitudes to Yoshihisa Obayashi, who draw Figures 1- 5 of the present paper. Particular thanks go to Hiroki Aoki, without whose helps the author may have not accomplish the identification of the ring of Eisenstein series of type $\mathrm{I_2}(p)$ with the ring of modular forms of the congruence group $\Gamma_1([p/2])$.   He expresses also gratitudes to Yoshihisa  Saito, Kenji Iohara, Takashi Takebe, Tomoyoshi Ibukiyama, Hiroyuki Yoshida, Masanobu Kaneko and Akio Fujii for helpful and inspiring discussions, and to Yota Shamoto for careful reading of manuscripts. This work was partially supported by JSPS KAKENHI Grant Number 18H01116.

\bibliographystyle{plain} 
\bibliography{saito}

\begin{thebibliography}{10}

\bibitem{A-I}
Hiroki Aoki and Tomoyoshi Ibukiyama.
\newblock Simple graded rings of {S}iegel modular forms, differential operators
  and {B}orcherds products.
\newblock {\em Internat. J. Math.}, 16(3):249--279, 2005.

\bibitem{A-S}
Hiroki Aoki and Kyoji Saito.
\newblock {M}odular forms from {W}eierstrass $\mathfrak{p}$-function and
  zeta-function.
\newblock forthcoming.

\bibitem{A-K-I}
Tsuneo Arakawa, Masanobu Kaneko, and Tomoyoshi Ibukiyama.
\newblock {\em Bernoulli numbers and Zeta functions}.
\newblock Springer Verlag, Tokyo, 2014.

\bibitem{Bourbaki}
Nicolas Bourbaki.
\newblock {\em \'{E}l\'{e}ments de math\'{e}matique}.
\newblock Masson, Paris, 1981.
\newblock Groupes et alg\`ebres de Lie. Chapitres 4, 5 et 6. [Lie groups and
  Lie algebras. Chapters 4, 5 and 6].

\bibitem{Br}
Egbert Brieskorn.
\newblock Die {F}undamentalgruppe des {R}aumes der regul\"{a}ren {O}rbits einer
  endlichen komplexen {S}piegelungsgruppe.
\newblock {\em Invent. Math.}, 12:57--61, 1971.

\bibitem{B-S}
Egbert Brieskorn and Kyoji Saito.
\newblock Artin-{G}ruppen und {C}oxeter-{G}ruppen.
\newblock {\em Invent. Math.}, 17:245--271, 1972.

\bibitem{B-G-H-Z}
Jan~Hendrik Bruinier, Gerard van~der Geer, G\"{u}nter Harder, and Don Zagier.
\newblock {\em The 1-2-3 of modular forms}.
\newblock Universitext. Springer-Verlag, Berlin, 2008.
\newblock Lectures from the Summer School on Modular Forms and their
  Applications held in Nordfjordeid, June 2004, Edited by Kristian Ranestad.

\bibitem{D}
Pierre Deligne.
\newblock Les immeubles des groupes de tresses g\'{e}n\'{e}ralis\'{e}s.
\newblock {\em Invent. Math.}, 17:273--302, 1972.

\bibitem{D-S}
Fred Diamond and Jerry Shurman.
\newblock {\em A first course in modular forms}, volume 228 of {\em Graduate
  Texts in Mathematics}.
\newblock Springer-Verlag, New York, 2005.

\bibitem{E-Z}
Martin Eichler and Don Zagier.
\newblock {\em The theory of {J}acobi forms}, volume~55 of {\em Progress in
  Mathematics}.
\newblock Birkh\"{a}user Boston, Inc., Boston, MA, 1985.

\bibitem{F-B}
Eberhard Freitag and Rolf Busam.
\newblock {\em Complex analysis}.
\newblock Universitext. Springer-Verlag, Berlin, second edition, 2009.

\bibitem{G}
F.~A. Garside.
\newblock The braid group and other groups.
\newblock {\em Quart. J. Math. Oxford Ser. (2)}, 20:235--254, 1969.

\bibitem{H-C}
Adolf Hurwitz and R.~Courant.
\newblock {\em Vorlesungen \"{u}ber allgemeine {F}unktionentheorie und
  elliptische {F}unktionen}.
\newblock Interscience Publishers, Inc., New York, 1944.

\bibitem{K-S-T}
Hiroshige Kajiura, Kyoji Saito, and Atsushi Takahashi.
\newblock Matrix factorization and representations of quivers. {II}. {T}ype
  {$ADE$} case.
\newblock {\em Adv. Math.}, 211(1):327--362, 2007.

\bibitem{Ko}
Neal Koblitz.
\newblock {\em Introduction to elliptic curves and modular forms}, volume~97 of
  {\em Graduate Texts in Mathematics}.
\newblock Springer-Verlag, New York, second edition, 1993.

\bibitem{Lang}
Serge Lang.
\newblock {\em Introduction to modular forms}.
\newblock Springer-Verlag, Berlin-New York, 1976.
\newblock Grundlehren der mathematischen Wissenschaften, No. 222.

\bibitem{Pic1883}
E.~{Picard}.
\newblock {Sur des fonctions de deux variables ind\'ependantes analogues aux
  fonctions modulaires.}
\newblock {\em {Acta Math.}}, 2:114--135, 1883.

\bibitem{R-W}
Jeremy Rouse and John~J. Webb.
\newblock On spaces of modular forms spanned by eta-quotients.
\newblock {\em Adv. Math.}, 272:200--224, 2015.

\bibitem{S1}
Kyoji Saito.
\newblock Period mapping associated to a primitive form.
\newblock {\em Publ. Res. Inst. Math. Sci.}, 19(3):1231--1264, 1983.

\bibitem{S0}
Kyoji Saito.
\newblock On a linear structure of the quotient variety by a finite reflexion
  group.
\newblock {\em Publ. Res. Inst. Math. Sci.}, 29(4):535--579, 1993.

\bibitem{S4}
Kyoji Saito.
\newblock Primitive automorphic forms.
\newblock In {\em Mathematics Unlimited ― 2001 and Beyond}, pages 1003--1018.
  Springer, 2001.

\bibitem{S3}
Kyoji Saito.
\newblock Polyhedra dual to the {W}eyl chamber decomposition: a pr\'{e}cis.
\newblock {\em Publ. Res. Inst. Math. Sci.}, 40(4):1337--1384, 2004.

\bibitem{S5}
Kyoji Saito.
\newblock Uniformization of the orbifold of a finite reflection group.
\newblock In {\em Frobenius manifolds}, Aspects Math., E36, pages 265--320.
  Friedr. Vieweg, Wiesbaden, 2004.

\bibitem{S6}
Kyoji Saito.
\newblock Jugendtraum of a mathematician.
\newblock {\em Asia Pac. Math. Newsl.}, 1(3):1--6, 2011.

\bibitem{Shi1963}
Goro Shimura.
\newblock On analytic families of polarized abelian varieties and automorphic
  functions.
\newblock {\em Ann. of Math. (2)}, 78:149--192, 1963.

\bibitem{Shi1964}
Goro Shimura.
\newblock On purely transcendental fields automorphic functions of several
  variable.
\newblock {\em Osaka Math. J.}, 1(1):1--14, 1964.

\bibitem{Si}
Carl~Ludwig Siegel.
\newblock {\em Topics in complex function theory. {V}ol. {I}}.
\newblock Wiley Classics Library. John Wiley \& Sons, Inc., New York, 1988.
\newblock Elliptic functions and uniformization theory, Translated from the
  German by A. Shenitzer and D. Solitar, With a preface by Wilhelm Magnus,
  Reprint of the 1969 edition, A Wiley-Interscience Publication.

\bibitem{Z}
Oscar Zariski.
\newblock {\em Collected papers. {V}ol. {III}}.
\newblock The MIT Press, Cambridge, Mass.-London, 1978.
\newblock Topology of curves and surfaces, and special topics in the theory of
  algebraic varieties, Edited and with an introduction by M. Artin and B.
  Mazur, Mathematicians of Our Time.

\end{thebibliography}

\end{document}